\documentclass{amsart}

\usepackage[utf8]{inputenc}
\usepackage[T1]{fontenc}
\usepackage{hyperref}
\usepackage{amsmath}
\usepackage{amssymb}
\usepackage{cleveref}
\usepackage{mathtools}  
\usepackage{amsthm}
\usepackage{microtype}
\usepackage[english]{babel}
\usepackage{csquotes} 
\usepackage{graphicx}
\usepackage{subcaption}
\usepackage{lmodern}
\usepackage{clrscode3e}  
\usepackage{booktabs}
\usepackage{nicefrac}  %
\usepackage{enumitem}
\usepackage[textsize=tiny]{todonotes}
\usepackage[font=small,labelfont=bf]{caption} 
\usepackage[section]{placeins} 
\usepackage{etoolbox}
\usepackage{marginnote}
\usepackage{float}

\newcommand*{\mailto}[1]{\href{mailto:#1}{\nolinkurl{#1}}}

\newcommand{\Dx}{{\Delta x}}
\newcommand{\R}{\mathbb{R}}

\newcommand\myeq{\stackrel{\mathclap{\normalfont\mbox{\tiny Tonelli}}}{=}}
\newcommand\myineq{\stackrel{\mathclap{\normalfont\mbox{\tiny CS}}}{\leq}}

\newlist{thmlist}{enumerate}{1}
\setlist[thmlist]{label=(\roman{thmlisti}), 
, noitemsep}

\newcommand{\arxiv}[1]{\href{http://arxiv.org/pdf/#1}{arXiv:#1}}

\newcommand{\D}{\mathcal{D}}
\newcommand{\F}{\mathcal{F}}
\newcommand{\M}{\mathcal{M}}
\renewcommand{\id}{\text{\normalfont{id}}}

\renewcommand{\d}[1]{\ensuremath{\operatorname{d}\!{#1}}}
\apptocmd{\lim}{\limits}{}{}
\numberwithin{equation}{section}
{
    \theoremstyle{plain}
    \newtheorem{definition}{Definition}[section]
    \newtheorem{remark}[definition]{Remark}
    \newtheorem{prop}[definition]{Proposition}
    \newtheorem{example}[definition]{Example}

    \newtheorem{theorem}[definition]{Theorem}
    
    \newtheorem{lemma}[definition]{Lemma}
    
}

\title[Convergence rate for the HS equation]{On the convergence rate of a numerical method for the Hunter--Saxton equation}

\author[T. Christiansen]{Thomas Christiansen}
\address{Department of Mathematical Sciences\\ NTNU Norwegian University of Science and Technology\\ NO-7491 Trondheim\\ Norway}
\email{\mailto{thomas.christiansen@ntnu.no}}
\urladdr{\url{https://www.ntnu.edu/employees/thomachr}}


\begin{document}

\counterwithin{equation}{section}
\raggedbottom
\allowdisplaybreaks

\begin{abstract}\label{sec:Abstract}
	We derive a robust error estimate for a recently proposed numerical method for $\alpha$-dissipative solutions of the Hunter--Saxton equation, where $\alpha \in [0, 1]$. In particular, if the following two conditions hold: 
	\emph{i)} there exist a constant $C > 0$ and $\beta \in (0, 1]$ such that the initial spatial derivative $\bar{u}_{x}$ satisfies $\|\bar{u}_x(\cdot + h) - \bar{u}_x(\cdot)\|_2 \leq Ch^{\beta}$ for all $h \in (0, 2]$, and \emph{ii)}, the singular continuous part of the initial energy measure is zero, then the numerical wave profile converges with order $\mathcal{O}(\Dx^{\nicefrac{\beta}{8}})$ in $L^{\infty}(\R)$.  

Moreover, if $\alpha=0$, then the rate improves to $\mathcal{O}(\Dx^{\nicefrac{1}{4}})$ without the above assumptions, and we also obtain a convergence rate for the associated energy measure -- it converges with order $\mathcal{O}(\Dx^{\nicefrac{1}{2}})$ in the bounded Lipschitz metric.

	These convergence rates are illustrated by several examples. 
\end{abstract}

\maketitle

 \vspace{-0.5cm}
\section{Introduction}
The Cauchy problem for the Hunter--Saxton (HS) equation reads
 \begin{align}\label{eq:HS}
 	u_t(t, x) + u u_x(t, x) &= \frac{1}{4}\int_{-\infty}^x u_x^2(t, z)dz - \frac{1}{4}\int_{x}^{\infty}u_x^2(t, z)dz, \quad u|_{t=0} = \bar{u}, 
\end{align}
where $(t, x) \in [0, \infty)\times \R$. This equation was proposed in \cite{DynamicsDirector} as an asymptotic model for nonlinear instabilities in the director field of nematic liquid crystals. 

The left-hand side of equation \eqref{eq:HS} contains a Burgers nonlinearity, while the right-hand side is a nonlocal source term. In contrast to what one would typically expect for the inviscid Burgers equation, where shock formations are to be anticipated, see e.g., \cite{FrontTracking}, the source term has a regularizing effect on the solution $u$, preventing shocks from occurring, see \cite{DafermosContinuous, AlphaHS}. As a consequence, $u(t, \cdot)$ remains H{\"o}lder continuous, and, in addition, $u_x(t, \cdot) \in L^2(\R)$ for all $t\geq 0$. 

However, classical solutions to \eqref{eq:HS} will in general cease to exist past finite time since the spatial derivative, $u_x$, blows up pointwise -- a phenomenon referred to as {\em wave breaking}. If wave breaking takes place at $t=t_c$, then there exists at least one point $x \in \R$ such that $u_x(t, x) \rightarrow -\infty$ as $t \uparrow t_c$. This leads to energy concentrations on sets of zero measure as the following example illustrates. 
\begin{example}\label{ex:simpleBreaking}
	Consider the following weak solution of \eqref{eq:HS}, valid for $ 0 \leq t < 2$, 
\begin{align*}
	u(t, x) &= \begin{cases}
	 - \frac{1}{8}t + \frac{1}{2}, & x <\frac{1}{16}(8- t)t, \\
	\frac{8x - (t+4)}{4(t-2)}, & \frac{1}{16}(8 - t)t \leq x \leq \frac{1}{16}(8+t^2), \\ 
	\frac{1}{8}t, & \frac{1}{16}(8+t^2) < x,
	\end{cases} 
\end{align*}
and assume that the associated cumulative energy is described by 
\begin{equation}\label{eq:simpleEnergy}
	F(t, x) = \int_{-\infty}^x u_x^2(t, z)dz.
\end{equation}
The pair $(u, F)$ is visualized in Figure~\ref{fig:ex1} for $t=0$, $1$, and $1.95$. As $t \uparrow 2$ we observe that the interval $[\frac{1}{16}(8 - t)t, \frac{1}{16}(8+t^2)]$ shrinks to the single point $x=\frac{3}{4}$, while $u_x(t, \frac{3}{4}) \rightarrow -\infty$. Moreover, $u_{x}^2(t)dx$ formally tends to the Dirac measure $\frac{1}{2}d\delta_{\frac{3}{4}}$ as $t \uparrow 2$, and hence \eqref{eq:simpleEnergy} is inadequate for describing the energy at $t=2$. In other words, the energy cannot be described by an absolutely continuous function for all $t\geq 0$.

\begin{figure}
	\includegraphics{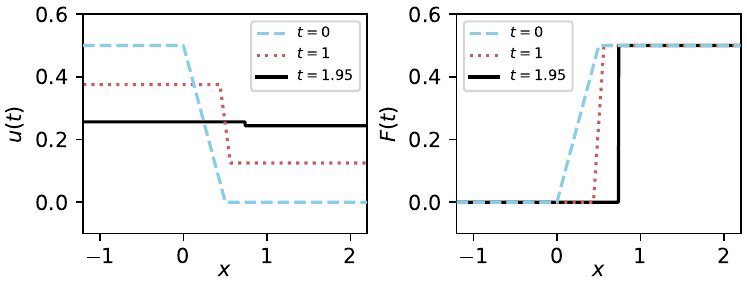}
	 \captionsetup{width=.95\linewidth}
	\caption{The solution $u$ (left) and the associated energy $F$ (right) from Example~\ref{ex:simpleBreaking} at the times $t=0$, $1$, and $1.95$.}
	\label{fig:ex1}
\end{figure}

\end{example}

To properly describe the solution in Example \ref{ex:simpleBreaking} and solutions to \eqref{eq:HS} in general,  one has to augment the wave profile $u$ with a finite, nonnegative Radon measure $\mu$ that represents the energy density. This measure satisfies  $d\mu_{\mathrm{ac}} = u_x^2 dx$ and encodes all the information about wave breaking. In particular, the singular part $\mu_{\mathrm{sing}}$ tells us where energy has concentrated, as was the case in Example~\ref{ex:simpleBreaking}, but the absolutely continuous part $\mu_{\mathrm{ac}}$ can also give rise to wave breaking which manifests itself in the form of infinitesimal energy concentrations, see e.g., \cite[Sec. 5]{AlphaAlgorithm} and \cite[Sec. 3]{NumericalConservative}. 

In order to extend local solutions of \eqref{eq:HS} into global ones, one has to make a choice about how to manipulate the concentrated energy at wave breaking.  There is a plethora of possible energy modifications, leading to a whole continuum of weak solutions. This has been a subject of extensive research, see for instance \cite{BressanGlobalSolutions, LipschitzConservative, ConsMetric, AlphaHS,  TandyLipschitz, DynamicsDirector, NonlinearVariational}. In this work however, we are concerned with the particular notion of $\alpha$-dissipative solutions, for fixed $\alpha \in [0, 1]$. As the name suggests, such solutions are obtained by removing a fixed $\alpha$-fraction of the concentrated energy at every wave breaking occurrence. Existence of weak $\alpha$-dissipative solutions was established in \cite{AlphaHS} via a generalized method of characteristics in the more general setting of $\alpha \in W^{1, \infty}(\R, [0, 1)) \cup \{1 \}$. Uniqueness on the other hand, has only been proved in the particular cases of {\em conservative} ($\alpha = 0$) and {\em dissipative} ($\alpha=1$) solutions, see \cite{UniquenessConservative} and \cite{DafermosCharacteristics}, respectively. 

The approach in \cite{AlphaHS} is based on introducing a nonlinear mapping $L$, which transforms the Eulerian initial data $(\bar{u}, \bar{\mu})$ into an initial triplet $(\bar{y}, \bar{U}, \bar{V})$  in Lagrangian coordinates, whose time evolution is governed by 
\begin{align}
\begin{aligned}\label{eq:ODEIntro}
	y_t(t, \xi) &= U(t, \xi), \\ 
	U_t(t, \xi) &= \frac{1}{2}V(t, \xi) - \frac{1}{4}V_{\infty}(t), \\
	V(t, \xi) & = \int_{-\infty}^{\xi}\left(1 - \alpha \chi_{\{\omega: t \geq \tau(\omega)>0\}}(\eta)\right)\!\bar{V}_{\xi}(\eta)d\eta.
\end{aligned}
\end{align}
Here {\small$\displaystyle V_{\infty}(t)= \lim_{\xi \rightarrow \infty} V(t, \xi)$} is the total cumulative Lagrangian energy and $\tau\!: \!\R \!\rightarrow \! [0, \infty]$ is the {\em wave breaking function} defined by 
\begin{align*}
	\tau(\xi) &= \begin{cases}
	0, & \bar{U}_{\xi}(\xi) = \bar{y}_{\xi}(\xi) = 0, \\
	-2\frac{\bar{y}_{\xi}(\xi)}{\bar{U}_{\xi}(\xi)}, & \bar{U}_{\xi}(\xi) < 0, \\
	\infty, & \text{otherwise}.
	\end{cases}
\end{align*}
It is natural to interpret $\xi$ as a label representing a particle, in which case $y(\cdot, \xi)$ represents the particle trajectory, $U(\cdot, \xi)$ its velocity, and $V_{\xi}(\cdot, \xi)$ its associated energy. In addition, $\tau(\xi)$ is the collision time of the particle, such that $\tau(\xi) = \infty$ indicates that the particle does not engage in any collisions. 

Naturally, several different types of numerical methods have been proposed for the HS equation, e.g., finite difference schemes of upwind-type \cite{FDDissipative}, discontinuous Galerkin methods \cite{DGMethod1, DGMethod2}, geometrical methods \cite{GeometricInt}, and schemes based on the method of characteristics and its generalizations \cite{Alpha2, AlphaAlgorithm, NumericalConservative}. Yet there are very few results regarding convergence rates. The author is only aware of two results in this direction: Let $\Dx$ denote a spatial discretization parameter for a uniform mesh on $\R$, then the first result is the $\mathcal{O}(\Dx^{\nicefrac{1}{2}})$-rate shown in  \cite{NumericalConservative} for conservative solutions under the assumption that $u \in W^{1, \infty}([0, T] \times \R)$. This regularity does however break down at wave breaking. The second result, shown in \cite{VanishingViscosity1} and \cite{RateVanishingViscosity}, is concerned with the {\em vanishing viscosity method} which converges with order $\mathcal{O}(\epsilon^{\nicefrac{1}{2}})$ towards dissipative solutions, where $\epsilon > 0$ is the viscosity parameter. Again, this result is valid either prior to wave breaking, or under the assumption that $\bar{u}_x \geq 0$, but this prevents wave breaking from taking place for any $t>0$.

We contribute to this line of research by deriving a {\em robust error estimate} (valid also after wave breaking) for the numerical scheme proposed in \cite{AlphaAlgorithm} for $\alpha$-dissipative solutions. This method takes advantage of the fact that if the initial data $\bar{u}$ in \eqref{eq:HS} is piecewise linear, then the corresponding solution, $u(t, \cdot)$, remains piecewise linear for all $t\geq 0$, see \cite{InverseScattering, PropertiesHS}. Motivated by this observation, the authors in \cite{AlphaAlgorithm} derive a piecewise linear projection operator, $P_{\Dx}$, which preserves the relation $d\mu_{\mathrm{ac}} = u_x^2dx$. This is essential in order to ensure that the numerical solutions $\{(u_{\Dx}, \mu_{\Dx})(t)\}_{\Dx > 0}$, which are obtained by solving \eqref{eq:HS} exactly with initial data $\{P_{\Dx}((\bar{u}, \bar{\mu}))\}_{\Dx > 0}$ via the outlined generalized method of characteristics, are in fact $\alpha$-dissipative solutions, i.e., they dissipate the correct amount of energy. It is then established that 
\begin{equation*}
	u_{\Dx} \rightarrow u \text{ in} C \big([0, \infty), L^{\infty}(\R) \big)  \hspace{0.3 cm} \text{ as } \Dx \rightarrow 0, 
\end{equation*}
and we aim to strengthen this result, by proving that $\|u(t) - u_{\Dx}(t)\|_{\infty} \!\leq \mathcal{O}(\Dx^{\gamma})$ for some $\gamma > 0$. With this in mind, let $\{(y_{\Dx}, U_{\Dx}, V_{\Dx})(t)\}_{\Dx > 0}$ denote the family of numerical solutions in Lagrangian coordinates. Thanks to \cite[Lem. 4.11]{AlphaAlgorithm}, we have, for any $t\geq 0$, 
\begin{equation*}
	\|u(t) - u_{\Dx}(t)\|_{\infty} \! \leq \!\|U(t) - U_{\Dx}(t)\|_{\infty} + \sqrt{\bar{\mu}(\R)}\|y(t) - y_{\Dx}(t)\|_{\infty}^{\nicefrac{1}{2}}. 
\end{equation*}
It therefore enough to derive convergence rates for $\|y(t)-y_{\Dx}(t)\|_{\infty}$ and \newline $\|U(t)-U_{\Dx}(t)\|_{\infty} = \|u(y(t)) - u_{\Dx}(y_{\Dx}(t))\|_{\infty}$. However, this is a delicate issue, because the relation between the Eulerian and Lagrangian variables is highly nonlinear. To be more precise, the initial energy density $\bar{\mu}$ may have a singular part, $\bar{\mu}_{\mathrm{sing}}$, which generates a set $\mathcal{A}$ of positive measure when mapping into Lagrangian coordinates. Since there is no natural way to relate the differentiated Eulerian and Lagrangian variables on $\mathcal{A}$, we only have a convergence rate for $\bar{V}_{\xi} - \bar{V}_{\Dx, \xi}$ over $\R \! \!\setminus \! \mathcal{A}$. This is problematic, because $\|U(t) - U_{\Dx}(t)\|_{\infty}$ and $\|y(t) - y_{\Dx}(t)\|_{\infty}$ heavily depend on $V_{\xi}(t) - V_{\Dx, \xi}(t)$ on all of $\R$, cf. \eqref{eq:ODEIntro}. 

To overcome this issue, we introduce a transformation which enables us to prove that the set $\mathcal{A}$ contributes at most with order $\mathcal{O}(\Dx^{\beta})$ to the aforementioned norms,  provided the following two conditions hold: \emph{i)} there exist some constant $C > 0$ and $\beta \in (0, 1]$ such that 
\begin{equation}\label{eq:extraCondIntro}
	\|\bar{u}_x(\cdot + h) - \bar{u}_x(\cdot)\|_2 \leq Ch^{\beta} \hspace{0.4cm} \text{for all } h \in (0, 2],
\end{equation}
and \emph{ii)}, the singular continuous part of $\bar{\mu}$ is zero. The latter condition is needed to avoid issues with measurability. The remaining contribution to $\|U(t) - U_{\Dx}(t)\|_{\infty}$ and $\|y(t) - y_{\Dx}(t)\|_{\infty}$ can be bounded in terms of the Lagrangian initial data over $\R \setminus \mathcal{A}$. One is therefore just left to relate the initial convergence rate in Eulerian coordinates to a rate of convergence for the Lagrangian variables on $\R \setminus \mathcal{A}$, from which one finally infers that $\|u(t) - u_{\Dx}(t)\|\leq \mathcal{O}(\Dx^{\nicefrac{\beta}{8}})$ for all $t \geq 0$. 

Before we illustrate the convergence rate on some important examples, we take a detour and investigate the particular case of conservative solutions. In this case we find that the rate can be improved significantly without the supplemented conditions on the initial data. Additionally, one obtains a convergence rate for the family $\{\mu_{\Dx}(t)\}_{\Dx > 0}$ of numerical energy densities in the bounded Lipschitz metric.

This paper is organized as follows. In Section~\ref{sec:Preliminaries} we remind the reader of the generalized method of characteristics based on \cite{LipschitzConservative} and \cite{AlphaHS}. In addition, we recall the numerical method from  \cite{AlphaAlgorithm}. Thereafter, in Section~\ref{sec:Convergenceux}, we study how \eqref{eq:extraCondIntro} enables us to deduce a convergence rate initially in $L^2(\R)$ for the numerical derivative $\bar{u}_{\Dx, x}$. This is essential, when we derive a convergence rate for $\|u(t)-u_{\Dx}(t)\|_{\infty}$ in Section~\ref{sec:ConvergenceRateu}.  Section~\ref{sec:convConservative} is devoted to the particular case of conservative solutions and we end by studying the convergence rates for some explicit examples in Section~\ref{sec:ExplicitRates}. This culminates in an $L^{\infty}(\R)$-error estimate of order $\mathcal{O}(\Dx^{\nicefrac{1}{48}})$ for the cusped initial data considered in \cite[Sec. 5]{AlphaAlgorithm}. 

\section{Preliminaries and the numerical method}\label{sec:Preliminaries}
In this section, we present the numerical method introduced in \cite{AlphaAlgorithm}. Since it relies on the construction of $\alpha$-dissipative solutions via the generalized method of characteristics from \cite{LipschitzConservative} and \cite{AlphaHS}, we start by recalling the main steps of this construction. 

Let $\alpha \in [0,1]$ be fixed and denote by $\M^+(\R)$, the space of nonnegative, finite Radon measures on $\R$. In order to define the set of admissible data to \eqref{eq:HS} we have to recall some function spaces from \cite{LipschitzConservative} and \cite{AlphaHS}. To this end, let $E$ be the Banach space \vspace{-0.05cm}
\begin{equation*}
	E := \{f \in L^{\infty}(\R): f' \in L^2(\R)\}, 
\end{equation*}
equipped with the norm $\|f\|_E := \|f\|_{\infty} + \|f'\|_2$. Furthermore, let 
\begin{equation*}
	H_d^1(\R) = H^1(\R) \times \R^d, \hspace{0.3cm} d = 1, 2, 
\end{equation*}
and decompose $\R$ into the overlapping union $\R\!=(-\infty, 1) \cup (-1, \infty)$. We pick a partition of unity, i.e., a pair $\chi^{+}, \chi^{-}$ belonging to $C^{\infty}(\R)$, satisfying $\chi^+ + \chi^- = 1$, $0 \leq \chi^{\pm} \leq 1$, in addition to $\mathrm{supp}(\chi^+) \subset (-1, \infty)$ and $\mathrm{supp}(\chi^-) \subset (-\infty, 1)$. This pair allows us to define the following linear, continuous, and injective maps 
\begin{align*}
	R_1: H_1^1(\R) &\rightarrow E, \hspace{1 cm}  (\bar{f}, a) \mapsto f = \bar{f} + a \chi^+, \\
	R_2: H_2^1(\R) &\rightarrow E, \hspace{1 cm}  (\bar{f}, a, b) \mapsto f = \bar{f} + a \chi^+ + b\chi^{-}.
\end{align*}
Accordingly, we define the Banach spaces $E_1$ and $E_2$ as \vspace{-0.05cm}
\begin{equation*}
	E_1 := R_1(H_1^1(\R)) \hspace{0.5 cm} \text{and} \hspace{ 0.5cm} E_2 := R_2(H^1_2(\R)),
\end{equation*}
endowed with the norms \vspace{-0.1cm}
\begin{align*}
	\|f\|_{E_1} &:= \left(\|f\|_{H^1(\R)}^2 + a^2\right)^{\nicefrac{1}{2}} \hspace{0.5cm} \text{and} \hspace{0.5cm} \|f\|_{E_2} := \left(\|f\|_{H^1(\R)}^2 + a^2 + b^2 \right)^{\nicefrac{1}{2}}\!\!\!.
\end{align*}
At this point, it is worth emphasizing that the spaces $E_1$ and $E_2$ do not depend on the particular choice of $\chi^+$ and $\chi^{-}$\!, see \cite{NonVanishingAsymptotes}. Furthermore, $R_1$ is also well-defined when applied to functions in $L_1^2(\R)\!:= L^{2}(\R) \times \R$, motivating the following space 
\begin{equation*}
	E_1^0:= R_1(L_1^2(\R)) \hspace{0.3cm} \text{with} \hspace{0.4cm} \|f\|_{E_1^0}:=\left(\|\bar{f}\|_2^2 + a^2\right)^{\nicefrac{1}{2}}\!\!\!.
\end{equation*}
At last, we can define the set of Eulerian coordinates $\D$, which contains the admissible initial data, as well as the numerical solutions (to be defined).    
\begin{definition}\label{def:EulerSet}
    The space $\D$ consists of all pairs $(u, \mu)$ satisfying
    \begin{enumerate}[label=(\roman*)]
        \item $u \in E_2$, 
        \item $\mu \in \M^{+}(\R)$,
         \item $\mu\left( (-\infty, \cdot) \right)\! \in E_1^0$,
        \item $d\mu_{\mathrm{ac}} = u_x^2dx$.\label{def:EulerCond3}
    \end{enumerate}
\end{definition}

It follows by \cite[Thm. 1.16]{RealAnalysisFolland} that there is a one-to-one relationship between $\mu\in \M^+(\R)$ and the cumulative energy $F(x) = \mu((-\infty, x))$, which justifies our upcoming, interchangeable use of both $(u, \mu)$ and $(u, F)$ to refer to the same element of $\D$. Moreover, recall that $F$ is an increasing, left continuous, and bounded function satisfying 
\begin{equation*}
	\lim_{x \rightarrow -\infty} F(x) = 0 \hspace{0.35cm} \text{and} \hspace{0.35cm} \lim_{x \rightarrow \infty} F(x)  = \mu(\R),
\end{equation*}
and that it can be decomposed into two parts,  \vspace{-0.1cm}
\begin{equation}\label{eq:splitF}
	F(x) = F_{\mathrm{ac}}(x) + F_{\mathrm{sing}}(x) := \mu_{\mathrm{ac}}((-\infty, x)) +\mu_{\mathrm{sing}}((-\infty, x)),
\end{equation}
where $\mu_{\mathrm{ac}}$ and $\mu_{\mathrm{sing}}$ are absolutely continuous and singular, respectively. 

Associated with $(\bar{u}, \bar{\mu}) \in \D$, we define the Lagrangian coordinates
\begin{subequations}\label{eq:MapL}
        \begin{align}
            \bar{y}(\xi) &= \sup \{x \in \R: x + \bar{\mu}((-\infty, x)) < \xi \} \label{eq:MapL1}, \\
            \bar{U}(\xi) &= \bar{u}(\bar{y}(\xi))\label{eq:MapL2}, \\ 
            \bar{V}(\xi) &= \xi - \bar{y}(\xi),  \label{eq:MapL3}  
        \end{align}
\end{subequations}
and we let $L$ denote the mapping which is given by $(\bar{y}, \bar{U}, \bar{V}) = L((\bar{u}, \bar{\mu}))$. Furthermore, the time evolution in Lagrangian coordinates is governed by, see \cite[Sec. 2]{AlphaHS}, 
\begin{subequations}\label{eq:LagrSystem}
\begin{align}
	y_t(t, \xi) &= U(t, \xi) \label{eq:LagrSystem1}, \\
	U_t(t, \xi) &= \frac{1}{2}V(t, \xi) - \frac{1}{4}V_{\infty}(t) \label{eq:LagrSystem2}, \\ 
	V(t, \xi) &= \int_{-\infty}^{\xi}\!\! \left(1 - \alpha \chi_{\{ \omega: t \geq \tau(\omega) > 0 \}}(\eta) \right)\!\bar{V}_{\xi}(\eta)d\eta \label{eq:LagrSystem3}. 
\end{align}
\end{subequations}
Here $V_{\infty}(t)$ denotes the right asymptote of $V(t, \xi)$ and $\tau:\R \rightarrow [0, \infty]$ is defined by 
\begin{align}\label{eq:waveBreaking}
	\tau(\xi) &= \begin{cases}
	0, & \bar{U}_{\xi}(\xi) = \bar{y}_{\xi}(\xi) = 0, \\
	-2\frac{\bar{y}_{\xi}(\xi)}{\bar{U}_{\xi}(\xi)}, & \bar{U}_{\xi}(\xi) < 0, \\
	\infty, & \text{otherwise},
	\end{cases}
\end{align}
and tells us when the solution experiences wave breaking in Lagrangian coordinates. In particular, the Lagrangian energy density $V_{\xi}(t, \xi)$ may drop abruptly at $t=\tau(\xi)$, cf. \eqref{eq:LagrSystem3}, and the right-hand side of \eqref{eq:LagrSystem} may therefore become discontinuous at the strictly positive, finite times given by $\tau$. In spite of that, it has been established in \cite[Lem. 2.2.1]{PhDNordli}, that \eqref{eq:LagrSystem} is globally well-posed for initial data belonging to the set  of Lagrangian coordinates $\F$, which in fact is an {\em invariant set} for \eqref{eq:LagrSystem}.   

 For its definition, we introduce the Banach space $B\! := E_2 \times E_2 \times E_1$, equipped with the norm 
\begin{equation*}
	\|(f_1, f_2, f_3)\|_{B} := \|f_1\|_{E_2} + \|f_2\|_{E_2} + \|f_3\|_{E_1}.
\end{equation*}
 
\begin{definition}\label{def:LagSet} Let $t \geq 0$. The solution $(y, U, V)(t)$ of \eqref{eq:LagrSystem} with initial data $(\bar{y}, \bar{U}, \bar{V})$ is said to belong to the space $\F$ for all $t\geq 0$, if the following holds. 
\begin{enumerate}[label=(\roman*)]
  \item $(y-\id, U, V)(t) \in \big[W^{1, \infty}(\R) \big]^3 \cap B$,
  \item $ y_{\xi}(t), \bar{V}_{\xi} \geq 0$ a.e., and there exists $c > 0$ such that $y_{\xi}(t) + \bar{V}_{\xi} \geq c $ holds a.e.,
  \item \label{def:LagSet3} $y_{\xi}V_{\xi}(t)= U_{\xi}^2(t)$ a.e., 
  \item $0 \leq V_{\xi}(t) \leq \bar{V}_{\xi}$ a.e., \label{def:relVH}
  \item If $\alpha = 1$, $y_{\xi}(t, \xi)\!=0$ implies $V_{\xi}(t, \xi)\!=0$ a.e., and $y_{\xi}(t, \xi) > 0$ implies \phantom{aa} $V_{\xi}(t, \xi)\! = \bar{V}_{\xi}(\xi)$ a.e.,
  \item If $\alpha \in [0, 1)$, then there exists a function $\kappa\!: \R \rightarrow \{1 - \alpha, 1 \}$ such that $V_{\xi}(t, \xi) = \kappa(\xi)\bar{V}_{\xi}(\xi)$ a.e., and $\kappa(\xi)=1$ whenever $U_{\xi}(t, \xi) < 0$. 
\end{enumerate}
\end{definition}

In addition, based on \eqref{eq:LagrSystem}, we also define the Lagrangian solution operator $S_t$. 

\begin{definition}\label{def:lagrangianSolution}
	For any $t \geq 0$ and $\bar{X} \in \F$, let $S_t(\bar{X}) \!= X(t)$, where $X(t)$ is the unique $\alpha$-dissipative solution of \eqref{eq:LagrSystem} satisfying $X(0)\! = \bar{X}$.
\end{definition} 

At last, in order to obtain the sought solution $(u, \mu)$ in Eulerian coordinates, one applies the mapping $M: \F \rightarrow \D$ given by $(u, \mu) = M((y, U, V))$, where 
    \begin{subequations}\label{eq:MapM}
        \begin{align}
            u(x) &= U(\xi) \text{ for any } \xi  \text{ such that } x = y(\xi)\label{eq:MapM1}, \\
            d\mu &= y_{\#} \left(V_{\xi}d\xi \right)\label{eq:MapM2}\!.
        \end{align}
    \end{subequations}
    Here $y_{\#}(V_{\xi}d\xi)$ denotes the pushforward measure given by 
    \begin{equation*}
    	\mu(A) = \int_{y^{-1}(A)} V_{\xi}(\xi)d\xi, \hspace{0.35cm} \text{for any Borel set } A \subseteq \R. 
	\end{equation*}
 
 To summarize, the $\alpha$-dissipative solution in Eulerian coordinates is obtained by combining the Lagrangian solution operator $S_t$ with the mappings $L$ and $M$, defined in \eqref{eq:MapL} and \eqref{eq:MapM}, respectively.  

\begin{definition}\label{def:alphaSolution}
	Let $(\bar{u}, \bar{\mu}) \in \D$, $t \geq 0$, and define $T_t: \D \rightarrow \D$ by $T_t := M \circ S_t \circ L$. The $\alpha$-dissipative solution at time $t$ is then given by 
	\begin{equation*}
		(u, \mu)(t) := T_t \left((\bar{u}, \bar{\mu}) \right) = M \circ S_t \circ L \left( (\bar{u}, \bar{\mu})\right)\!.
	\end{equation*}
\end{definition}

 Let us pause for a moment and compare the outlined approach with the one presented in e.g.,  \cite{AlphaAlgorithm, AlphaHS, PhDNordli} as there are some differences. More precisely, we have chosen to represent Eulerian solutions as pairs $(u, \mu)$ rather than triplets $(u, \mu, \nu)$ as considered in the aforementioned works. The reason is that $\nu$ is a measure added exclusively for technical reasons, in the sense that one recovers the same Eulerian solution $(u, \mu)(t)$ for all $t \geq 0$ irregardless of which $\nu$ one picks initially, see \cite[Lem. 2.13]{NewestLipschitzMetric} for a precise statement. However, this comes at the cost of losing the semigroup property of $T_t$, but from a numerical viewpoint this is of no concern, because we only map into Lagrangian coordinates initially and back again to Eulerian coordinates at the time for which we seek the numerical solution. In other words, viewing Eulerian coordinates as triplets, with the appropriate modifications of Definition~\ref{def:EulerSet}, \eqref{eq:MapL}, Definition~\ref{def:LagSet}, and \eqref{eq:MapM}, does not influence the upcoming error analysis, but it leads to more cumbersome notation.

The procedure outlined in \cite[Def. 3.1]{AlphaAlgorithm} for obtaining numerical solutions is closely related to that of Definition~\ref{def:alphaSolution}, but it involves one additional step: projecting the initial data onto a uniform mesh. In particular, by projecting the initial data $(\bar{u}, \bar{F})$ into the space of piecewise linear functions, the solution operator $T_t$ can be implemented exactly, as the solution associated with this projected initial data will remain piecewise linear at all later times, have a look at \cite{InverseScattering, PropertiesHS} for more details.  

To present the projection, let $\{x_j \}_{j \in \mathbb{Z}}$ denote a uniform discretization of $\R$, with $x_j = j \Dx$ for $\Dx > 0$ and $j \in \mathbb{Z}$. Furthermore, let $f_j = f(x_j)$ denote the values attained at the gridpoints $\{x_j\}_{j \in \mathbb{Z}}$ for any function $f:\R \rightarrow \R$ and introduce the difference operator \vspace{-0.1cm}
\begin{equation*}
	Df_{2j} = \frac{f_{2j+2} - f_{2j}}{2\Dx},
\end{equation*} 
in addition to the quantity  \vspace{-0.15cm}
\begin{equation}\label{eq:square_root}
	q_{2j} := \sqrt{DF_{\mathrm{ac}, 2j} - (Du_{2j})^2}. 
\end{equation}
The projection operator $P_{\Dx}$ is then defined as follows.  \vspace{-0.1cm}

\begin{definition}[{\cite[Def. 3.2]{AlphaAlgorithm}}]\label{def:ProjOP}
    Define $P_{\Delta x}\!:\!\D\to \D$ by $P_{\Delta x} \left((u,F) \right)\!=(u_{\Delta x},F_{\Delta x})$, where
    \begin{align}
    u_{\Delta x}(x) &= \begin{cases}
    u_{2j} + \left(Du_{2j} \mp  q_{2j}\right)(x-x_{2j}), & x_{2j} < x \leq x_{2j+1}, \\
    u_{2j+2} + \left(Du_{2j} \pm  q_{2j}\right)(x-x_{2j+2}), & x_{2j+ 1} < x \leq x_{2j+2}, 
    \end{cases}
    \label{eq:Proj_u} \end{align}
    and 
    \begin{equation}
      F_{\Delta x}(x) = F_{\Delta x, \mathrm{ac}}(x) + F_{\Delta x, \mathrm{sing}}(x),  \hspace{0.45cm} \text{for all } x \in \R. \label{eq:Proj_F} 
    \end{equation}
   Moreover, the absolutely continuous part of $F_{\Delta x}$ is given by 
    \begin{align}
    F_{\Delta x, \mathrm{ac}}(x) &= \begin{cases} F_{\mathrm{ac}, 2j} + \left(Du_{2j} \mp q_{2j}\right)^2 \left(x - x_{2j} \right)\!, & x_{2j} < x \leq x_{2j+1}, \\ 
    \frac{1}{2}\left(F_{\mathrm{ac}, 2j+2}+ F_{\mathrm{ac}, 2j}\right) \mp 2q_{2j}Du_{2j} \Delta x \\ \quad + \left(Du_{2j} \pm q_{2j} \right)^2(x-x_{2j+1}), & x_{2j+1} < x \leq x_{2j+2},
    \end{cases}
    \label{eq:Proj_Fac} \end{align} 
    while the singular part is defined as
    \begin{align}\label{eq:projSingular} 
    F_{\Delta x, \mathrm{sing}}(x) &= F_{2j+2} - F_{ \mathrm{ac}, 2j+2}=F_{\mathrm{sing}, 2j+2}, \qquad  x_{2j} < x \leq x_{2j+2}. 
    \end{align}
\end{definition} 

This particular projection operator is tailored to preserve the total energy $F_{\infty}$, the continuity of $u$, and Definition~\ref{def:EulerSet}~\ref{def:EulerCond3} -- it therefore maps $\D$ into itself. 

The observant reader may have noticed that there is an ambiguity in \eqref{eq:Proj_u} and \eqref{eq:Proj_Fac}, i.e., a choice of sign that has to be made. In the upcoming analysis the choice of sign is irrelevant, but in practical applications one would typically impose a sign-selection criterion, see \cite[Sec. 3.1]{AlphaAlgorithm} for further details. 

Finally, the numerical solution is defined in the following way. 

\begin{definition}[{\cite[Def. 3.1]{AlphaAlgorithm}}]The numerical solution $(u_{\Dx}, F_{\Dx})(t)$ at $t\geq 0$ is for any $(\bar{u}, \bar{F}) \in \D$ defined by 
\begin{equation*}
	(u_{\Dx}, F_{\Dx})(t) := T_t \circ P_{\Dx}((\bar{u}, \bar{F})).
\end{equation*}
\end{definition}

The interested reader is referred to \cite[Sec. 3.2]{AlphaAlgorithm} for a detailed account about the numerical implementation of the solution operator $T_t$. Here we would like to emphasize that the mapping $L$, defined by \eqref{eq:MapL}, maps the uniform mesh in Eulerian coordinates $\{x_j \}_{j \in \mathbb{Z}}$ into a nonuniform discretization in Lagrangian coordinates $\{\xi_j \}_{j \in \mathbb{Z}}$, where each pair $[x_{2j}, x_{2j+1}]$ and $[x_{2j+1}, x_{2j+2}]$ possibly is mapped into three Lagrangian gridcells with endpoints given by 
\begin{align} \label{eq:lagr_grid}
\begin{aligned}
	\xi_{3j} &= x_{2j} + \bar{F}_{\Dx}(x_{2j}), \hspace{1.4cm} \xi_{3j+1} = x_{2j} + \bar{F}_{\Dx}(x_{2j}+), \\ 
	\xi_{3j+2} &= x_{2j+1} + \bar{F}_{\Dx}(x_{2j+1}), \hspace{0.7cm}
	\xi_{3j+3} = x_{2j+2} + \bar{F}_{\Dx}(x_{2j+2}).
\end{aligned}
\end{align}
The notation $x_{2j}+$ is used to denote the right limit at $x=x_{2j}$ and we observe that $\xi_{3j+1}\! = \xi_{3j}$ if $\bar{F}_{\Dx}$ does not admit a jump discontinuity at $x=x_{2j}$. Furthermore, \eqref{eq:MapL} and Definition~\ref{def:ProjOP} imply that the numerical Lagrangian initial data,  
\begin{equation}\label{eq:projectedLagr}
	\bar{X}_{\Dx} = L \circ P_{\Dx}((\bar{u}, \bar{\mu})), 
\end{equation}
is continuous and piecewise linear with nodes situated at $\{\xi_j \}_{j \in \mathbb{Z}}$. The value of $\bar{X}_{\Dx}(\xi)$, for any $\xi \in \R$, can therefore be computed with linear interpolation based on the values attained at $\{\xi_j\}_{j \in \mathbb{Z}}$. As a consequence, the derivatives $\bar{X}_{\Dx, \xi}=(\bar{y}_{\Dx, \xi}, \bar{U}_{\Dx, \xi}, \bar{V}_{\Dx, \xi})$ are piecewise constant functions that may have jump discontinuities at the Lagrangian gridpoints $\{\xi_j\}_{j \in \mathbb{Z}}$. Furthermore, the numerical wave breaking function $\tau_{\Dx}$ is computed via \eqref{eq:waveBreaking}, when we replace $\bar{X}$ with the projected data $\bar{X}_{\Dx}$ from \eqref{eq:projectedLagr}. Consequently, one infers that the times and points for which wave breaking occurs for the exact and numerical solution in general differ. This complicates the upcoming error analysis.

\section{Convergence rate initially in Eulerian coordinates}\label{sec:Convergenceux}
Assume from now on that $\Dx \leq 1$. We start by  investigating the rate of convergence for the projected initial data $(\bar{u}_{\Dx}, \bar{F}_{\Dx}) = P_{\Dx} ((\bar{u}, \bar{F}))$. To this end, let us recall the following result.

\begin{prop}[{\cite[Prop. 4.1 and Lem. 4.2]{AlphaAlgorithm}}]\label{prop:ProjEstimates}
    Given $(\bar{u}, \bar{F}) \in \D$, let $(\bar{u}_{\Delta x}, \bar{F}_{\Delta x})=P_{\Delta x} \left((\bar{u}, \bar{F}) \right)$, then \vspace{-0.1cm}
 \begin{subequations}
    \begin{align}
        \|\bar{u} - \bar{u}_{\Delta x}\|_{\infty} & \leq \left(1 + \sqrt{2} \right)\!F_{\mathrm{ac}, \infty}^{\nicefrac{1}{2}} \Delta x^{\nicefrac{1}{2}}, \label{eq:Proj_u_infty}\\
        \|\bar{u} - \bar{u}_{\Delta x}\|_{2} & \leq \sqrt{2}\left(1+ \sqrt{2} \right) \!F_{\mathrm{ac}, \infty}^{\nicefrac{1}{2}} \Delta x, \label{eq:Proj_u_L2}  \\ 
        \| \bar{F} - \bar{F}_{\Delta x}\|_p &\leq 2 F_{\infty}\Delta x^{\nicefrac{1}{p}},  \quad \mathrm{for} \hspace{0.17cm} p=1,2.\label{eq:Proj_F_Lp}
        \end{align}
    \end{subequations}
    Furthermore, as $\Dx \rightarrow 0$ we have \vspace{-0.1cm}
        \begin{equation}
        		\bar{u}_{\Dx, x} \rightarrow \bar{u}_x \text{ in }  L^2(\R) \label{eq:Proj_ux}.
         \end{equation}
\end{prop}
The aim of this section is to strengthen \eqref{eq:Proj_ux}, namely derive a bound of the form $\|\bar{u}_x - \bar{u}_{\Dx, x} \|_2 \!\leq \mathcal{O}(\Dx^{\gamma})$ for some $\gamma > 0$. This will be essential in the next section when we examine how the time evolution in Lagrangian coordinates affects the initial convergence rate in \eqref{eq:Proj_u_infty}. In particular, if we lack a convergence rate for $\bar{u}_{\Dx, x}$, then the initial order of convergence from \eqref{eq:Proj_u_infty} will be completely deteriorated by the solution operator $S_t$ from Definition~\ref{def:lagrangianSolution} for $\alpha$-dissipative solutions. This derives from the fact that, for any $t > 0$, the set \vspace{-0.1cm}
\begin{align*}
	E_t &:=  \Bigl \{ \xi \in \R: \bar{u}_x(y(\xi)) = - \frac{2}{t} \Bigr \}, 
\end{align*}
completely determines all the characteristics that are involved in wave breaking at time $t>0$ in Lagrangian coordinates, and thereby the amount of energy removed at this time upon wave breaking. The set \vspace{-0.1cm}
\begin{equation*}
	E_{\Dx, t} := \Bigl \{ \xi \in \R:  \bar{u}_{\Dx, x}(y_{\Dx}(\xi)) = - \frac{2}{t} \! \Bigr \}, 
\end{equation*}
plays a similar role for the numerical solution. Thus, in order to avoid that the error between the exact- and numerical Lagrangian solution becomes too large, which would prevent us from obtaining a convergence rate for $t > 0$, one needs some way of controlling the difference in measure between $E_s$ and $E_{\Dx, s}$ for all $s \in (0, t]$. This is provided by the stronger $L^2(\R)$-convergence rate for $\bar{u}_{\Dx, x}$. 

Since $\bar{u}_{\Dx}$, given by \eqref{eq:Proj_u}, is piecewise linear,  we are approximating $\bar{u}_x$ by the piecewise constant function, $\bar{u}_{\Dx, x}$. A convergence rate in $L^2(\R)$ is therefore only obtainable if we impose additional regularity on $\bar{u}_x$. 

The following family of spaces \vspace{-0.075cm}
\begin{align}\label{eq:LipSpaces}
	B_2^{\beta}\!:= \! \left\{f \!\in \! L^2(\R)\!: \exists \hspace{0.05cm} C\! > 0 \text{ s.t. }\! \|\delta_hf \|_2 \! \leq C h^{\beta} \hspace{0.05cm} \text{for all } h \in (0, 2]\right\}\!, 
\end{align}
appears naturally in this context. Here $\beta \in (0, 1]$ is fixed and $\delta_{h}f(x) := f(x +h)-f(x)$. Furthermore, the smallest admissible constant $C>0$ is denoted by $|f|_{2, \beta}$ and defines a seminorm, namely 
\begin{equation}\label{eq:semiNorm}
		|f|_{2, \beta} := \sup_{t \in (0, 2]}  \Big( t^{-\beta} \sup_{h \in [0, t]}\|\delta_hf \|_2 \Big).
\end{equation} 
These spaces are closely related to the {\em Besov spaces}, $B_{\infty}^{\beta}(L^2(\R))$, and the {\em Lipschitz spaces}, $\mathrm{Lip}(\beta, L^2(\R))$, see e.g., \cite{BogachevNonlinearInt, NonlinearApproxBook, ConstructiveApprox, Nikolskii}. Moreover, $B_2^{\beta}$ contains all the $L^2(\R)$-functions that optimally can be approximated with order $\mathcal{O}(\Dx^{\beta})$ by piecewise constant functions on equidistant grids, see \cite[Sec. 3]{NonlinearApproxBook} and \cite[Chp. 12]{ConstructiveApprox}. Hence these spaces are well-suited in our setting. Also note that we have restricted ourselves to $\beta \leq 1$, because when $\beta > 1$, \eqref{eq:semiNorm} is finite only for functions that are almost everywhere constant, see \cite[Chp. 2]{ConstructiveApprox}. 
 
\begin{lemma}\label{lem:Rateux}
    For $\beta \in (0, 1]$ and $(\bar{u}, \bar{\mu}) \in \D$ with $\bar{u}_{x} \in B_2^{\beta}$, let $(\bar{u}_{\Dx}, \bar{F}_{\Dx}) = P_{\Dx}((\bar{u}, \bar{F}))$, then there exists a constant $D$ dependent on $\beta$, $F_{\mathrm{ac}, \infty}$ and $|\bar{u}_x|_{2, \beta}$ such that \label{lem:Rateux1} 
    	\begin{equation*}
        \|\bar{u}_x - \bar{u}_{\Delta x, x} \|_2 \leq D\Delta x^{\nicefrac{\beta}{2}}.
            \end{equation*}
\end{lemma}

\newcommand\myiineq{\stackrel{\mathclap{\normalfont\mbox{\eqref{eq:LipSpaces}}}}{\leq}}

\begin{proof} 
   Combining \eqref{eq:Proj_u} with the estimate $(a+b)^2 \leq 2 (a^2 + b^2)$ where $a = \bar{u}_x - D\bar{u}_{2j}$ and $b=\pm q_{2j}=\pm \sqrt{D\bar{F}_{\mathrm{ac}, 2j} - (D\bar{u}_{2j})^2}$ leads to 
    \begin{subequations}
    \begin{align}
        \|\bar{u}_x - \bar{u}_{\Delta x, x} \|_2^2 &\leq 2\sum_{j \in \mathbb{Z}} \int_{x_{2j}}^{x_{2j+2}} \left(\bar{u}_x(x) - D\bar{u}_{2j} \right)^2dx \label{eq:con_ux1}
        \\ & \qquad + 2\sum_{j \in \mathbb{Z}} \int_{x_{2j}}^{x_{2j+2}} \left(D\bar{F}_{\mathrm{ac}, 2j} - (D\bar{u}_{2j})^2 \right)dx, \label{eq:con_ux2}
    \end{align}
    \end{subequations}
    and by further decomposing \eqref{eq:con_ux2} we get
     \begin{subequations}\label{eq:con_ux3}
\begin{align}
    \sum_{j \in \mathbb{Z}} \int_{x_{2j}}^{x_{2j+2}} \left(D\bar{F}_{\mathrm{ac}, 2j} - (D\bar{u}_{2j})^2 \right)dx &\leq \sum_{j \in \mathbb{Z}} \int_{x_{2j}}^{x_{2j+2}} \left|D\bar{F}_{\mathrm{ac}, 2j} - \bar{u}_x^2(x) \right|dx \label{eq:con_ux3_1}\\ & \qquad + \sum_{j \in \mathbb{Z}} \int_{x_{2j}}^{x_{2j+2}} \left|\bar{u}_x^2(x) - (D\bar{u}_{2j})^2 \right|dx.  \label{eq:con_ux3_2}
\end{align}
\end{subequations}

We start by estimating \eqref{eq:con_ux1}, to which we apply the Cauchy--Schwarz (CS) inequality and argue as follows    
\begin{align}
    \sum_{j \in \mathbb{Z}} \int_{x_{2j}}^{x_{2j+2}} \left(\bar{u}_x(x) - D\bar{u}_{2j} \right)^2dx \nonumber 
    &= \sum_{j \in \mathbb{Z}} \int_{x_{2j}}^{x_{2j+2}} \bigg(\frac{1}{2\Delta x} \int_{x_{2j}}^{x_{2j+2}}(\bar{u}_x(x) - \bar{u}_x(z))dz \bigg)^{\!2}\!dx  \nonumber 
    \\ & \myineq  \sum_{j \in \mathbb{Z}} \int_{x_{2j}}^{x_{2j+2}} \frac{1}{2\Delta x} \int_{x - 2\Delta x}^{x + 2\Delta x}|\bar{u}_x(x) - \bar{u}_x(z)|^2dz dx \nonumber 
    \\ & = \sum_{j \in \mathbb{Z}} \int_{x_{2j}}^{x_{2j+2}} \! \! \frac{1}{2\Delta x} \int_{-2\Delta x}^{2\Delta x} \!|\bar{u}_x(x) - \bar{u}_x(x + h)|^2dh dx \nonumber 
    \\&\myeq \frac{1}{2\Dx} \int_{-2\Dx}^{2\Dx} \|\delta_h\bar{u}_x \|_2^2dh \nonumber 
    \\ &\!\myiineq \frac{1}{2\Delta x} \int_{-2\Delta x}^{2\Delta x} \!|\bar{u}_x|_{2, \beta}^2|h|^{2\beta}dh = \frac{2^{1+2\beta}|\bar{u}_x|_{2, \beta}^2}{2\beta + 1}\Delta x^{2\beta}.
    \label{eq:ux_est1}
\end{align}

Note that \eqref{eq:con_ux3_1} can be recast on the form 
\begin{equation*}
	 \sum_{j \in \mathbb{Z}} \int_{x_{2j}}^{x_{2j+2}} \left|D\bar{F}_{\mathrm{ac}, 2j} - \bar{u}_x^2(x) \right|dx =  \sum_{j \in \mathbb{Z}} \int_{x_{2j}}^{x_{2j+2}} \! \bigg|\frac{1}{2\Dx} \!\int_{x_{2j}}^{x_{2j+2}} \! \!\left(\bar{u}_x^2(z) - \bar{u}_x^2(x)\right) \!dz \bigg|dx,
\end{equation*}
and, in addition, by the CS inequality, 
\begin{equation}
    \|\delta_h\bar{u}_x^2\|_1 \leq \|\bar{u}_x(\cdot +h) + \bar{u}_x(\cdot) \|_2 \|\delta_h\bar{u}_x\|_2 \leq 2\bar{F}_{\mathrm{ac}, \infty}^{\nicefrac{1}{2}}|\bar{u}_x|_{2, \beta}|h|^{\beta}, 
    \label{eq:L1_normLip}
\end{equation}
which in combination with \eqref{eq:ux_est1} implies 
\begin{equation}
    \sum_{j \in \mathbb{Z}} \int_{x_{2j}}^{x_{2j+2}} \left|D\bar{F}_{\mathrm{ac}, 2j} - \bar{u}_x^2(x) \right|dx \leq \frac{2^{2+\beta}\bar{F}_{\mathrm{ac}, \infty}^{\nicefrac{1}{2}}|\bar{u}_x|_{2, \beta}}{\beta + 1}\Delta x^{\beta}. 
    \label{eq:ux_est2_1}
\end{equation}
It therefore remains to consider \eqref{eq:con_ux3_2}. By proceeding as in \eqref{eq:L1_normLip}, we find 
\begin{align}\label{eq:ux_est2_2_1}
    \sum_{j \in \mathbb{Z}} \int_{x_{2j}}^{x_{2j+2}} &\bigl |\bar{u}_x^2(x) - (D\bar{u}_{2j})^2 \bigr|dx \nonumber \\
    &\hspace{-0.35cm}\leq \! \bigg(\sum_{j \in \mathbb{Z}} \int_{x_{2j}}^{x_{2j+2}} \! \!\left|\bar{u}_x(x) - D\bar{u}_{2j} \right|^2 \!dx\! \bigg)^{\nicefrac{1}{2}} \! \bigg(\sum_{j \in \mathbb{Z}} \int_{x_{2j}}^{x_{2j+2}} \! \! \left|\bar{u}_x(x) + D\bar{u}_{2j}\right|^2 \!dx\! \bigg)^{\nicefrac{1}{2}} \nonumber
    \\ &\hspace{-0.35cm} \leq  2^{1+\beta}\sqrt{\frac{2}{2\beta + 1}} \bar{F}_{\mathrm{ac}, \infty}^{\nicefrac{1}{2}}|\bar{u}_x|_{2, \beta}\Delta x^{\beta}\!, 
\end{align}
where we used that 
\begin{align*}
   \sum_{j \in \mathbb{Z}} \int_{x_{2j}}^{x_{2j+2}}|&\bar{u}_x(x) + D\bar{u}_{2j}|^2dx \nonumber \\  
   & \leq 2 \sum_{j \in \mathbb{Z}} \int_{x_{2j}}^{x_{2j+2}}\bar{u}_x^2(x)dx + 2 \sum_{j \in \mathbb{Z}} \int_{x_{2j}}^{x_{2j+2}}\biggl(\frac{1}{2\Delta x}\int_{x_{2j}}^{x_{2j+2}}\bar{u}_x(z)dz \biggr)^{\!2} dx \nonumber \\ 
    & \leq 2\bar{F}_{\text{ac}, \infty} + 2 \sum_{j \in \mathbb{Z}} \frac{1}{2\Delta x} \int_{x_{2j}}^{x_{2j+2}} \int_{x_{2j}}^{x_{2j+2}}\bar{u}_x^2(z)dz dx = 4\bar{F}_{\mathrm{ac}, \infty}.
\end{align*}
A combination of \eqref{eq:ux_est1} and \eqref{eq:ux_est2_1}--\eqref{eq:ux_est2_2_1} now yields the result.
\end{proof}

A closer inspection of the above proof, reveals that all the terms that appear when estimating $\|\bar{u}_x - \bar{u}_{\Dx, x}\|_2$ involve the factor $\|\delta_h \bar{u}_x\|_2$. It is therefore to be expected, that the convergence rate in Lemma~\ref{lem:Rateux} deteriorates with decreasing values of $\beta$, because $\bar{u}_x$ behaves more poorly with respect to $L^2$-translates as $\beta$ becomes smaller.

\begin{remark}
	As functions in $B_2^{\beta}$ optimally can be approximated of order $\mathcal{O}(\Dx^{\beta})$ by piecewise constants, the convergence rate obtained in Lemma~\ref{lem:Rateux} is suboptimal. The proof reveals that $q_{2j}$, defined in \eqref{eq:square_root}, is the cause of this suboptimality, but this term is crucial in order to ensure that $d\bar{\mu}_{\Delta x, \textrm{ac}}=\bar{u}_{\Delta x, x}^2dx$ is preserved. 
	
	If this property is violated, then as pointed out in \cite{AlphaAlgorithm, NumericalConservative}, the numerical approximations would be solutions of the two-component HS system which is accompanied by a density $\rho$. It has been established in \cite[Thm. 6.1]{Conservative2CH}, that  if $\rho$ is nonzero on an interval $[x_1, x_2]$, then it induces a repulsive force that prevents characteristics starting within $[x_1, x_2]$ from meeting each other such that no wave breaking occurs for these characteristics. 
\end{remark}

\section{Convergence rate in the $\alpha$-dissipative case, $\alpha \in [0, 1]$}\label{sec:ConvergenceRateu}
According to \cite[Lem. 4.11]{AlphaAlgorithm}, we have
\begin{equation}\label{eq:est_u}
	\|u(t) - u_{\Dx}(t)\|_{\infty} \leq \|U(t) - U_{\Dx}(t)\|_{\infty} + \bar{F}_{\infty}^{\nicefrac{1}{2}}\|y(t) - y_{\Dx}(t)\|_{\infty}^{\nicefrac{1}{2}}.
\end{equation}
Hence, in order to derive a convergence rate in $L^{\infty}(\R)$ for the family $\{u_{\Delta x}(t)\}_{\Dx > 0}$, it is enough to obtain a convergence rate for $(y_{\Dx}\!-\!\id, U_{\Dx})(t)$ in $[L^{\infty}(\R)]^2$. 
However, there might exist a set $\mathcal{A}$ of positive measure on which we have no initial convergence rate for $\bar{V}_{\Dx, \xi}$ which is problematic, because both $U_{\Dx}(t)$ and $y_{\Dx}(t)$ heavily depend on $\bar{V}_{\Dx, \xi}$, cf. \eqref{eq:LagrSystem}. To resolve this issue, we introduce a transformation that enables us to treat the error on the set $\mathcal{A}$, while we develop an a priori estimate for the remaining error on $\mathcal{A}^c$. 

Let us start by recalling the following error estimates for the Lagrangian data.   

\begin{lemma}[{\cite[Lem. 4.4]{AlphaAlgorithm}}]\label{lem:convInt}
    Given $(\bar{u}, \bar{\mu}) \in \D$, let $\bar{X} = \left(\bar{y}, \bar{U}, \bar{V} \right)= L \left((\bar{u}, \bar{\mu}) \right)$, and $\bar{X}_{\Delta x}= (\bar{y}_{\Delta x}, \bar{U}_{\Delta x}, \bar{V}_{\Delta x}) = L \circ P_{\Delta x} \left( (\bar{u}, \bar{\mu}) \right)$, then 
    \begin{subequations}
    \begin{align}
        \|\bar{y} - \bar{y}_{\Delta x} \|_{\infty} &\leq 2\Delta x, \label{eq:initial_approx_y}\\
        \|\bar{U} - \bar{U}_{\Delta x} \|_{\infty} & \leq \left(1 + 2\sqrt{2} \right)\!\bar{F}_{\mathrm{ac}, \infty}^{\nicefrac{1}{2}}\Delta x^{\nicefrac{1}{2}},\label{eq:initial_approx_U} \\
        \|\bar{V} - \bar{V}_{\Delta x} \|_{\infty} & \leq 2 \Delta x. \label{eq:initial_approx_H}
    \end{align}
    \end{subequations}
\end{lemma}

Furthermore, we collect two estimates that readily follow from \eqref{eq:LagrSystem}. 
\begin{prop}\label{prop:simpleEst}
	Given $(\bar{u}, \bar{\mu}) \in \D$ and $t \geq 0$, let $\bar{X} = L \left( (\bar{u}, \bar{\mu}) \right)$ and \newline $\bar{X}_{\Dx} = L \circ P_{\Dx} \left((\bar{u}, \bar{\mu}) \right)$. In addition, define $(y_{\Dx}, U_{\Dx}, V_{\Dx})(t) = S_t (\bar{X}_{\Dx})$ and $(y, U, V)(t) = S_t(\bar{X})$, then  \vspace{-0.1cm}
	\begin{equation}\label{eq:esty}
		\|y(t) - y_{\Dx}(t)\|_{\infty} \leq \|\bar{y} - \bar{y}_{\Dx}\|_{\infty} + \int_0^t\|U(s) - U_{\Dx}(s)\|_{\infty}ds.
		\end{equation}
	Moreover, the following pointwise estimate holds for any $\xi \in \R$, 
	\begin{align}\label{eq:estU}
	|U(t, \xi) &- U_{\Dx}(t, \xi)| \leq  \|\bar{U} - \bar{U}_{\Dx}\|_{\infty} \nonumber
		\\ & \hspace{-0.2cm}+\! \frac{1}{4} \bigg|\int_0^t\!\bigg( \int_{-\infty}^{\xi}\! \!\left(V_{\xi}(s, \eta) - V_{\Dx, \xi}(s, \eta)\right)\!\d\eta -\! \int_{\xi}^{\infty}\! \!\left(V_{\xi}(s, \eta) - V_{\Dx, \xi}(s, \eta)\right)\!d\eta \bigg)ds \bigg|. \raisetag{-20pt}
	\end{align}
\end{prop}
Due to Lemma~\ref{lem:convInt} and \eqref{eq:esty}, it only remains to estimate the second term in \eqref{eq:estU}. To this end, let us recall the function $g(X)$ from \cite[Def. 4.3]{AlphaHS} which anticipates the future energy loss. This function is based on disjointly partitioning $\R$ into a set of points where no wave breaking will take place in the future, $\Omega_c$, and its complement, $\Omega_d$, containing all the points that eventually will experience wave breaking. More precisely, let
\begin{equation}\label{eq:omegas}
\begin{aligned}
	\Omega_c(X) &= \{\xi \in \R\!: U_{\xi}(\xi) \geq 0 \}, \\
	\Omega_d(X) &= \{ \xi \in \R\!: U_{\xi}(\xi) < 0 \}, 
\end{aligned}
\end{equation}
and define \vspace{-0.1cm}
\begin{align}\label{eq:function_g}
	g(X)(\xi) &:= \begin{cases} 
		(1-\alpha)V_{\xi}(\xi), & \xi \in \Omega_d(X), \\
		V_{\xi}(\xi), & \xi \in \Omega_c(X).
		\end{cases}
\end{align}
The following result has been established in \cite[Thm. 4.9]{AlphaAlgorithm}. 

\begin{lemma}\label{lem:estimateVxiTime}
For $(\bar{u}, \bar{\mu}) \in \D$ and $t\geq 0$, let $\bar{X} = L \left( (\bar{u}, \bar{\mu})\right)$ and \newline $\bar{X}_{\Dx} = L \circ P_{\Delta x} \left((\bar{u}, \bar{\mu})\right)$. Moreover, introduce $X(t) = S_t(\bar{X})$ and $X_{\Delta x}(t) = S_{t}(\bar{X}_{\Dx})$, then
\begin{align}\label{eq:VxiTimeEst}
	\bigg|\int_0^t\!\bigg( \int_{-\infty}^{\xi}\! \!\left(V_{\xi}(s, \eta) - V_{\Dx, \xi}(s, \eta)\right)\!\d\eta &-\! \int_{\xi}^{\infty}\! \!\left(V_{\xi}(s, \eta) - V_{\Dx, \xi}(s, \eta)\right)\!d\eta \bigg)ds \bigg| \nonumber 
	\\ 
	&\leq (1+\alpha)t\|\bar{V}_{\xi} - \bar{V}_{\Delta x, \xi} \|_1 \nonumber \\ 
	& \quad + t\|g(\bar{X}) - g(\bar{X}_{\Delta x}) \|_1 \nonumber  \\ 
	& \quad + 2\alpha \sqrt{2 \Big(1 + \frac{1}{4}t^2 \Big)\bar{F}_{\infty}} \|\bar{U}_{\xi} - \bar{U}_{\Delta x, \xi} \|_2. 
\end{align}
\end{lemma}

Combining Lemma~\ref{lem:convInt}, Proposition~\ref{prop:simpleEst}, and Lemma~\ref{lem:estimateVxiTime} reveals that the task of obtaining a convergence rate for $(y_{\Dx}\!-\!\id, U_{\Dx})(t)$ in $[L^{\infty}(\R)]^2$ is reduced to understanding how the mapping $L$ transports the convergence rate in Lemma~\ref{lem:Rateux} for $\bar{u}_{\Dx, x}$ to the variables $(\bar{U}_{\Dx, \xi}, \bar{V}_{\Dx, \xi}, g(\bar{X}_{\Dx}))$. Unfortunately, there is no natural way to obtain a convergence rate for these variables on the set where either the exact or the numerical solution experience wave breaking initially. In particular, \eqref{eq:waveBreaking} combined with Definition~\ref{def:LagSet} imply that the exact solution breaks initially in Lagrangian coordinates at the points,
\begin{equation}\label{eq:S}
    \mathcal{S} := \{\xi \in \R: \bar{y}_{\xi}(\xi) = 0 \} = \{\xi \in \R: \bar{V}_{\xi}(\xi) = 1 \}. 
\end{equation}
Analogously, let $S_{\Dx}$ denote the corresponding set of points where wave breaking takes places initially for the numerical solution, 
\begin{equation}\label{eq:num_S}
    \mathcal{S}_{\Dx} \!:= \{\xi \in \R: \bar{y}_{\Dx, \xi}(\xi)\! = 0 \} = \{\xi \in \R: \bar{V}_{\Dx, \xi}(\xi) =1 \}.
\end{equation}
The reason we fail to obtain a convergence rate over $\mathcal{A} := \mathcal{S} \cup \mathcal{S}_{\Dx}$ can be explained as follows. Introduce $B=\bar{y}(\mathcal{S})$. From the proof of \cite[Thm. 27]{AlphaCH}, we have 
 \begin{equation}\label{eq:absmu}
    \bar{\mu}_{\mathrm{ac}}=\bar{\mu}\vert_{B^c} \quad \text{ and } \quad \bar{\mu}_{\mathrm{sing}}=\bar{\mu}\vert_B,
\end{equation}
where $\bar{\mu}\vert_B$ denotes the restriction of $\bar{\mu}$ to $B$, i.e., for any Borel set $A$,  $\bar{\mu}\vert_B(A)=\bar{\mu}(A\cap B)$. 
Similarly, set $B_{\Dx} = y_{\Dx}(\mathcal{S}_{\Dx})$, then 
\begin{equation}\label{eq:absmudx}
	 \bar{\mu}_{\Dx, \mathrm{ac}}=\bar{\mu}_{\Dx}\vert_{B_{\Dx}^c} \quad \text{ and } \quad \bar{\mu}_{\Dx, \mathrm{sing}}=\bar{\mu}_{\Dx}\vert_{B_{\Dx}}.
\end{equation}
The proof of \cite[Thm. 3.8]{LagrangianViewCH} reveals that \eqref{eq:absmu} and \eqref{eq:absmudx} together with Definition~\ref{def:EulerSet}~\ref{def:EulerCond3} and \eqref{eq:MapL} yield 
\begin{equation}\label{eq:rep_y}
	\bar{y}_{\xi}(\xi) = \frac{1}{1 + \bar{u}_{x}^2(\bar{y}(\xi))} \hspace{0.25 cm} \text{and } \hspace{0.25 cm} \bar{y}_{\Dx, \xi}(\xi) = \frac{1}{1 + \bar{u}_{\Dx, x}^2(\bar{y}_{\Dx}(\xi))}  \hspace{0.35cm} \text{for a.e. } \xi \in \mathcal{S}^c \cap \mathcal{S}_{\Dx}^c.
\end{equation}
These identities enable us to transport the rate of convergence in Lemma~\ref{lem:Rateux} to convergence rates for $(\bar{U}_{\Dx, \xi}, \bar{V}_{\Dx, \xi}, g(\bar{X}_{\Dx}))$ on $\mathcal{S}^c \cap \mathcal{S}_{\Dx}^c$. However, the identities in \eqref{eq:rep_y} fail to hold (either one or both of them) over $\mathcal{A}$, and one is therefore unable to obtain a convergence rate over this set. 

\begin{remark}\label{rem:Sdx}
	Actually, due to the very definition of $P_{\Dx}$, $\mathcal{S}_{\Dx}$ is a union of closed intervals. To see this, recall that $\bar{\mu} = \bar{\mu}_{\mathrm{ac}} + \bar{\mu}_{\mathrm{sing}}$ by \eqref{eq:splitF}, and that $\bar{\mu}_{\mathrm{sing}}$ can be decomposed further into two mutually singular measures, see e.g.,  \cite[Sec. 3]{RealAnalysisFolland} or \cite[Thm. 9.7]{McDonaldAnalysis}, 
\begin{equation*}
  	\bar{\mu}_{\mathrm{sing}} = \bar{\mu}_{\mathrm{sc}} + \bar{\mu}_{\mathrm{d}}, 
\end{equation*}
where $\bar{\mu}_{\mathrm{sc}}$ is singular continuous and $\bar{\mu}_{\mathrm{d}}$ is purely discrete. It then follows from Definition~\ref{def:ProjOP} and \eqref{eq:lagr_grid} that  
\begin{equation}\label{eq:explicitS_dx}
    \mathcal{S}_{\Delta x} = \bigcup_{j \in \mathbb{Z}}[\xi_{3j}, \xi_{3j+1}]= \bigcup_{j \in \mathbb{Z}}[x_{2j} + \bar{F}(x_{2j}), x_{2j} + \bar{F}(x_{2j}) + b_{2j}], 
\end{equation}
where \vspace{-0.125cm}
\begin{equation}\label{eq:singContrib}
	b_{2j} = \sum_{k \in I_{2j}}\bar{\mu}_{\mathrm{d}}(\{z_k \}) + \bar{\mu}_{\mathrm{sc}}([x_{2j}, x_{2j+2})).
\end{equation}
Here $I_{2j}\! \subseteq \! \mathcal{J}$ is an index set labeling the points of $K= \mathrm{supp}(\bar{\mu}_{\mathrm{d}}) = \{z_k \}_{k \in \mathcal{J}}$ contained within $[x_{2j}, x_{2j+2})$, where $\mathcal{J} \subseteq \mathbb{N}$ might be empty.  
\end{remark}

\subsection{The case $\mathbf{\boldsymbol{\bar{\mu}}_{\mathrm{sing}} = 0}$}\label{sec:noSingular} 
Thanks to \eqref{eq:absmu} and \eqref{eq:absmudx}, we have\footnote{The notation $|A|$ is used to denote the Lebesgue measure of the set $A \subseteq \R$} 
\begin{equation*}
	|\mathcal{S}| = \mu_{\mathrm{sing}}(\R) \hspace{0.35cm} \text{and} \hspace{0.35cm} |\mathcal{S}_{\Dx}| = \mu_{\Dx, \mathrm{sing}}(\R).
\end{equation*}
Moreover, since Definition~\ref{def:ProjOP}, or more precisely \eqref{eq:projSingular}, implies that $\bar{\mu}_{\Dx}$ is purely absolutely continuous whenever $\bar{\mu}$ is, it therefore follows that $|\mathcal{S} \cup \mathcal{S}_{\Dx}| = 0$ in this case. Nevertheless, wave breaking might still happen initially for the exact solution at isolated points $x \in \R$ satisfying $\bar{u}_{x}(x) = -\infty$, see e.g., Example~\ref{ex:Cusp}. However, as $\bar{u}_{x} \in L^2(\R)$ and each such point is mapped by $L$ from \eqref{eq:MapL} into a single point in Lagrangian coordinates, this set (on which $y_{\xi}$ and $V_{\xi}$ are undefined) has zero measure.  

In conclusion, the estimate from Lemma~\ref{lem:estimateVxiTime} can be used to derive a convergence rate for $\|U(t) - U_{\Dx}(t)\|_{\infty}$ whenever $\bar{\mu}$ is purely absolutely continuous, because $\mathcal{S}^c \cap \mathcal{S}_{\Dx}^c$ is then of full measure, which means that the identities in \eqref{eq:rep_y} can be used to relate the convergence rate from Lemma~\ref{lem:Rateux} to convergence rates for $(\bar{U}_{\Dx, \xi}, \bar{V}_{\Dx, \xi}, g(\bar{X}_{\Dx}))$ in $L^2(\R) \times [L^1(\R)]^2$.

\subsection{A mapping for treating initial energy concentrations}
Next, we extend our analysis to the case where $\bar{\mu}$ may have a singular part, $\bar{\mu}_{\mathrm{sing}}$ which is purely discrete. In this case, $\mathcal{S} \cup \mathcal{S}_{\Dx}$ may have positive measure, and our goal will be to derive a refinement of Lemma~\ref{lem:estimateVxiTime}. 

The idea is to define a bijection $f\!:\R \rightarrow \R$ which satisfies $f(\mathcal{S})\! = \mathcal{S}_{\Dx}\setminus E$ for a set $E$ with $|E|=0$, i.e., it should map the points where the exact solution breaks initially over to the set where the numerical solution breaks initially modulo some null set. Combining this function together with the fact that $\bar{\mu}_{\Dx, \mathrm{sing}}\big ([x_{2j}, x_{2j+2})\big) = \bar{\mu}_{\mathrm{sing}}\big([x_{2j}, x_{2j+2})\big)$ for any $j \in \mathbb{Z}$, i.e., $P_{\Dx}$ preserves the local mass of $\bar{\mu}_{\mathrm{sing}}$, enables us to prove that the set $\mathcal{A}$ contributes with at most order $\mathcal{O}(\Dx^{\beta})$ to the second term in \eqref{eq:estU}, whenever $\bar{u}_x \in B_2^{\beta}$. For the remaining error, which originates from the complement $\mathcal{A}^c$, we derive an estimate similar to the one in Lemma~\ref{lem:estimateVxiTime}. 

The key quantities involved in our argument, as well as their corresponding convergence rates are visualized in Figure~\ref{fig:schematicRep}. 

\begin{figure}
	\includegraphics{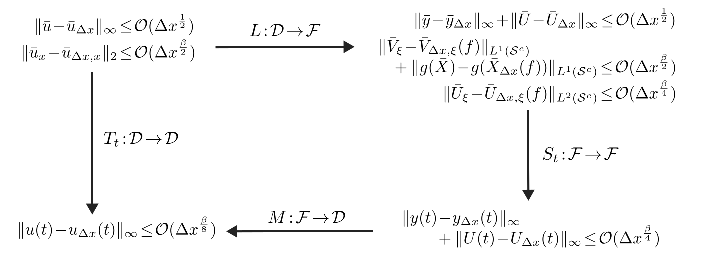}
	\captionsetup{width=.975\linewidth}
	\vspace{-2ex}
	\caption{A schematic representation of all the quantities involved when deriving a convergence rate for $\{u_{\Dx}(t)\}_{\Dx > 0}$, and their respective rates of convergence. Here $f$ (to be defined) denotes a bijection which maps $\mathcal{S}$ to $\mathcal{S}_{\Dx}$.}
	\label{fig:schematicRep}
\end{figure}

Let us start by observing that for any $j \in \mathbb{Z}$ we have, due to \eqref{eq:MapM}, Definition~\ref{def:ProjOP}, and \eqref{eq:absmu}--\eqref{eq:absmudx}, 
\begin{align*}
 \int_{[\xi_{3j}, \xi_{3j+3}) \cap \mathcal{S}_{\Dx}} \bar{V}_{\Dx, \xi}(\xi)d\xi  &=  \bar{\mu}_{\Dx, \mathrm{sing}}([x_{2j}, x_{2j+2}))  \nonumber
     \\ & = \bar{F}_{\Dx, \mathrm{sing}}(x_{2j+2}) - \bar{F}_{\Dx, \mathrm{sing}}(x_{2j}) \nonumber
     \\ &= \bar{F}_{\mathrm{sing}, 2j+2} - \bar{F}_{\mathrm{sing}, 2j} \nonumber 
     \\ &= \bar{\mu}_{\mathrm{sing}}([x_{2j}, x_{2j+2})) = \int_{[\xi_{3j}, \xi_{3j+3}) \cap \mathcal{S}}\bar{V}_{\xi}(\xi)d\xi, 
\end{align*} 
which implies, cf. \eqref{eq:S} and \eqref{eq:num_S}, \vspace{-0.1cm}
\begin{equation}\label{eq:measureS}
	\big|\mathcal{S}_{\Dx, 3j}\big| = \big|\mathcal{S}_{3j} \big|, 
\end{equation}
where we introduced 
\begin{equation}\label{eq:S3j}
	\mathcal{S}_{\Dx, 3j} = \mathcal{S}_{\Dx} \cap [\xi_{3j}, \xi_{3j+3}) \hspace{0.35cm} \text{and} \hspace{0.35cm} \mathcal{S}_{3j} :=  \mathcal{S} \cap [\xi_{3j}, \xi_{3j+3}).
\end{equation}
In other words, $P_{\Dx}$ is constructed in such a way that it preserves the length of $\mathcal{S}_{3j}$. Furthermore, as discussed in \cite[Sec. 5]{EquivalenceEulerLagrange}, any point $x$ where $\bar{F}$ is discontinuous corresponds to a maximal interval in Lagrangian coordinates $[\xi_{x, \mathrm{l}}, \xi_{x, \mathrm{r}}]$ such that
\begin{equation*}
	\bar{y}(\xi) = x \hspace{0.2cm} \text{ for } \xi \in  [\xi_{x, \mathrm{l}}, \xi_{x, \mathrm{r}}] \hspace{0.4 cm} \text{where} \hspace{0.4cm} \xi_{x, \mathrm{r}} - \xi_{x, \mathrm{l}} = \bar{\mu}(\{x\}) = \bar{F}(x+) - \bar{F}(x). 
\end{equation*}
Therefore, recalling \eqref{eq:MapL} and $K = \{z_k \}_{k\in \mathcal{J}}$ from Remark~\ref{rem:Sdx}, which coincides with the set of discontinuities for $\bar{F}$, it follows  that  \eqref{eq:S} can be recast on the form \vspace{-0.075cm}
\begin{equation} \label{eq:simplifiedS}
	\mathcal{S} = \bigcup_{z \in K} \!\left[\xi_{z, \mathrm{l}}, \xi_{z, \mathrm{r}}\right] \! =\bigcup_{k \in \mathcal{J}} \! \left[z_k + \bar{F}(z_k), z_k +\bar{F}(z_k+) \right]\!.
\end{equation}

Keeping \eqref{eq:measureS} in mind and comparing \eqref{eq:explicitS_dx} to \eqref{eq:simplifiedS}, since \newline $b_{2j} = \sum_{k \in I_{2j}} \bar{\mu}_{\mathrm{d}}(\{z_k\}) = \sum_{k \in I_{2j}}(\bar{F}(z_k+) - \bar{F}(z_k))$ by assumption, suggests that is should be possible to find a function $f\!:\R \rightarrow \R$ which maps the points where wave breaking occurs for $\bar{X}$ over to the set of points where $\bar{X}_{\Dx}$ breaks. That is, the function should satisfy 
\begin{equation}\label{eq:mapS}
	\bar{y}_{\Dx, \xi}(f(\xi)) = 0 \hspace{0.45cm} \text{for a.e. } \xi \in \mathcal{S}. 
\end{equation}
Thanks to \eqref{eq:MapL}, this will imply 
\begin{equation}\label{eq:vanishing}
	\bar{V}_{\xi}(\xi) - \bar{V}_{\Dx, \xi}(f(\xi)) = 0 \hspace{0.45cm} \text{for a.e. } \xi \in \mathcal{S}. 
\end{equation}
Let us now proceed for any $j \in \mathbb{Z}$ by constructing a local function $f_{3j}$ satisfying \eqref{eq:mapS} on $\mathcal{S}_{3j}$, and by thereafter gluing together these local functions, we obtain the desired, globally-defined function $f$ satisfying \eqref{eq:mapS}. 

To this end, recall $I_{2j}\! \! \subseteq \! \!\mathbb{N}$ which labels the indices such that $\{z_k\}_{k \in I_{2j}} = [x_{2j}, x_{2j+2}) \cap K$, and let $N_{2j}$ denote the cardinality of this index set.  We then distinguish between two cases: 

\textbf{If $\mathbf{N_{2j} = 0}$}: then $\bar{F}$ does not admit any discontinuities inside $[x_{2j}, x_{2j+2})$, hence $|\mathcal{S}_{3j}|=0 = |\mathcal{S}_{\Dx, 3j}|$ by \eqref{eq:measureS}, which in turn implies that both $\bar{y}$ and $\bar{y}_{\Dx}$ are strictly increasing a.e. on $[\xi_{3j}, \xi_{3j+3})$. Thus, let
\begin{equation}\label{eq:f_no}
	f_{3j}(\xi) =  \xi, \hspace{0.5cm} \xi \in[\xi_{3j}, \xi_{3j+3}). 
\end{equation}

\textbf{If} $\mathbf{N_{2j} \geq 1}$: let $\{k_m \}_{m=1}^{N_{2j}}$ be an enumeration of $I_{2j}$  and introduce 
\begin{equation}\label{eq:discPointLagr}
    \xi_{k_m, \mathrm{l}} = z_{k_m} + \bar{F}(z_{k_m}) \hspace{0.5cm} \text{and} \hspace{0.5cm} \xi_{k_m, r} = z_{k_m} + \bar{F}(z_{k_m}+), \quad m \in \{1, ... N_{2j}\},
\end{equation}
in which case \vspace{-0.05cm}
\begin{equation}\label{eq:repS3j}
	S_{3j} = \bigcup_{m=1}^{N_{2j}}[\xi_{k_m, \mathrm{l}}, \xi_{k_m, \mathrm{r}}].
\end{equation}
Here $\{\xi_{k_{N_{2j}}, \mathrm{l}}, \xi_{k_{N_{2j}}, \mathrm{r}}\}$ labels the endpoints of the rightmost interval, even if there is a countable infinite number of such intervals. Furthermore, to simplify notation, set $\xi_{k_0, \mathrm{r}} = \xi_{3j}$ and $z_{k_0} =x_{2j}$. We can then define $f_{3j}\!: [\xi_{3j}, \xi_{3j+3}) \rightarrow [\xi_{3j}, \xi_{3j+3}) $ as follows 
\begin{align}\label{eq:locInv}
	f_{3j}(\xi) &=\! \begin{cases}
	\xi + \sum_{n=1}^{N_{2j}} \left(\xi_{k_n, \mathrm{r}} - \xi_{k_n, \mathrm{l}} \right)\!, & \xi  \! \in \! [\xi_{3j}, \xi_{k_1, \mathrm{l}}), \\ 
	\! \!\xi + \sum_{n=m}^{N_{2j}} \left(\xi_{k_n, \mathrm{r}} - \xi_{k_n, \mathrm{l}} \right)\!, & \xi  \! \in \! (\xi_{k_{m-1}, \mathrm{r}}, \xi_{k_m, \mathrm{l}}), \hspace{0.25cm} m\geq 2,  \\
	\xi - \sum_{n=1}^m \left( \xi_{k_n, \mathrm{l}} - \xi_{k_{n-1}, \mathrm{r}} \right)\!, & \xi \! \in \![\xi_{k_m, \mathrm{l}}, \xi_{k_m, \mathrm{r}}], \hspace{0.65cm} m \geq 1, \\
	\xi, & \xi \! \in\! (\xi_{k_{N_{2j}}, \mathrm{r}}, \xi_{3j+3}), 
	\end{cases}
\end{align}
for $m \in \{1, ..., N_{2j}\}$. 

After gluing together these locally defined functions,  one obtains the global function $f\!:\R \rightarrow \R$ given by 
\begin{align}\label{eq:map_f}
	f(\xi) &= \begin{cases} 
		\vdots \\
		f_{3j-3}(\xi), & \xi \in [\xi_{3j-3}, \xi_{3j}), \\
		f_{3j}(\xi), & \xi \in [\xi_{3j}, \xi_{3j+3}). \\ 
		\vdots \end{cases}
\end{align}

Figure~\ref{fig:Illustrationf} illustrates the behavior of $f_{3j}$ when applied to $\bar{y}_{\Dx}$ and $\bar{V}_{\Dx, \xi}$; the points in $[\xi_{3j}, \xi_{3j+3})$ are rearranged in such a way that $\bar{y}_{\Dx}(f_{3j})$ is piecewise constant on $S_{3j}$, where $\bar{y}$ is constant. As a consequence, it holds that  $\bar{V}_{\xi}(\xi) - \bar{V}_{\Dx, \xi}(f_{3j}(\xi)) = 0$ for a.e. $\xi \in \mathcal{S}_{3j}$, cf. \eqref{eq:vanishing}, but this comes at the cost of losing the continuity of $\bar{y}_{\Dx}$ and $\bar{V}_{\Dx}$. In spite of that, the mapping $f$,  when applied to $\bar{X}_{\Dx, \xi} = (\bar{y}_{\Dx, \xi}, \bar{U}_{\Dx, \xi}, \bar{V}_{\Dx, \xi})$, still preserves much of the underlying structure of $\F$ from Definition~\ref{def:LagSet}. 

\begin{figure}
	\includegraphics{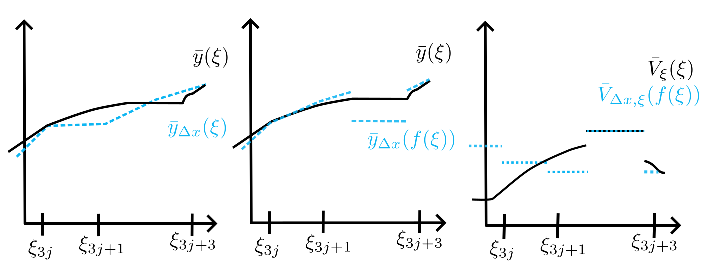}
	\captionsetup{width=.925\linewidth}
	\vspace{-0.3cm}
	\caption{The mapping $f$ rearranges the points within $[\xi_{3j}, \xi_{3j+3})$ in such a way that $\bar{y}_{\Dx}(f)$ is piecewise constant on $\mathcal{S}_{3j}$. As a consequence, $\bar{V}_{\xi}(\xi)\! -\! \bar{V}_{\Dx, \xi}(f(\xi))\!=\!0$ for a.e. $\xi \in\! \mathcal{S}_{3j}$.}
	\label{fig:Illustrationf}
\end{figure}

\begin{lemma}\label{lem:groupAction}
	Given $(\bar{u}, \bar{F}) \! \in \D$, let $(\bar{u}_{\Dx}, \bar{F}_{\Dx}) = P_{\Dx}((\bar{u}, \bar{F}))$ and \newline $\bar{X}_{\Dx} = L((\bar{u}_{\Dx}, \bar{F}_{\Dx}))$. Moreover, denote by $f : \!\R \rightarrow \! \R$ the mapping defined in \eqref{eq:map_f} and let $(\widehat{y}_{\Dx, \xi}, \widehat{U}_{\Dx, \xi}, \widehat{V}_{\Dx, \xi}) = (\bar{y}_{\Dx, \xi}\! \circ f, \bar{U}_{\Dx, \xi}\! \circ f, \bar{V}_{\Dx, \xi} \!\circ f)$. Then the following holds.
	\begin{enumerate}[label=(\roman*)]
		\item $(\widehat{y}_{\Dx, \xi}-1, \widehat{U}_{\Dx, \xi}, \widehat{V}_{\Dx, \xi}) \in [L^{\infty}(\R) \cap L^2(\R)]^3$,
		\item $\widehat{y}_{\Dx, \xi} \!\geq \!0$, $\widehat{V}_{\Dx, \xi} \!\geq \!0$ a.e. and there exists $c\!>\! 0$ such that $\widehat{y}_{\Dx, \xi} + \widehat{V}_{\Dx, \xi} \!\!\geq \! c$ a.e.,
		\item $\widehat{y}_{\Dx, \xi} \widehat{V}_{\Dx, \xi} = \widehat{U}_{\Dx, \xi}^2 $ a.e.  
		\item for a.e. $\xi \in \mathcal{S}$,
		\begin{equation}\label{eq:iden_yxif}
			\widehat{y}_{\Dx, \xi}(\xi) = \frac{1}{1+\bar{u}_{\Dx, x}^2(\bar{y}_{\Dx}(f(\xi)))}. 
		\end{equation}
		\end{enumerate}
\end{lemma}

\begin{proof}
It turns out to be convenient to establish \emph{ii)} and \emph{iii)} before showing \emph{i)}. To this end, fix $j \in \mathbb{Z}$ and consider the interval $[\xi_{3j}, \xi_{3j+3})$.

 If $N_{2j} = 0$, then \emph{ii)} and \emph{iii)} hold over $[\xi_{3j}, \xi_{3j+3})\setminus \{\xi_{3j}, \xi_{3j+1}, \xi_{3j+2}\}$, because $f_{3j} = \id$ by \eqref{eq:f_no} and each component of $\bar{X}_{\Dx, \xi}$ is piecewise constant with possible jump discontinuities at $\{\xi_j\}_{j \in \mathbb{Z}}$.

If $N_{2j} \geq 1$ on the other hand, then $f_{3j}$ is given by \eqref{eq:locInv} and admits jump discontinuities at the points 
\begin{equation*}
	A_{3j} := \{\xi_{k_m \mathrm{l}}\}_{m=1}^{N_{2j}} \cup \{\xi_{k_m, \mathrm{r}}\}_{m=1}^{N_{2j}}. 
\end{equation*}
Thus, if we introduce
\begin{equation*}
	C_{3j} :=  \{f^{-1}(\xi_{3j}), f^{-1}(\xi_{3j+1}), f^{-1}(\xi_{3j+2})\},
\end{equation*}
which contains three points as $f$ is a bijection, then $\bar{X}_{\Dx, \xi}(f)$ satisfies \emph{ii)} and \emph{iii)} on $[\xi_{3j}, \xi_{3j+3})\setminus (A_{3j} \cup C_{3j})$, i.e., almost everywhere as $A_{3j} \cup C_{3j}$ is countable.

Since $f$ is a combination of the locally defined functions $\{f_{3j}\}_{j \in \mathbb{Z}}$ glued together at the Lagrangian gridpoints $\{\xi_{3j}\}_{j \in \mathbb{Z}}$, it follows that each component of \newline $\widehat{X}_{\Dx, \xi}=\bar{X}_{\Dx, \xi} \circ f$ is piecewise constant in the sense that it might admit jump discontinuities at the Lebesgue null set  
\begin{align*}
		\mathcal{C}:= \{\xi_{3j}\}_{j \in \mathbb{Z}} \cup \bigcup_{j \in \mathbb{Z}} \bigl(A_{3j } \cup C_{3j}\bigr), 
\end{align*}
and it is piecewise constant outside this set. Here we write $A_{3j} = \emptyset$ whenever $N_{2j}=0$. Thus, due to \eqref{eq:MapL} and the fact that $\bar{X}_{\Dx}\! \!\in \F$, we have 
\begin{equation}\label{eq:nonnegativity}
		\widehat{y}_{\Dx, \xi}(\xi), \widehat{V}_{\Dx, \xi}(\xi) \geq 0 \hspace{0.8cm} \xi \in \mathcal{C}^c, 
\end{equation}
and, furthermore
\begin{subequations}\label{eq:transformId}
\begin{align}
	0 < (\widehat{y}_{\Dx, \xi} + \widehat{V}_{\Dx, \xi})(\xi) &= 1 \hspace{1.9cm} \xi \in \mathcal{C}^c,  \label{eq:id1}\\
	\widehat{y}_{\Dx, \xi} \widehat{V}_{\Dx, \xi}(\xi) &= \widehat{U}_{\Dx, \xi}^2(\xi)\hspace{0.8cm} \xi \in \mathcal{C}^c. \label{eq:id2}
\end{align}
\end{subequations}
This proves \emph{ii)} and \emph{iii)}, because $\mathcal{C}^c$ is of full measure. Moreover,  \emph{iv)} follows by \eqref{eq:rep_y} and the fact that $f$ is bijective and satisfies $f^{-1}(\mathcal{S}_{\Dx}) = \mathcal{S} \setminus (\bigcup_{j \in \mathbb{Z}}A_{3j})$. 

Next we show that $\widehat{V}_{\Dx, \xi} \in L^1(\R)$. To this end, note that  
\begin{equation*}
	\bar{F}_{\infty} = \int_{\R}\bar{V}_{\Dx, \xi}(\xi)d\xi  = 
	\sum_{j \in \mathbb{Z}} \int_{\xi_{3j}}^{\xi_{3j+3}} \bar{V}_{\Dx, \xi}(\xi)d\xi,
\end{equation*}
and as each term is nonnegative, this series is {\em absolutely convergent}. As a consequence, since $f$ maps $[\xi_{3j}, \xi_{3j+3})$ onto itself and $\widehat{V}_{\Dx, \xi}$ is nonnegative by \eqref{eq:nonnegativity}, it suffices to show, for any $j \in \mathbb{Z}$, that  
\begin{equation}\label{eq:energyPreserved}
	\int_{\xi_{3j}}^{\xi_{3j+3}} \bar{V}_{\Dx, \xi}(\xi)d\xi = \int_{\xi_{3j}}^{\xi_{3j+3}} \widehat{V}_{\Dx, \xi}(\xi)d\xi.  
\end{equation}
First, note that \eqref{eq:energyPreserved} is an immediate consequence of \eqref{eq:f_no} for $N_{2j}=0$. We therefore assume that $N_{2j}\!  \geq 1$, for which we obtain, after using \eqref{eq:locInv}, 
\begin{align}\label{eq:decompVxi}
	\int_{\xi_{3j}}^{\xi_{3j+3}}\bar{V}_{\Dx, \xi}(\xi)d\xi &= \int_{\xi_{3j}}^{\xi_{3j+1}}\bar{V}_{\Dx, \xi}(\xi)d\xi  + \int_{\xi_{3j+1}}^{\xi_{k_{N_{2j}}\!, \mathrm{r}}} \bar{V}_{\Dx, \xi}(\xi)d\xi \nonumber
	\\ & \qquad + \int_{\xi_{k_{N_{2j}}\!, \mathrm{r}}}^{\xi_{3j+3}} \bar{V}_{\Dx, \xi}(f(\xi))d\xi. 
\end{align}

Recall that $\bar{\mu}(\{z_{k_n}\}) = \bar{F}(z_{k_n}+) - \bar{F}(z_{k_n})$,  $\xi_{k_0, \mathrm{r}}= \xi_{3j}$, and let $\lambda_m := \sum_{n=1}^m \bar{\mu}(\{z_{k_n}\})$ with $\lambda_0=0$. Due to \eqref{eq:explicitS_dx}--\eqref{eq:singContrib} we have $\xi_{3j+1} = \xi_{3j} + \lambda_{N_{2j}}$. Furthermore, using the notation from \eqref{eq:discPointLagr}, the first term in \eqref{eq:decompVxi} takes the form  
\begin{align}\label{eq:fSecond}
	 \int_{\xi_{3j}}^{\xi_{3j+1}}\bar{V}_{\Dx, \xi}(\xi)d\xi  &= \sum_{m=1}^{N_{2j}} \int_{\xi_{3j} + \lambda_{m-1}}^{\xi_{3j} + \lambda _m} \bar{V}_{\Dx, \xi}(\xi)d\xi \nonumber
	 \\ &= \sum_{m=1}^{N_{2j}}\int_{\xi_{k_m\!, \mathrm{l}} - \sum_{n=1}^{m}\!(\xi_{k_n\!, \mathrm{l}} - \xi_{k_{n-1}\!, \mathrm{r}})}^{\xi_{k_{m\!, \mathrm{r}}} - \sum_{n=1}^{m}\!(\xi_{k_n\!, \mathrm{l}} - \xi_{k_{n-1}\!, \mathrm{r}})}\!\!\bar{V}_{\Dx, \xi}(\xi)d\xi \nonumber
	 \\ &= \sum_{m=1}^{N_{2j}}\int_{\xi_{k_m\!, \mathrm{l}}}^{\xi_{k_m\!, \mathrm{r}}}\! \!\bar{V}_{\Dx, \xi}\bigg(\xi - \sum_{n=1}^m(\xi_{k_n\!, \mathrm{l}}-\xi_{k_{n-1}\!, \mathrm{r}})\!\bigg)d\xi \nonumber
	 \\ &= \sum_{m=1}^{N_{2j}}\int_{\xi_{k_m\!, \mathrm{l}}}^{\xi_{k_m\!, \mathrm{r}}}\! \!\bar{V}_{\Dx, \xi}(f(\xi))d\xi.  
\end{align}
 
 By proceeding similarly for the second term in \eqref{eq:decompVxi}, one obtains 
 \begin{align}\label{eq:fLast}
 	\int_{\xi_{3j+1}}^{\xi_{k_{N_{2j}}\!, \mathrm{r}}} \bar{V}_{\Dx, \xi}(\xi)d\xi &= \sum_{m=1}^{N_{2j}} \nonumber
 \int_{\xi_{k_{m-1}\!, \mathrm{r}} + \lambda_{N_{2j}} - \lambda_{m-1}}^{\xi_{k_m\!, \mathrm{r}} +  \lambda_{N_{2j}} - \lambda_m}\!\!\bar{V}_{\Delta x, \xi}(\xi)d\xi \nonumber 
 \\ &= \sum_{m=1}^{N_{2j}} \int_{\xi_{k_{m-1}\!, \mathrm{r}}}^{\xi_{k_m\!, \mathrm{l}}} \bar{V}_{\Dx, \xi}(f(\xi))d\xi. 
\end{align}
Combining \eqref{eq:decompVxi}, \eqref{eq:fSecond},  and \eqref{eq:fLast} gives \eqref{eq:energyPreserved}, which in turn implies $\|\widehat{V}_{\Dx, \xi}\|_1 = \bar{F}_{\infty}$.
Moreover, \eqref{eq:nonnegativity} and \eqref{eq:id1} imply $\bar{V}_{\Dx, \xi}(f(\xi)) \leq 1$ for $\xi \in \mathcal{C}^c$, hence  
\begin{equation}\label{eq:argVhat}
	\|\widehat{V}_{\Dx, \xi}\|_2^2 \leq \int_{\R} |\widehat{V}_{\Dx, \xi}(\xi)|d\xi = \bar{F}_{\infty}. 
\end{equation}
Note that $\widehat{y}_{\Dx, \xi} - 1 \in L^2(\R) \cap L^{\infty}(\R)$ is an immediate consequence of \eqref{eq:nonnegativity} and \eqref{eq:id1}. At last, applying  \eqref{eq:id2} yields 
\begin{equation*}
	\|\widehat{U}_{\Dx, \xi}\|_{\infty} = \|\sqrt{\widehat{V}_{\Dx, \xi}\widehat{y}_{\Dx, \xi}}\|_{\infty} \leq \frac{1}{2}\|\widehat{V}_{\Dx, \xi}\|_{\infty} + \frac{1}{2}\|\widehat{y}_{\Dx, \xi}\|_{\infty} \leq 1, 
\end{equation*}
and, moreover
\begin{equation*}
	\int_{\R} \widehat{U}_{\Dx,\xi}^2(\xi)d\xi = \int_{\R} \widehat{V}_{\Dx, \xi}\widehat{y}_{\Dx, \xi}(\xi)d\xi \leq \int_{\R} \widehat{V}_{\Dx, \xi}(\xi)d\xi = \bar{F}_{\infty}.
\end{equation*}
This establishes \emph{i)} and hence the asserted statement. 
\end{proof}

\subsection{A priori estimate for $\boldsymbol{\mu}_{\mathrm{sing}}=\boldsymbol{\mu}_{\mathrm{d}}$}
 
The following error estimate is fundamental for establishing a convergence rate in Lagrangian coordinates.

\begin{theorem}\label{thm:Discrepancy}
	Given $(\bar{u}, \bar{\mu}) \in \D$ with $\bar{\mu}_{\mathrm{sc}} = 0$, let $\bar{X}=L \left((\bar{u}, \bar{\mu})  \right)$ and \newline $\bar{X}_{\Dx} = L \circ P_{\Dx} \left((\bar{u}, \bar{\mu}) \right)$.   Moreover, for any $t\geq 0$, introduce $X(t) = S_t(\bar{X})$ and $X_{\Delta x}(t) = S_{t}(\bar{X}_{\Dx})$, then 
	\begin{align} \label{eq:bound_U}
		\|U(t) - U_{\Dx}(t) \|_{\infty} &\leq(1 + 2\sqrt{2})\bar{F}_{\mathrm{ac}, \infty}^{\nicefrac{1}{2}}\Dx^{\nicefrac{1}{2}} + t\Dx  + t\|\delta_{2\Dx}\bar{u}_{x}^2\|_1 \nonumber 
		\\ & \qquad + \frac{1}{4}(1+\alpha)t \|\bar{V}_{\xi} - \bar{V}_{\Dx, \xi}(f)\|_{L^1(\mathcal{S}^c)} \nonumber
		 \\ & \qquad + \frac{1}{4}t \|g(\bar{X}) - g(\bar{X}_{\Dx})(f)  \|_{L^1(\mathcal{S}^c)} 
		 \\ & \qquad + \frac{\alpha}{2} \sqrt{2 \Big(1 + \frac{1}{4}t^2 \Big)\bar{F}_{\infty}}\|\bar{U}_{\xi} - \bar{U}_{\Dx, \xi}(f)\|_{L^2(\mathcal{S}^c)},  \nonumber
	\end{align}
	where $f$ is the function defined in \eqref{eq:map_f}. If $\bar{\mu}\! = \bar{\mu}_{\mathrm{ac}}$, then $f = \id$ and $\mathcal{S}^c$ is of full measure. 
\end{theorem}

\begin{remark}
	The reason we require the singular continuous part of $\bar{\mu}$ to be zero can be explained as follows. Define the generalized inverses  
	\begin{equation*}
		\mathcal{Z}_1(x) := \inf \{\xi \in \R: \bar{y}(\xi) = x\} \hspace{0.5cm} \text{ and } \hspace{0.5cm} 
		\mathcal{Z}_2(x):= \sup \{\xi \in \R: \bar{y}(\xi) = x \}, 
	\end{equation*}
	then $\mathcal{S}=\mathcal{S}_{\mathrm{d}} \cup \mathcal{S}_{\mathrm{sc}}$, where 
	\begin{align*}
		\mathcal{S}_{\mathrm{d}} &= \{\xi \in \mathcal{S}: \mathcal{Z}_1(\bar{y}(\xi)) < \mathcal{Z}_2(\bar{y}(\xi))\}, \\
		\mathcal{S}_{\mathrm{sc}} &= \{\xi \in \mathcal{S}: \mathcal{Z}_1(\bar{y}(\xi)) = \mathcal{Z}_2(\bar{y}(\xi)) \}. 
	\end{align*}
	Furthermore, by \cite[Prop. 2.1]{RegularityStructure}, we have
	\begin{equation*}
		\bar{\mu}_{\mathrm{d}} = \bar{\mu}\vert_{\bar{y}(\mathcal{S}_{\mathrm{d}})} \hspace{0.5cm} \text{and} \hspace{0.5cm} \bar{\mu}_{\mathrm{sc}} = \bar{\mu}\vert_{\bar{y}(\mathcal{S}_{\mathrm{sc}})}. 
	\end{equation*}
	Thus, $\bar{\mu}_{\mathrm{d}}$ can be associated with intervals in $\mathcal{S}_{\mathrm{d}}$, while $\bar{\mu}_{\mathrm{sc}}$ is associated with the totally disconnected set $\mathcal{S}_{\mathrm{sc}}$. The case $\bar{\mu}_{\mathrm{sc}} \neq 0$ would therefore lead to two major issues: \emph{i)} we cannot construct $f$ explicitly and \emph{ii)} a function $f$ satisfying $f^{-1}(\mathcal{S}_{\Dx}) = \mathcal{S}\setminus E$, where $|E| =0$, need not be measurable as the support of $\bar{\mu}_{\mathrm{sc}}$ is uncountable. 
\end{remark}

\begin{proof}
See Section~\ref{sec:noSingular} for the case $\bar{\mu}_{\mathrm{sing}}\! = 0$. 

Therefore, let $\bar{\mu}_{\mathrm{sing}} = \bar{\mu}_{\mathrm{d}}$. Thanks to Lemma~\ref{lem:convInt}, it  only remains to estimate the second term in \eqref{eq:estU}. To this end, note that for any $\xi \in \R$ there is $N \in \mathbb{Z}$ such that $\xi \in [\xi_{3N}, \xi_{3N+3})$ and hence 
\begin{subequations}
\begin{align}
    \biggl| \int_0^t \biggl( \int_{-\infty}^{\xi}&\left(V_{\xi}(s, \eta)-V_{\Dx, \xi}(s, \eta) \right)d\eta  - \int_{\xi}^{\infty} \left(V_{\xi}(s, \eta) - V_{\Dx, \xi}(s, \eta) \right)d\eta \biggr)ds \biggr| \nonumber 
    \\ &\leq \biggl| \int_0^t \int_{-\infty}^{\xi_{3N}} (V_{\xi}(s, \eta)-V_{\Dx, \xi}(s, \eta))d\eta ds \biggr | \label{eq:twoFirst1}
    \\ & \quad + \left | \int_0^t \int_{\xi_{3N+3}}^{\infty}(V_{\xi}(s, \eta)-V_{\Dx, \xi}(s, \eta))d\eta ds \right| \label{eq:twoFirst2} 
    \\ & \quad + \int_0^t \int_{\xi_{3N}}^{\xi_{3N+3}} \big|V_{\xi}(s, \eta) - V_{\Dx, \xi}(s, \eta) \big|d\eta ds. \label{eq:twoLast}
 \end{align}
 \end{subequations}

In order to estimate \eqref{eq:twoFirst1} and \eqref{eq:twoFirst2}, we use the mapping $f$ defined in \eqref{eq:map_f}. In particular, by \eqref{eq:energyPreserved}, we have 
\begin{equation}\label{eq:change1}
	\int_{-\infty}^{\xi_{3N}}V_{\Dx, \xi}(t, \eta)d\eta =\int_{-\infty}^{\xi_{3N}} V_{\Dx, \xi}(t, f(\eta))d\eta,
\end{equation} 
and 
\begin{equation}\label{eq:change2}
	\int_{\xi_{3N+3}}^{\infty}V_{\Dx, \xi}(t, \eta)d\eta = \int_{\xi_{3N+3}}^{\infty}V_{\Dx, \xi}(t, f(\eta))d\eta. 
\end{equation}
In addition, by differentiating \eqref{eq:LagrSystem} with respect to $\xi$ and subsequently solving the corresponding system of ODEs with initial data $(\bar{y}_{\Dx, \xi}(f), \bar{U}_{\Dx, \xi}(f), \bar{V}_{\Dx, \xi}(f))$, one obtains, for a.e. $\xi \in \R$, 
\begin{align}\label{eq:explicitExp}
	\widehat{y}_{\Dx, \xi}(t, \xi) &= \begin{cases}
	\bar{y}_{\Dx, \xi}(f(\xi)) + t \bar{U}_{\Dx, \xi}(f(\xi)) 
	\\ \quad + \frac{t^2}{4} \bar{V}_{\Dx, \xi}(f(\xi)), & t < \widehat{\tau}_{\Dx}(\xi) \text{ or} \widehat{\tau}_{\Dx}(\xi)  = 0,  \\
	 \frac{1}{4}(1-\alpha) (t-\widehat{\tau}_{\Dx}(\xi))^2\bar{V}_{\Dx, \xi}(f(\xi)), & 0 < \widehat{\tau}_{\Dx}(\xi) \leq t, 
	 \end{cases} \nonumber \\
	\widehat{U}_{\Dx, \xi}(t, \xi) &= \begin{cases}
	 \bar{U}_{\Dx, \xi}(f(\xi)) + \frac{t}{2} \bar{V}_{\Dx, \xi}(f(\xi)), & t < \widehat{\tau}_{\Dx}(\xi) \text{ or} \widehat{\tau}_{\Dx}(\xi)  = 0,  \\
	 \frac{1}{2}(1-\alpha)(t-\widehat{\tau}_{\Dx}(\xi))\bar{V}_{\Dx, \xi}(f(\xi)), & 0 < \widehat{\tau}_{\Dx}(\xi) \leq t, 
	 \end{cases} \nonumber\\ 
	\widehat{V}_{\Dx, \xi}(t, \xi) &= \begin{cases}
	\bar{V}_{\Dx, \xi}(f(\xi)), & t < \widehat{\tau}_{\Dx}(\xi)  \text{ or } \widehat{\tau}_{\Dx}(\xi)  = 0,  \\
	(1-\alpha)\bar{V}_{\Dx, \xi}(f(\xi)), & 0 < \widehat{\tau}_{\Dx}(\xi)  \leq t, 
	\end{cases} \raisetag{-35pt}
\end{align}
\vspace{0.1cm}
where
\begin{align*}
	\widehat{\tau}_{\Dx}(\xi) &= 
	\begin{cases}
	0, & \xi \in \mathcal{S}, \\
	\tau(f(\xi)), & \text{otherwise}, 
	\end{cases}
\end{align*}
and $\widehat{X}_{\Dx, \xi} = (\widehat{y}_{\Dx, \xi}, \widehat{U}_{\Dx, \xi}, \widehat{V}_{\Dx, \xi})(t, \xi) = (y_{\Dx, \xi}, U_{\Dx, \xi}, V_{\Dx, \xi})(t, f(\xi))$. The linear structure of \eqref{eq:explicitExp} implies that $\widehat{X}_{\Dx, \xi}$ preserves \emph{i)} and \emph{iii)} of Lemma~\ref{lem:groupAction}, and \emph{ii)} with $\widehat{y}_{\Dx, \xi}(t) + \widehat{V}_{\Dx, \xi}(0) > c$ almost everywhere, for some $c > 0$.  Moreover, due to  \eqref{eq:LagrSystem3} and  \eqref{eq:transformId} we have, for all $t\geq 0$, 
\begin{align}\label{eq:evolBreak}
\begin{aligned}
    V_{\xi}(t, \xi) &= \bar{V}_{\xi}(\xi) = 1, \hspace{1.4cm} \text{ a.e. }\xi \in \mathcal{S} , \\
    \widehat{V}_{\Dx, \xi}(t, \xi) &= \bar{V}_{\Dx, \xi}(f(\xi)) = 1, \hspace{0.5cm} \text{ a.e. }\xi \in \mathcal{S}.
\end{aligned}
\end{align}
Consequently, combining \eqref{eq:change1}, \eqref{eq:change2}, and \eqref{eq:evolBreak} we get
\begin{align}\label{eq:transformedVxi}
	 \biggl |\int_0^t \int_{-\infty}^{\xi_{3N}} \!\!\bigl(V_{\xi}(s, \eta)&-V_{\Dx, \xi}(s, \eta)\bigr)d\eta ds \biggr | \! \!+ 	\biggl |\int_0^t \int_{\xi_{3N+3}}^{\infty}\!\! \bigl(V_{\xi}(s, \eta)-V_{\Dx, \xi}(s, \eta) \bigr)d\eta ds \biggr| 		\nonumber
    \\ & \leq \int_0^t \int_{\mathcal{S}^c} \left|V_{\xi}(s, \eta) - \widehat{V}_{\Dx, \xi}(s, \eta) \right|d\eta ds. 
\end{align}
To estimate this term further, we follow the argument in the proof of \cite[Thm. 4.9]{AlphaAlgorithm}. The details are included here for completeness. Firstly, in order to ease notation, we introduce, based on the sets from \eqref{eq:omegas}, 
\begin{equation}\label{eq:setNotation}
	\Omega_{m, n}(t) := \Omega_m(X(t)) \cap \Omega_n(\widehat{X}_{\Dx} (t)), \hspace{0.35cm} m, n \in \{c, d\}, 
\end{equation}
where $\widehat{X}_{\Dx}(t, \xi) = (y_{\Dx}, U_{\Dx}, V_{\Dx})(t, f(\xi))$. These sets form a disjoint partition of $\R$ for any $t \geq 0$. In addition let $\bar{\Omega}_{m, n} = \mathcal{S}^c \cap \Omega_{m, n }(0)$. 

Since, for any finite $T \geq 0$, $V_{\xi} \in L^1([0, T] \times \R)$ and $\widehat{V}_{\Dx, \xi} \in L^1([0, T] \times \R)$, where the latter follows by \eqref{eq:explicitExp} and the proof of Lemma~\ref{lem:groupAction}, we can apply Fubini's theorem to write
\begin{subequations}
\begin{align}
	\int_0^t \int_{\mathcal{S}^c} \left|V_{\xi}(s, \eta) - \widehat{V}_{\Dx, \xi}(s, \eta) \right|d\eta ds &= \int_{\bar{\Omega}_{c, c}} \int_0^t \big|V_{\xi}(s, \eta) - \widehat{V}_{\Dx, \xi}(s, \eta) \big|ds d\eta  \nonumber
	\\ & \quad +  \int_{\bar{\Omega}_{c, d}} \int_0^t \big|V_{\xi}(s, \eta) - \widehat{V}_{\Dx, \xi}(s, \eta) \big|ds d\eta  \label{eq:term2}
	\\ & \quad + \int_{\bar{\Omega}_{d, c}} \int_0^t \big|V_{\xi}(s, \eta) - \widehat{V}_{\Dx, \xi}(s, \eta) \big|ds d\eta  \label{eq:term3}
	\\ & \quad + \int_{\bar{\Omega}_{d, d}} \int_0^t \big|V_{\xi}(s, \eta) - \widehat{V}_{\Dx, \xi}(s, \eta) \big|ds d\eta. \label{eq:term4}
\end{align}
\end{subequations}
Recall that $\bar{\Omega}_{c, c}$ consists of the points for where neither $X(t)$ nor $\widehat{X}_{\Dx}(t)$ experience wave breaking, therefore by \eqref{eq:LagrSystem} and \eqref{eq:explicitExp}, 
\begin{equation}\label{eq:simplest}
	\int_{\bar{\Omega}_{c, c}} \int_0^t \big|V_{\xi}(s, \eta) - \widehat{V}_{\Dx, \xi}(s, \eta) \big|ds d\eta   = \int_{\bar{\Omega}_{c, c}} \int_0^t \big|\bar{V}_{\xi}(\eta) - \bar{V}_{\Dx, \xi}(f(\eta)) \big|ds d\eta.
\end{equation}
The terms \eqref{eq:term3} and \eqref{eq:term4} can be estimated in the same way, hence we just focus on  \eqref{eq:term3}. Any $\xi \in \Omega_d(\widehat{X}_{\Dx}(0))$ can enter $\Omega_c(\widehat{X}_{\Dx}(s))$ at some time $s \in (0, t]$, in particular, this happens at $s=\widehat{\tau}_{\Dx}(\xi)$, cf. \eqref{eq:explicitExp}. Thus, recalling \eqref{eq:function_g}, we find 
\begin{align}\label{eq:est2}	
	\int_{\bar{\Omega}_{c, d}} \int_0^t \big|V_{\xi}(s, \eta) &- \widehat{V}_{\Dx, \xi}(s, \eta) \big|ds d\eta \nonumber
	\\ &= \int_{\bar{\Omega}_{c, d}\cap \Omega_{c, d}(t)} \int_0^t \big|\bar{V}_{\xi}(\eta) - \bar{V}_{\Dx, \xi}(f(\eta)) \big|ds d\eta \nonumber
	\\ & \qquad + \int_{\bar{\Omega}_{c, d}\cap \Omega_{c, c}(t)}\int_0^{\widehat{\tau}_{\Dx}(\eta)}\big|\bar{V}_{\xi}(\eta) - \bar{V}_{\Dx, \xi}(f(\eta)) \big|ds d\eta \nonumber
	\\ & \qquad + \int_{\bar{\Omega}_{c, d}\cap \Omega_{c, c}(t)}\int_{\widehat{\tau}_{\Dx}(\eta)}^t \big|\bar{V}_{\xi}(\eta) - (1-\alpha)\bar{V}_{\Dx, \xi}(f(\eta)) \big|ds d\eta \nonumber
	\\ & \leq \int_{\bar{\Omega}_{c, d}} \int_0^t \big|\bar{V}_{\xi}(\eta) - \bar{V}_{\Dx, \xi}(f(\eta)) \big|ds d\eta \nonumber
	\\ & \qquad + \int_{\bar{\Omega}_{c, d}} \int_0^t \big| g(\bar{X})(\eta) - g(\bar{X}_{\Dx})(f(\eta))\big|ds d\eta. 
\end{align}

As both $X(t)$ and $\widehat{X}_{\Dx}(t)$ can experience wave breaking in $\bar{\Omega}_{d, d}$, it is natural to decompose \eqref{eq:term4} further based on the sets from \eqref{eq:setNotation} at time $t$. This yields  
\begin{subequations}
 \begin{align}
 	\int_{\bar{\Omega}_{d, d}} \int_0^t \big|V_{\xi}(s, \eta) &- \widehat{V}_{\Dx, \xi}(s, \eta) \big|ds d\eta \nonumber
	\\ &= \int_{\bar{\Omega}_{d, d} \cap \Omega_{d, d}(t)} \int_0^t \big|V_{\xi}(s, \eta) - \widehat{V}_{\Dx, \xi}(s, \eta) \big|ds d\eta \label{eq:term20}
	\\ & \qquad +\int_{\bar{\Omega}_{d, d} \cap \Omega_{d, c}(t)} \int_0^t \big|V_{\xi}(s, \eta) - \widehat{V}_{\Dx, \xi}(s, \eta) \big|ds d\eta \label{eq:term21}
	\\ & \qquad +   \int_{\bar{\Omega}_{d, d} \cap \Omega_{c, d}(t)} \int_0^t \big|V_{\xi}(s, \eta) - \widehat{V}_{\Dx, \xi}(s, \eta) \big|ds d\eta \label{eq:term22}
 	\\ & \qquad + \int_{\bar{\Omega}_{d, d} \cap \Omega_{c, c}(t)} \int_0^t \big|V_{\xi}(s, \eta) - \widehat{V}_{\Dx, \xi}(s, \eta) \big|ds d\eta.  \label{eq:term23}
\end{align}
  \end{subequations}  
  No wave breaking takes place in \eqref{eq:term20}, such that one can proceed as in \eqref{eq:simplest}. Moreover, the structures of \eqref{eq:term21} and \eqref{eq:term22} are similar, we therefore only present details on how to estimate \eqref{eq:term21}. In particular, 
 \begin{align*}
 	\int_{\bar{\Omega}_{d, d} \cap \Omega_{d, c}(t)} \int_0^t \big|V_{\xi}&(s, \eta) - \widehat{V}_{\Dx, \xi}(s, \eta) \big|ds d\eta 
	\\ & \leq \int_{\bar{\Omega}_{d, d} \cap \Omega_{d, c}(t)} \int_0^{\widehat{\tau}_{\Dx}(\eta)}\big|\bar{V}_{\xi}(\eta) - \bar{V}_{\Dx, \xi}(f(\eta)) \big|ds d\eta 
	\\ & \qquad + \int_{\bar{\Omega}_{d, d} \cap \Omega_{d, c}(t)} \int_{\widehat{\tau}_{\Dx}(\eta)}^t \big|\bar{V}_{\xi}(\eta) - (1-\alpha)\bar{V}_{\Dx, \xi}(f(\eta)) \big|ds d\eta 
	\\ & \leq  \int_{\bar{\Omega}_{d, d} \cap \Omega_{d, c}(t)} \int_0^t \big|\bar{V}_{\xi}(\eta) - \bar{V}_{\Dx, \xi}(f(\eta)) \big|ds d\eta
	\\ & \qquad + \int_{\bar{\Omega}_{d, d} \cap \Omega_{d, c}(t)} \int_{\widehat{\tau}_{\Dx}(\eta)}^t \alpha \bar{V}_{\xi}(\eta)ds d\eta. 
 \end{align*}
  Furthermore, note that $U_{\xi}(\cdot, \xi)$ and $\widehat{U}_{\Dx, \xi}(\cdot, \xi)$ are continuous and increasing functions which vanish at $t=\tau(\xi)$ and $t=\widehat{\tau}_{\Dx}(\xi)$, respectively. Hence,  by \eqref{eq:LagrSystem} and \eqref{eq:explicitExp}, we have 
  \begin{align} \label{eq:oneBreaksBefore}
  	\int_{\bar{\Omega}_{d, d} \cap \Omega_{d, c}(t)} \!\int_{\widehat{\tau}_{\Dx}(\eta)}^t \! \! \! \!\alpha \bar{V}_{\xi}(\eta)ds d\eta &= 2\alpha \int_{\bar{\Omega}_{d, d} \cap \Omega_{d, c}(t)}\big(U_{\xi}(t, \eta) - U_{\xi}( \widehat{\tau}_{\Dx}(\eta), \eta)\big)d\eta \nonumber
	\\ &\leq \! 2\alpha \! \int_{\bar{\Omega}_{d, d} \cap \Omega_{d, c}(t)} \! \! \!\big(\widehat{U}_{\Dx, \xi}(\widehat{\tau}_{\Dx}(\eta), \eta) - U_{\xi}(\widehat{\tau}_{\Dx}(\eta), \eta)\big)d\eta \nonumber
	\\ &\leq \!\alpha \! \int_{\bar{\Omega}_{d, d} \cap \Omega_{d, c}(t)} \int_0^{\widehat{\tau}_{\Dx}(\eta)} \! \! \!\big|\bar{V}_{\xi}(\eta) - \bar{V}_{\Dx, \xi}(f(\eta)) \big| ds d\eta \nonumber
	\\ & \qquad + 2\alpha \int_{\bar{\Omega}_{d, d} \cap \Omega_{d, c}(t)} \big|\bar{U}_{\xi}(\eta) - \bar{U}_{\Dx, \xi}(f(\eta))\big|d\eta.
\end{align} 
 Finally, wave breaking occurs for both $X$ and $\widehat{X}_{\Dx}$ on $\bar{\Omega}_{d, d} \cap \Omega_{c, c}(t)$, such that \eqref{eq:term23} can be decomposed into two parts; one where $X$ breaks before $\widehat{X}_{\Dx}$ and one part where the opposite holds. This enables us to proceed similarly to how we did for \eqref{eq:term21}, we therefore omit the details. After combining all the derived estimates, one ends up with 
 \begin{align*}
 	\int_0^t \int_{\mathcal{S}^c} \left|V_{\xi}(s, \eta) - \widehat{V}_{\Dx, \xi}(s, \eta) \right|d\eta ds &\leq (1+\alpha)\int_{\mathcal{S}^c} \int_0^t\big|\bar{V}_{\xi}(\eta) - \bar{V}_{\Dx, \xi}(f(\eta)) \big| ds d\eta
	\\ & \qquad +  \int_{\mathcal{S}^c} \int_0^t \big| g(\bar{X}) - g(\bar{X}_{\Dx}(f))\big|ds d\eta
	\\ & \qquad + \int_{\mathcal{S}^c \cap D(t)} \big|\bar{U}_{\xi}(\eta) - \bar{U}_{\Dx, \xi}(f(\eta)) \big|d\eta,  
 \end{align*}
 where 
 \begin{equation*}
 	D(t) := \{\xi: 0 < \tau(\xi) \leq t \} \bigcup \{\xi: 0 < \widehat{\tau}_{\Dx}(\xi) \leq t \}. 
\end{equation*}
Moreover, \cite[Cor. 2.4]{AlphaHS} ensures that $|\{\xi: 0 < \tau(\xi) \leq t \}| \leq (1+\frac{1}{4}t^2)\bar{F}_{\infty}$, and thanks to \eqref{eq:transformId} and \eqref{eq:explicitExp}, it is possible to prove that $|\{\xi: 0 < \widehat{\tau}_{\Dx}(\xi) \leq t \}| \leq (1+\frac{1}{4}t^2)\bar{F}_{\infty}$ by mimicking the proof of  \cite[Cor. 2.4]{AlphaHS}. As a consequence, by applying the CS-inequality we obtain  
 \begin{align}
 \begin{aligned} \label{eq:absolutelyContrib}
   	\int_0^t \int_{\mathcal{S}^c} |V_{\xi}(s, \eta) &- \widehat{V}_{\Dx, \xi}(s, f(\eta))|d\eta ds
        \\ & \leq (1 + \alpha)t \|\bar{V}_{\xi}- \bar{V}_{\Dx, \xi}(f)\|_{L^1(\mathcal{S}^c)} \\ & \qquad + t \|g(\bar{X}) - g(\bar{X}_{\Dx})(f) \|_{L^1(\mathcal{S}^c)} 
    \\ & \qquad + 2\alpha \sqrt{2 \Big(1 + \frac{1}{4}t^2 \Big)\bar{F}_{\infty}} \|\bar{U}_{\xi}- \bar{U}_{\Dx, \xi}(f) \|_{L^2(\mathcal{S}^c)}. \raisetag{-65pt} 
 \end{aligned}
 \end{align}
At long last, it only remains to estimate \eqref{eq:twoLast}. The idea is to decompose the integration domain based on the sets $\mathcal{S}$ and $\mathcal{S}_{\Dx}$ from \eqref{eq:S} and \eqref{eq:num_S}, respectively. In this regard, recall the notation from \eqref{eq:S3j} and set
\begin{equation}\label{eq:complementS3j}
	\mathcal{S}^c_{3N} := \mathcal{S}^c \cap [\xi_{3N}, \xi_{3N+3}) \hspace{0.35cm} \text{and} \hspace{0.35cm} \mathcal{S}^c_{\Dx, 3N} := \mathcal{S}_{\Dx}^c \cap [\xi_{3N}, \xi_{3N+3}).
\end{equation}
 Combining \eqref{eq:evolBreak} with the fact that $V_{\Dx, \xi}(t, \xi)\! =\! 1$ for all $t \geq 0$ and $\xi \in \mathcal{S}_{\Dx}$ gives \vspace{-0.075cm}
    \begin{subequations}
\begin{align}
\int_0^t \! \! \int_{\xi_{3N}}^{\xi_{3N+3}}\big|V_{\xi}(s, \eta) - V_{\Dx, \xi}(s, \eta)\big|&d\eta ds \leq \int_0^t\int_{\mathcal{S}^c \cap \mathcal{S}_{\Dx, 3N}} \big|V_{\xi}(s, \eta) - 1\big|d\eta ds \label{eq:energy1}\\
	&  \quad + \int_0^t\int_{\mathcal{S}_{3N} \cap \mathcal{S}_{\Dx}^c }\big|1 - V_{\Dx, \xi}(s, \eta) \big|d\eta ds \label{eq:energy2}\\
 & \quad + \int_0^t\int_{\mathcal{S}_{3N}^c \cap \mathcal{S}_{\Dx, 3N}^c}\big|V_{\xi}(s, \eta) - V_{\Dx, \xi}(s, \eta) \big|d\eta ds.\label{eq:energy3}
 \end{align}
 \end{subequations}
Let us start by estimating \eqref{eq:energy1}. To this end, observe that the difference $V_{\xi}(t, \xi) - \bar{V}_{\xi}(\xi)$ becomes nonzero once wave breaking occurs, cf. \eqref{eq:LagrSystem3}, more precisely 
\begin{align*}
	V_{\xi}(t, \xi) - \bar{V}_{\xi}(\xi) = \begin{cases}
	- \alpha \bar{V}_{\xi}(\xi), & \xi \in \Omega_d(\bar{X}) \cap \Omega_c(X(t)), \\
	0, & \text{otherwise}. \end{cases}
\end{align*}
Furthermore, \eqref{eq:MapL} implies $\bar{V}_{\xi}(\xi) -1 = -\bar{y}_{\xi}(\xi)$. Hence, by Fubini's theorem,  
\begin{align}
	\int_0^t \int_{\mathcal{S}^c \cap \mathcal{S}_{\Dx, 3N}} &\big|V_{\xi}(s, \eta) - 1 \big| d\eta ds \nonumber 
	\\ & \hspace{-0.8cm}\leq \int_{\mathcal{S}^c \cap \mathcal{S}_{\Dx, 3N}\cap (B(t))^c}\int_0^t\bar{y}_{\xi}(\eta)dsd\eta  + \int_{\mathcal{S}^c \cap \mathcal{S}_{\Dx, 3N}\cap B(t)}\int_0^{\tau(\eta)}\bar{y}_{\xi}(\eta)dsd\eta \nonumber
	\\ & \hspace{-0.8cm}\qquad + \int_{\mathcal{S}^c \cap \mathcal{S}_{\Dx, 3N}\cap B(t)}\int_{\tau(\eta)}^t\left(\bar{y}_{\xi}(\eta) + \alpha \bar{V}_{\xi}(\eta)\right)ds d\eta, \nonumber
	\\ & \hspace{-0.8cm} \leq \int_{\mathcal{S}^c \cap \mathcal{S}_{\Dx, 3N}} \int_0^t \bar{y}_{\xi}(\eta)d\eta  + \int_{\mathcal{S}^c \cap \mathcal{S}_{\Dx, 3N}\cap B(t)}\int_{0}^t \alpha \bar{V}_{\xi}(\eta)ds d\eta \label{eq:energy11}
\end{align}
where $B(t) = \Omega_d(\bar{X}) \cap \Omega_c(X(t))$. Since $\bar{y}(\xi_{3j}) = x_{2j}$ for any $j \in \mathbb{Z}$, we have, after a change of variables,
\begin{align}\label{eq:energy1_res1}
	\int_{\mathcal{S}^c \cap \mathcal{S}_{\Dx, 3N}} \int_0^t \bar{y}_{\xi}(\eta)d\eta &\leq t\int_{[\xi_{3N}\!, \hspace{0.05cm}\xi_{3N+3})}\bar{y}_{\xi}(\eta) d\eta = 2t\Dx.
\end{align}
Moreover, \eqref{eq:MapL} and \eqref{eq:rep_y} imply $\bar{V}_{\xi}(\xi) = \bar{u}_x^2(\bar{y}(\xi))\bar{y}_{\xi}(\xi)$ for a.e. $\xi \in \mathcal{S}^c$, which in turn leads to the following estimate for the second term in \eqref{eq:energy11}, 
\begin{align}\label{eq:energy1_res2}
	\int_{\mathcal{S}^c \cap \mathcal{S}_{\Dx, 3N} \cap B(t)}\int_{\tau(\eta)}^t  \alpha \bar{V}_{\xi}(\eta) dsd\eta &\leq t\int_{\mathcal{S}_{3N}^c} \bar{u}_x^2(\bar{y}(\eta))\bar{y}_{\xi}(\eta)d\eta \nonumber 
	\\ &\leq t\int_{x_{2N}}^{x_{2N+2}}\bar{u}_x^2(x)dx \nonumber
	\\ &= t \left(\int_{-\infty}^{x_{2N+2}}\bar{u}_x^2(x)dx - \int_{-\infty}^{x_{2N}}\bar{u}_x^2(x)dx \right) \nonumber
	\\ & \leq t\|\delta_{2\Dx}\bar{u}_x^2 \|_1. 
 \end{align}
One can proceed in the same way when deriving an upper bound on \eqref{eq:energy2}, because by Definition~\ref{def:ProjOP} we have $y_{\Dx}(\xi_{3j}) = x_{2j}$ for all $j \in \mathbb{Z}$, and, in addition, 
\begin{equation}\label{eq:identityFac}
	\int_{-\infty}^{x_{2j}}\bar{u}_x^2(z)dz = \bar{F}_{\mathrm{ac}}(x_{2j}) = \bar{F}_{\Dx, \mathrm{ac}}(x_{2j}) = \int_{-\infty}^{x_{2j}}\bar{u}_{\Dx, x}^2(z)dz \hspace{0.35cm} \text{for all} j \in \mathbb{Z}. 
\end{equation}

Finally, we are left to consider \eqref{eq:energy3}. By \eqref{eq:LagrSystem} one has $V_{\xi}(s, \xi)\! \leq \bar{u}_x^2(\bar{y}(\xi))\bar{y}_{\xi}(\xi)$ for a.e. $\xi\! \in\! \mathcal{S}^c$ and  $V_{\Dx, \xi}(s, \xi)\! \leq \bar{u}_{\Dx, x}^2(\bar{y}_{\Dx}(\xi))\bar{y}_{\Dx, \xi}(\xi)$ for a.e. $\xi \in \mathcal{S}_{\Dx}^c$. Therefore, by combining \eqref{eq:identityFac} with the argument leading to \eqref{eq:energy1_res2}, one obtains
\begin{align*}
	 \int_0^t\int_{\mathcal{S}_{3N}^c \cap \mathcal{S}_{\Dx, 3N}^c} \! \!\! \!\big|V_{\xi}(s, \eta) - V_{\Dx, \xi}(s, \eta) \big|d\eta ds & \leq t \int_{\mathcal{S}_{3N}^c } \bar{u}_x^2(\bar{y}(\eta))\bar{y}_{\xi}(\eta)d\eta
	 \\ & \qquad + t \int_{\mathcal{S}_{\Dx, 3N}^c } \bar{u}_{\Dx, x}^2(\bar{y}_{\Dx}(\eta))\bar{y}_{\Dx, \xi}(\eta)d\eta
	 \\ & \leq 2t \|\delta_{2\Dx}\bar{u}_{x}^2\|_1. 
\end{align*}
 Combining this with \eqref{eq:energy1_res1}--\eqref{eq:energy1_res2} and similar estimates for \eqref{eq:energy2} gives 
 \begin{equation}\label{eq:energy_result}
 	\int_0^t\int_{\xi_{3N}}^{\xi_{3N+3}} \big|V_{\xi}(s, \eta) - V_{\Dx, \xi}(s, \eta) \big|d\eta ds \leq 4 t\Dx + 4t\|\delta_{2\Dx}\bar{u}_x^2 \|_1.
 \end{equation}
 At last, after combining \eqref{eq:absolutelyContrib} and \eqref{eq:energy_result} with Lemma~\ref{lem:convInt} and Proposition~\ref{prop:simpleEst}, the asserted inequality, \eqref{eq:bound_U}, follows.
 \end{proof}

\subsection{Convergence rate for $\alpha$-dissipative solutions}\label{sec:ratesLagr}
In order to derive the sought rate of convergence for $\{u_{\Dx}(t)\}_{\Dx > 0}$, it suffices by \eqref{eq:est_u}, Lemma~\ref{lem:convInt}, Proposition~\ref{prop:simpleEst}, and Theorem~\ref{thm:Discrepancy} to derive convergence rates for $\|\bar{V}_{\xi} - \bar{V}_{\Dx, \xi}(f)\|_{L^1(\mathcal{S}^c)}$, $\|g(\bar{X})- g(\bar{X}_{\Dx})(f)\|_{L^1(\mathcal{S}^c)}$, and $\|\bar{U}_{\xi} - \bar{U}_{\Dx, \xi}(f) \|_{L^2(\mathcal{S}^c)}$. The identities \eqref{eq:rep_y} and \eqref{eq:iden_yxif} together with Lemma~\ref{lem:Rateux} are the crux of this matter.

\begin{lemma}\label{lem:rateHxi}
	Given $(\bar{u}, \bar{\mu}) \in \D$ with $\bar{u}_x \in B_2^{\beta}$ and $\bar{\mu}_{\mathrm{sc}}=0$, let $\bar{X}_{\Dx} = L \circ P_{\Dx}((\bar{u}, \bar{\mu}))$ and $\bar{X} = L ((\bar{u}, \bar{\mu}))$. Moreover, let $f$ denote the mapping from \eqref{eq:map_f}.  Then there exists a constant $C>0$ only depending on $\beta$, $\bar{F}_{\mathrm{ac}, \infty}$, and $|\bar{u}_x|_{2, \beta}$ such that  
\begin{align}\label{eq:rateHxi}
	\|\bar{V}_{\xi} - \bar{V}_{\Dx, \xi}(f) \|_{L^1(\mathcal{S}^c)} &\leq C\Dx^{\nicefrac{\beta}{2}}\!.
	\end{align}
\end{lemma}

\begin{proof}
To simplify notation, let us suppress the argument $\xi$. Note that $\bar{V}_{\xi} \! = 1 - \bar{y}_{\xi}$ and $\bar{V}_{\Dx, \xi}(f) = 1 - \bar{y}_{\Dx, \xi}(f)$ a.e. by \eqref{eq:MapL} and \eqref{eq:id1}, respectively. Hence, due \eqref{eq:rep_y} and Lemma~\ref{lem:groupAction},  we have
\begin{subequations}
\begin{align}
		\|\bar{V}_{\xi} - \bar{V}_{\Dx, \xi}(f) \|_{L^1(\mathcal{S}^c)} 
		&= \int_{\mathcal{S}^c} \bigg| \frac{\bar{u}_x^2 (\bar{y})}{1 + \bar{u}_x^2(\bar{y})} - \frac{\bar{u}_{\Dx, x}^2(\bar{y}_{\Dx}(f))}{1 + \bar{u}_{\Dx, x}^2(\bar{y}_{\Dx}(f))} \bigg|d\xi  \nonumber 
		\\ &= \int_{\mathcal{S}^c} \left|\bar{u}_x^2 (\bar{y}) - \bar{u}_{\Dx, x}^2(\bar{y}_{\Dx}(f)) \right|\bar{y}_{\xi} \bar{y}_{\Dx, \xi}(f)d\xi \nonumber
		\\ & \leq \int_{\mathcal{S}^c} \big|\bar{u}_x^2(\bar{y}) - \bar{u}_{\Dx, x}^2(\bar{y}) \big|\bar{y}_{\xi}\bar{y}_{\Dx, \xi}(f)d\xi\ \label{eq:est_Hxi1}
		\\ & \quad + \int_{\mathcal{S}^c} \big|\bar{u}_{\Dx, x}^2(\bar{y}) - \bar{u}_{\Dx, x}^2(\bar{y}_{\Dx}(f)) \big|\bar{y}_{\xi} \bar{y}_{\Dx, \xi}(f)d\xi  \label{eq:est_Hxi2}. 
	\end{align}
	\end{subequations}
After using the change of variables $x=\bar{y}(\xi)$, that $0 \leq \bar{y}_{\Dx, \xi}(f(\xi)) \leq 1$ for a.e. $\xi$, and the CS inequality, we end up with the following estimate for \eqref{eq:est_Hxi1} 
\begin{align}\label{eq:CSArg}
	\int_{\mathcal{S}^c} \! \! \left| \left(\bar{u}_x^2 - \bar{u}_{\Dx, x}^2 \right)\! (\bar{y})\right|\bar{y}_{\xi} \bar{y}_{\Dx, \xi}(f)d\xi &\leq \int_{\mathcal{S}^c} \! \!\big|\bar{u}_x(\bar{y}) + \bar{u}_{\Dx, x}(\bar{y}) \big| \big|\bar{u}_x(\bar{y}) - \bar{u}_{\Dx, x}(\bar{y}) \big|\bar{y}_{\xi}d\xi \nonumber \\ 
	&\leq 2\bar{F}_{\mathrm{ac}, \infty}^{\nicefrac{1}{2}} \|\bar{u}_x- \bar{u}_{\Dx, x}\|_2. 
\end{align}

On the other hand, Definition~\ref{def:ProjOP} implies that
\begin{align*}
	\int_{\mathcal{S}^c}\! \bigl|\bar{u}_{\Dx, x}^2(\bar{y}) - \bar{u}_{\Dx, x}^2 &(\bar{y}_{\Dx}(f))  \bigr|\bar{y}_{\xi} \bar{y}_{\Dx, \xi}(f)d\xi \nonumber 
\\ & \leq \sum_{j \in \mathbb{Z}} \int_{B_{3j}}\left|\bar{u}_{\Dx, x}^2(\bar{y}) - \bar{u}_{\Dx, x}^2 (\bar{y}_{\Dx}(f)) \right|\bar{y}_{\xi}d\xi, 
\end{align*}
where
\begin{align}
    B_{3j} :\!&= \bigl \{ \xi \in  [\xi_{3j}, \xi_{3j+3})\!: \bar{y}_{\Delta x}(f(\xi)) \leq x_{2j+1} < \bar{y}(\xi)  \nonumber
    \\ & \hspace{3.0cm} \text{ or } \bar{y}(\xi) \leq x_{2j+1} < \bar{y}_{\Delta x}(f(\xi)) \bigr\}.
    \label{eq:set_B3j}
\end{align}
In addition, recalling \eqref{eq:square_root} reveals that for a.e. $\xi \in B_{3j}$ we have
\begin{equation*}
      \left |\bar{u}_{\Delta x, x}^2 (\bar{y}) - \bar{u}_{\Delta x, x}^2 (\bar{y}_{\Dx}(f)) \right|\! = 4 \left|D\bar{u}_{2j}\right| q_{2j}= 4 \left|D\bar{u}_{2j}\right| \!\sqrt{D\bar{F}_{\mathrm{ac}, 2j} - \left(D\bar{u}_{2j} \right)^2 },  
\end{equation*}
and thus, by the CS inequality, we obtain the following upper bound on \eqref{eq:est_Hxi2}, 
\begin{align*}
	\int_{\mathcal{S}^c} \bigl|\bar{u}_{\Dx, x}^2&(\bar{y}) - \bar{u}_{\Dx, x}^2(\bar{y}_{\Dx}(f)) \bigr| \bar{y}_{\xi}\bar{y}_{\Dx, \xi}(f)d\xi \nonumber 
    \\ & \leq 4\bigg(\sum_{j \in \mathbb{Z}} \int_{B_{3j}} \! \! |D\bar{u}_{2j}|^2\bar{y}_{\xi}d\xi \bigg)^{\nicefrac{1}{2}} \!
    \bigg(\sum_{j \in \mathbb{Z}} \int_{B_{3j}} \! \!|D\bar{F}_{\text{ac}, 2j} - (D\bar{u}_{2j})^2|\bar{y}_{\xi}d\xi \bigg)^{\nicefrac{1}{2}}\!\!. 
\end{align*}

Since $\bar{y}(B_{3j}) \subseteq [x_{2j}, x_{2j+2}]$, another application of the CS inequality leads to
\begin{align*}
    \bigg(\sum_{j \in \mathbb{Z}} \int_{B_{3j}}\! |D\bar{u}_{2j}|^2\bar{y}_{\xi}d\xi \bigg)^{\nicefrac{1}{2}} \!&\leq \bigg(\sum_{j \in \mathbb{Z}} \int_{B_{3j}} \frac{1}{2\Dx}\int_{x_{2j}}^{x_{2j+2}}\bar{u}_x^2(z)dz \bar{y}_{\xi}d\xi\bigg)^{\nicefrac{1}{2}} 
    \\ &\leq \bigg( \sum_{j \in \mathbb{Z}}\frac{\bar{F}_{\mathrm{ac}, 2j+2} - \bar{F}_{\mathrm{ac}, 2j}}{2\Dx}\int_{B_{3j}}\bar{y}_{\xi}d\xi \bigg)^{\nicefrac{1}{2}}\! \!\leq \bar{F}_{\mathrm{ac}, \infty}^{\nicefrac{1}{2}}, 
\end{align*}
and, after combining \eqref{eq:con_ux3} with \eqref{eq:ux_est2_1}--\eqref{eq:ux_est2_2_1}, one obtains
\begin{align}\label{eq:estEnergyDiff}
	 \sum_{j \in \mathbb{Z}} \int_{B_{3j}} \!\!|D\bar{F}_{\mathrm{ac}, 2j} - (D\bar{u}_{2j})^2|\bar{y}_{\xi}d\xi &\leq 2^{1+\beta}\bar{F}_{\mathrm{ac}, \infty}^{\nicefrac{1}{2}}\! \left( \frac{2}{\beta+1} + \sqrt{\frac{2}{2\beta +1}}\right)\! |\bar{u}_x|_{2, \beta}\Dx^{\beta}\!.
\end{align}
Combining this together with \eqref{eq:CSArg} yields 
\begin{align*}
	\|\bar{V}_{\xi} - \bar{V}_{\Dx, \xi}(f)\|_{L^1(\mathcal{S}^c)} &\leq 2\bar{F}_{\mathrm{ac}, \infty}^{\nicefrac{1}{2}} \|\bar{u}_x - \bar{u}_{\Dx, x}  \|_2 \nonumber
	\\ & \quad \hspace{-0.2cm} + 2^{\nicefrac{(5+\beta)}{2}}\bar{F}_{\mathrm{ac}, \infty}^{\nicefrac{3}{4}}\!\! \left(\frac{2}{\beta+1} + \sqrt{\frac{2}{2\beta + 1}}\right)^{\nicefrac{1}{2}}\! \!|\bar{u}_x|_{2, \beta}^{\nicefrac{1}{2}}\Dx^{\nicefrac{\beta}{2}}\!, 
\end{align*}
and by subsequently using the bound in Lemma~\ref{lem:Rateux} one gets \eqref{eq:rateHxi}.
\end{proof}

\begin{lemma}\label{lem:rateg}
	Suppose $(\bar{u}, \bar{\mu}) \in \D$ satisfy $\bar{u}_x \in B_2^{\beta}$ and $\bar{\mu}_{\mathrm{sc}}=0$. Moreover, recall the function $f$ from \eqref{eq:map_f} and let $\bar{X}_{\Dx} = L \circ P_{\Dx}((\bar{u}, \bar{\mu}))$ and $\bar{X} = L ((\bar{u}, \bar{\mu}))$. Then 
	\begin{align*}
	\|g(\bar{X}) - g(\bar{X}_{\Dx})(f)\|_{L^1(\mathcal{S}^c)}\! &\leq\! \tilde{D}\Dx^{\nicefrac{\beta}{2}}, 
\end{align*}
where $\tilde{D}>0$ is a constant only dependent on $\beta$, $\bar{F}_{\mathrm{ac}, \infty}$ and $|\bar{u}_x|_{2, \beta}$. Recall that the latter quantity is defined by \eqref{eq:semiNorm}. 
\end{lemma}

\begin{proof}
Combining the very definition of $g$, i.e.,  \eqref{eq:function_g} together with the sets from \eqref{eq:setNotation} yields
\begin{align*}
	\|g(\bar{X}) &- g(\bar{X}_{\Dx})(f) \|_{L^1(\mathcal{S}^c)} \nonumber 
	\\ & \leq \int_{\bar{\Omega}_{c, c} \cup \bar{\Omega}_{d, d}}\big|\bar{V}_{\xi} - \bar{V}_{\Dx, \xi}(f)\big|d\xi +   \int_{\bar{\Omega}_{d, c}}\big|(1-\alpha)\bar{V}_{\xi} - \bar{V}_{\Dx, \xi}(f) \big|d\xi \nonumber
	\\ &\qquad \quad +\int_{\bar{\Omega}_{c, d}}\!\left|\bar{V}_{\xi} - (1-\alpha)\bar{V}_{\Dx, \xi}(f)\right|d\xi \nonumber
		\\ & \leq \int_{\mathcal{S}^c}\! \left|\bar{V}_{\xi} - \bar{V}_{\Dx, \xi}(f) \right|d\xi 
		 + \alpha \int_{\bar{\Omega}_{c, d}}\bar{u}_x^2(\bar{y})\bar{y}_{\xi}d\xi 
		 +  \alpha \int_{\bar{\Omega}_{d, c}}\bar{u}_{x}^2(\bar{y})\bar{y}_{\xi}d\xi, 
\end{align*}
because $\bar{V}_{\xi}(\xi)=\bar{u}_x^2(\bar{y}(\xi))\bar{y}_{\xi}(\xi)$ for a.e. $\xi \in \mathcal{S}^c$. 

Lemma~\ref{lem:rateHxi} provides an upper bound on the first term. It therefore remains to examine the two other terms. Since their structures are similar, we only pay attention to the second term. The key observation is that  $\bar{u}_x(\bar{y}(\xi))$ and $\bar{u}_{\Dx, x}(\bar{y}_{\Dx}(f(\xi)))$ have opposite signs for a.e. $\xi \in \bar{\Omega}_{c, d} \cup \bar{\Omega}_{d, c}$, hence 
\begin{subequations}
\begin{align}
	\alpha  \int_{\bar{\Omega}_{c, d}}\bar{u}_x^2(\bar{y})\bar{y}_{\xi}d\xi &\leq \alpha  \int_{\bar{\Omega}_{c, d}}\bar{u}_x(\bar{y}) \big(\bar{u}_x(\bar{y}) - \bar{u}_{\Dx, x}(\bar{y}_{\Dx}(f)) \big)\bar{y}_{\xi}d\xi \nonumber
	\\ &\leq  \int_{\mathcal{S}^c}\bar{u}_x(\bar{y})\big(\bar{u}_x - \bar{u}_{\Dx, x}\big)(\bar{y})\bar{y}_{\xi}d\xi \label{eq:mixed1}
	\\ & \qquad + \int_{\mathcal{S}^c}\bar{u}_x(\bar{y})\big(\bar{u}_{\Dx, x}(\bar{y}) - \bar{u}_{\Dx, x}(\bar{y}_{\Dx}(f))\big)\bar{y}_{\xi}d\xi \label{eq:mixed2}.
\end{align}
\end{subequations}
The term \eqref{eq:mixed1} resembles \eqref{eq:CSArg} and can be estimated similarly, in particular, 
\begin{equation*}
	 \int_{\mathcal{S}^c}\bar{u}_x(\bar{y})\big(\bar{u}_x - \bar{u}_{\Dx, x}\big)(\bar{y})\bar{y}_{\xi}d\xi \leq \bar{F}_{\mathrm{ac}, \infty}^{\nicefrac{1}{2}}\|\bar{u}_x - \bar{u}_{\Dx, x}\|_2.
\end{equation*}
On the other hand, the integrand in \eqref{eq:mixed2} vanishes a.e. outside the sets from \eqref{eq:set_B3j}.  More precisely, by Definition~\ref{def:ProjOP},  
\begin{equation*}
	|\bar{u}_{\Dx, x}(\bar{y}(\xi)) - \bar{u}_{\Dx, x}(\bar{y}_{\Dx}(f(\xi)))| = 2\sqrt{D\bar{F}_{\mathrm{ac}, 2j} - (D\bar{u}_{2j})^2} \hspace{0.1cm}  \text{a.e.} \hspace{0.05cm} \xi \!\in B_{3j},
\end{equation*}
which in turn, when combined with the CS inequality and \eqref{eq:estEnergyDiff} yields 
\begin{align}\label{eq:lastg}
	\int_{\mathcal{S}^c}\bar{u}_x(y)\big(\bar{u}_{\Dx, x}(\bar{y}) &- \bar{u}_{\Dx, x}(\bar{y}_{\Dx}(f))\big)\bar{y}_{\xi}d\xi \nonumber
	\\ &\leq 2 \Big( \int_{\R}\bar{u}_x^2(\bar{y})\bar{y}_{\xi}d\xi \Big)^{\nicefrac{1}{2}} \Big( \sum_{j \in \mathbb{Z}}\int_{B_{3j}} \big|D\bar{F}_{\mathrm{ac}, 2j} - (D\bar{u}_{2j})^2 \big|\bar{y}_{\xi}d\xi \Big)^{\nicefrac{1}{2}} \nonumber
	\\ &\leq 2^{\nicefrac{(3+\beta)}{2}} \bar{F}_{\mathrm{ac}, \infty}^{\nicefrac{3}{4}}\! \! \left(\frac{2}{\beta+1} + \sqrt{\frac{2}{2\beta +1}}\right)^{\nicefrac{1}{2}}\! |\bar{u}_x|_{2, \beta}^{\nicefrac{1}{2}} \Dx^{\nicefrac{\beta}{2}}.
\end{align}
To summarize, we have shown that 
\begin{align*}
	\|g(\bar{X}) - g(\bar{X}_{\Dx})(f)\|_{L^1(\mathcal{S}^c)} &\leq \|\bar{V}_{\xi} - \bar{V}_{\Dx, \xi}(f)\|_{L^1(\mathcal{S}^c)} + 2\bar{F}_{\mathrm{ac}, \infty}^{\nicefrac{1}{2}} \|\bar{u}_x - \bar{u}_{\Dx, x}\|_2 \nonumber
	\\ & \quad \hspace{-0.3cm}+ 2^{\nicefrac{(5+\beta)}{2}}\bar{F}_{\mathrm{ac}, \infty}^{\nicefrac{3}{4}}\!\! \left(\frac{2}{\beta+1} + \sqrt{\frac{2}{2\beta +1}}\right)^{\nicefrac{1}{2}}\!|\bar{u}_x|_{2, \beta}^{\nicefrac{1}{2}} \Dx^{\nicefrac{\beta}{2}}\!, 
\end{align*}
and the asserted statement therefore follows by Lemma~\ref{lem:Rateux} and Lemma~\ref{lem:rateHxi}.  
\end{proof}

\begin{lemma}\label{lem:rateUxi}
	Given $(\bar{u}, \bar{\mu}) \!\in \! \D$ with $\bar{u}_x \! \in \!B_2^{\beta}$ and $\bar{\mu}_{\mathrm{sc}}=0$, let $\bar{X}_{\Dx} = L \circ P_{\Dx}((\bar{u}, \bar{\mu}))$ and $\bar{X} = L ((\bar{u}, \bar{\mu}))$. Moreover, let $f$ denote the mapping from \eqref{eq:map_f}.  Then there exists $\hat{D}>0$ which only depends on $\beta$, $\bar{F}_{\mathrm{ac}, \infty}$ and $|\bar{u}_x|_{2, \beta}$ such that
	\begin{align*}
		\|\bar{U}_{\xi} - \bar{U}_{\Dx, \xi}(f)\|_{L^2(\mathcal{S}^c)} &\leq \hat{D}\Dx^{\nicefrac{\beta}{4}}.	\end{align*}
\end{lemma}
 
\begin{proof}
	To begin with, observe that for a.e. $\xi \in \mathcal{S}^c$ we have
	\begin{equation*}
		\bar{U}_{\xi}(\xi) = \bar{u}_x (\bar{y}(\xi))\bar{y}_{\xi}(\xi) \hspace{0.3cm} \text{and} \hspace{0.3cm} \bar{U}_{\Dx, \xi}(f(\xi)) = \bar{u}_{\Dx, x}(\bar{y}_{\Dx}(f(\xi))\bar{y}_{\Dx, \xi}(f(\xi)), 
	\end{equation*}
	which implies  
\begin{subequations}\label{eq:decompUxi}
\begin{align}
     \|\bar{U}_{\xi} - \bar{U}_{\Delta x, \xi}(f)\|_{L^2(\mathcal{S}^c)} &= \|\bar{u}_x(\bar{y}) \bar{y}_{\xi} - \bar{u}_{\Dx, x}(\bar{y}_{\Dx}(f)) \bar{y}_{\Dx, \xi}(f)\|_{L^2(\mathcal{S}^c)} \nonumber 
     \\ & \leq \|\!\left(\bar{u}_x - \bar{u}_{\Dx, x} \right)(\bar{y})\bar{y}_{\xi}\|_{L^2(\mathcal{S}^c)}  \label{eq:Uxi1}
     \\ & \quad + \|\!\left(\bar{u}_{\Dx, x}(\bar{y}) - \bar{u}_{\Dx, x}(\bar{y}_{\Dx}(f)) \right)\bar{y}_{\xi}\|_{L^2(\mathcal{S}^c)}\label{eq:Uxi2}
     \\ &\quad + \|\bar{u}_{\Dx, x} (\bar{y}_{\Dx}(f)) \left(\bar{y}_{\xi} - \bar{y}_{\Dx, \xi}(f)\right)\|_{L^2(\mathcal{S}^c)}\label{eq:Uxi3}.
     \end{align}
 \end{subequations}
 
To estimate \eqref{eq:Uxi1}, we use the substitution $x=\bar{y}({\xi})$ and the fact that $0 \leq \bar{y}_{\xi}\leq 1$ a.e., which leads to 
 \begin{equation*} 
 	\|\!\left(\bar{u}_x - \bar{u}_{\Dx, x}\right)\!(\bar{y})\bar{y}_{\xi}\|_{L^2(\mathcal{S}^c)} \leq \|\bar{u}_x - \bar{u}_{\Dx, x}\|_2,
\end{equation*}
while by arguing similarly to how we did for in \eqref{eq:lastg}, we end up with the following estimate for \eqref{eq:Uxi2}
\begin{align*}
	 \|(\bar{u}_{\Dx, x}(\bar{y}) - \bar{u}_{\Dx, x}(&\bar{y}_{\Dx}(f)))\bar{y}_{\xi}\|_{L^2(\mathcal{S}^c)} 
	 \\ &\leq \! 2^{\nicefrac{(3+\beta)}{2}}\bar{F}_{\mathrm{ac}, \infty}^{\nicefrac{1}{4}}\! \! \left(\frac{2}{\beta+1} + \sqrt{\frac{2}{2\beta +1}}\right)^{\nicefrac{1}{2}}\! \! \!|\bar{u}_x|_{2, \beta}^{\nicefrac{1}{2}} \Dx^{\nicefrac{\beta}{2}}\!. 
\end{align*}
 The term \eqref{eq:Uxi3} on the other hand, requires a bit more work. We start by splitting it in the following way
 \begin{subequations}
 \begin{align}
 	\|\bar{u}_{\Dx, x} (\bar{y}_{\Dx}(f)) &\left(\bar{y}_{\xi} - \bar{y}_{\Dx, \xi}(f) \right)\|_{L^2(\mathcal{S}^c)}^2 \nonumber 
	\\ &=  \int_{\mathcal{S}^c} \bar{u}_{\Dx, x}^2(\bar{y}_{\Dx}(f))\bar{y}_{\Dx, \xi}(f)(\bar{y}_{\Dx, \xi}(f) - \bar{y}_{\xi})d\xi\label{eq:Uxi31}
	\\ & \qquad + \int_{\mathcal{S}^c} \bar{u}_{\Dx, x}^2(\bar{y}_{\Dx}(f))\bar{y}_{\xi}\left(\bar{y}_{\xi} - \bar{y}_{\Dx, \xi}(f)\right)d\xi. \label{eq:Uxi32}
	\end{align}
	\end{subequations}
	Moreover, we observe by  \eqref{eq:MapL}, Lemma~\ref{lem:groupAction},  and \eqref{eq:transformId}  that $\bar{u}_{\Dx, x}^2(\bar{y}_{\Dx}(f(\xi)))\bar{y}_{\Dx, \xi}(f(\xi)) = \bar{V}_{\Dx, \xi}(f(\xi)) \leq 1$ and  $\bar{y}_{\Dx, \xi}(f)-\bar{y}_{\xi} =\bar{V}_{\xi} - \bar{V}_{\Dx, \xi}(f)$ for a.e. $\xi \in \mathcal{S}^c$ such that
\begin{equation*}
		\int_{\mathcal{S}^c} \bar{u}_{\Dx, x}^2(\bar{y}_{\Dx}(f))\bar{y}_{\Dx, \xi}(f)(\bar{y}_{\Dx, \xi}(f) - \bar{y}_{\xi})d\xi \leq \|\bar{V}_{\xi} - \bar{V}_{\Dx, \xi}(f)\|_{L^1(\mathcal{S}^c)}.  
\end{equation*}
It therefore remains to estimate \eqref{eq:Uxi32}, which can be decomposed into the more familiar terms 
\begin{subequations}
	\begin{align}
		\int_{\mathcal{S}^c} \! \! \bar{u}_{\Dx, x}^2(\bar{y}_{\Dx}(f))\bar{y}_{\xi}\! \left(\bar{y}_{\xi} - \bar{y}_{\Dx, \xi}(f)\right)\!d\xi &\leq \int_{\mathcal{S}^c} \! \! \bar{u}_x^2(\bar{y})\bar{y}_{\xi}\! \left(\bar{y}_{\xi} - \bar{y}_{\Dx, \xi}(f)\right)\!d\xi \nonumber
		\\ & \hspace{-1.5cm} \!\!+  \int_{\mathcal{S}^c} \!\!\left(\bar{u}_{\Dx, x}^2 - \bar{u}_x^2\right)\!(\bar{y})\bar{y}_{\xi} \!\left(\bar{y}_{\xi} - \bar{y}_{\Dx, \xi}(f)\right)\!d\xi \label{eq:Uxi311}
		\\ & \hspace{-1.5cm} + \int_{\mathcal{S}^c}\! \! \left(\bar{u}_{\Dx, x}^2(\bar{y}_{\Dx}(f)) - \bar{u}_{\Dx, x}^2(\bar{y})\right)\bar{y}_{\xi}\!\left(\bar{y}_{\xi} - \bar{y}_{\Dx, \xi}(f)\right)\!d\xi. \label{eq:Uxi312}
	\end{align}
	\end{subequations}
For the first term one can proceed precisely as for \eqref{eq:Uxi31}, yielding
\begin{equation*}
	\int_{\mathcal{S}^c} \bar{u}_x^2(\bar{y})\bar{y}_{\xi} \big(\bar{y}_{\xi} - \bar{y}_{\Dx, \xi}(f)\big)d\xi \leq \|\bar{V}_{\xi} - \bar{V}_{\Dx, \xi}(f)\|_{L^1(\mathcal{S}^c)}, 
\end{equation*}
whilst for \eqref{eq:Uxi311} we can exploit that $|\bar{y}_{\xi}- \bar{y}_{\Dx, \xi}(f)| \leq 1$ for a.e. $\xi \in \mathcal{S}^c$ and argue as in \eqref{eq:est_Hxi1}, which results in   
 \begin{align*}
 	\int_{\mathcal{S}^c}\!\big(\bar{u}_{\Dx, x}^2 - \bar{u}_x^2\big)(\bar{y})\bar{y}_{\xi} \big(\bar{y}_{\xi} - \bar{y}_{\Dx, \xi}(f)\big)d\xi &\leq \int_{\mathcal{S}^c} \!\!\big|\bar{u}_{\Dx, x}^2 - \bar{u}_x^2\big|(\bar{y})\bar{y}_{\xi}d\xi 
	\\ &\leq 2 \bar{F}_{\mathrm{ac}, \infty}\|\bar{u} - \bar{u}_{\Dx, x}\|_2. 
\end{align*}

At last, \eqref{eq:Uxi312} can be bounded from above by following the argument for \eqref{eq:est_Hxi2},  that is
\begin{align*}
	\int_{\mathcal{S}^c}\! \big(\bar{u}_{\Dx, x}^2(\bar{y}_{\Dx}(f)) &- \bar{u}_{\Dx, x}^2(\bar{y})\big)\bar{y}_{\xi}\big(\bar{y}_{\xi} - \bar{y}_{\Dx, \xi}(f)\big)d\xi 
	\\ &\leq \int_{\mathcal{S}^c}\! \big|\bar{u}_{\Dx, x}^2(\bar{y}_{\Dx}(f)) - \bar{u}_{\Dx, x}^2(\bar{y})\big|\bar{y}_{\xi}d\xi
	\\ & \leq 2^{\nicefrac{(5+\beta)}{2}}\bar{F}_{\mathrm{ac}, \infty}^{\nicefrac{3}{4}}\!\! \left(\frac{2}{\beta+1} + \sqrt{\frac{2}{2\beta + 1}}\right)^{\nicefrac{1}{2}}\! \!|\bar{u}_x|_{2, \beta}^{\nicefrac{1}{2}}\Dx^{\nicefrac{\beta}{2}}. 
\end{align*}
Lemma~\ref{lem:rateUxi} is therefore a consequence of Lemma~\ref{lem:Rateux} and Lemma~\ref{lem:rateHxi}. 
\end{proof}

Finally, combining Lemmas~\ref{lem:rateHxi}--\ref{lem:rateUxi} together with the a priori estimate in Theorem~\ref{thm:Discrepancy} culminates in the following error estimate. 

\begin{theorem}\label{thm:ConvRateu}
	Given  $(\bar{u}, \bar{\mu}) \!\in \! \D$ with $\bar{u}_{x} \!\! \in \! B_2^{\beta}$, $\Dx \leq 1$, and $\bar{\mu}_{\mathrm{sc}}=0$, let $\bar{X}_{\Dx} = L \circ P_{\Dx} \left((\bar{u}, \bar{\mu})\right)$ and $\bar{X} = L \left((\bar{u}, \bar{\mu})\right)$. Moreover, for any $t \in (0, \infty)$ introduce $(y, U, V)(t) = S_t(\bar{X})$ and $(y_{\Dx}, U_{\Dx}, V_{\Dx})(t) = S_t(\bar{X}_{\Dx})$, then \vspace{-0.075cm}
		\begin{align}
		\|U(t) - U_{\Dx}(t)\|_{\infty} &\leq (1 + t)C_1 \Dx^{\nicefrac{\beta}{4}}  , \label{eq:rateU}\\
		\|y(t) - y_{\Dx}(t) \|_{\infty} &\leq (1 +t)tC_2 \Dx^{\nicefrac{\beta}{4}} + 2\Dx^{\nicefrac{\beta}{4}}, \label{eq:ratey}
	\end{align}
	where $C_1$ and $C_2$ only depend on $\beta$, $\bar{F}_{\infty}$, and $|\bar{u}_x|_{2, \beta}$. 
	Furthermore, let $(u, \mu)(t) = T_t((\bar{u}, \bar{\mu}))$ and $(u_{\Dx}, \mu_{\Dx})(t) = T_t \circ P_{\Dx} ((\bar{u}, \bar{\mu}))$, it then holds that  \vspace{-0.05cm}
	\begin{align}\label{eq:convRateu}
		\|u(t) - u_{\Dx}(t) \|_{\infty} &\leq D(1+\sqrt{t}+t)\Dx^{\nicefrac{\beta}{8}}, 
	\end{align}
	for a constant $D$ which only depends on $\beta$, $\bar{F}_{\infty}$, and $|\bar{u}_x|_{2, \beta}$. 
\end{theorem}

\begin{proof} 
	From \eqref{eq:L1_normLip} and Theorem~\ref{thm:Discrepancy} we have  \vspace{-0.075cm}
 \begin{align*}
 	\| U(t) - U_{\Dx}(t)\|_{\infty} &\leq t\Dx + (1+2\sqrt{2})\bar{F}_{\mathrm{ac}, \infty}^{\nicefrac{1}{2}}\Dx^{\nicefrac{1}{2}} + 2^{1+\beta}|\bar{u}_x|_{2, \beta}\bar{F}_{\mathrm{ac}, \infty}^{\nicefrac{1}{2}}t\Dx^{\beta}
	\\ & \quad + \frac{1}{4}(1+\alpha)t\|\bar{V}_{\xi} - \bar{V}_{\Dx, \xi}(f)\|_{L^1(\mathcal{S}^c)}
	\\ & \quad + \frac{t}{4}\|g(\bar{X}) - g(\bar{X}_{\Dx})(f)\|_{L^1(\mathcal{S}^c)}
	\\ & \quad + \frac{\alpha}{2} \sqrt{2\Big(1 + \frac{1}{4}t^2\Big)\bar{F}_{\infty}}\|\bar{U}_{\xi} - \bar{U}_{\Dx, \xi}(f)\|_{L^2(\mathcal{S}^c)},
\end{align*}
and by inserting the bounds from Lemmas~\ref{lem:rateHxi}--\ref{lem:rateUxi},  we obtain \eqref{eq:rateU}. Combining this bound with Lemma~\ref{lem:convInt} and Proposition~\ref{prop:simpleEst} gives \eqref{eq:ratey}, which in turn, when combined with  \eqref{eq:est_u} and \eqref{eq:rateU}, immediately establishes \eqref{eq:convRateu}. 
 \end{proof}
 
 \begin{remark} \label{rem:explanation}
 	Mapping into and evolving in Lagrangian coordinates reduces the order $\mathcal{O}(\Dx^{\nicefrac{\beta}{2}})$ from Lemma~\ref{lem:Rateux} by a half, see Figure~\ref{fig:schematicRep}. The main reason for this reduction is the term $\|\bar{U}_{\xi} - \bar{U}_{\Dx, \xi}(f)\|_{L^2(\mathcal{S}^c)}$, which accounts for the energy discrepancy which arises when either $X$ or $X_{\Dx}$ experiences wave breaking before the other solution, cf. \eqref{eq:oneBreaksBefore}. 
	
Mapping back to Eulerian coordinates induces an additional reduction in the rate of convergence, because $u$ is in general only H{\"o}lder continuous with exponent $\nicefrac{1}{2}$, we therefore end up with $\|u(t) - u_{\Dx}(t)\|_{\infty} = \mathcal{O}(\Dx^{\nicefrac{\beta}{8}})$. 
  \end{remark}

\section{Convergence rate for $\alpha=0$}\label{sec:convConservative}
For conservative solutions there is no energy dissipation, and mapping to and evolving in Lagrangian coordinates does therefore not deteriorate the rate of convergence. On the other hand, by exploiting that $V(t, \xi)\!  = \! \bar{V}(\xi)$ and $V_{\Dx}(t, \xi)\! \!=  \!\bar{V}_{\Dx}(\xi)$ for all $t\!\geq \! 0$ and $\xi \in \R$, one achieves the following much stronger result. 

\begin{lemma}\label{lem:RateConservative}
	Let $\alpha=0$, $(\bar{u}, \bar{\mu}) \in \D$,  and introduce $(y, U, V)(t) = S_t \circ L \left((\bar{u}, \bar{\mu})\right)$ and $(y_{\Dx}, U_{\Dx}, V_{\Dx})(t) = S_t \circ L \circ P_{\Dx}  \left((\bar{u}, \bar{\mu})\right)$ for $t \in [0, \infty)$, then 
	\begin{subequations}
	\begin{align}
		\|U(t) - U_{\Dx}(t)\|_{\infty} &\leq (1 + 2\sqrt{2})\bar{F}_{\mathrm{ac}, \infty}^{\nicefrac{1}{2}}\Delta x^{\nicefrac{1}{2}} + t \Dx, \label{eq:consU}\\ 
		\|y(t) - y_{\Dx}(t)\|_{\infty} &\leq (1 + 2\sqrt{2} )\bar{F}_{\mathrm{ac}, \infty}^{\nicefrac{1}{2}}t \Delta x^{\nicefrac{1}{2}} + \big(2+\frac{1}{2}t^2\big)\Dx. \label{eq:consy}
	\end{align}
	\end{subequations} 
	Furthermore, let the exact and numerical conservative solution be denoted by $(u, \mu)(t) = T_t \left( (\bar{u}, \bar{\mu}) \right)$ and $(u_{\Dx}, \mu_{\Dx})(t) = T_t \circ P_{\Dx} \left( (\bar{u}, \bar{\mu}) \right)$, respectively, then
	\begin{align}\label{eq:consu}
		\|u(t) - u_{\Dx}(t)\|_{\infty} &\leq \sqrt{\big(1 + 2\sqrt{2}\big)} \bar{F}_{\infty}^{\nicefrac{3}{4}} t^{\nicefrac{1}{2}}\Dx^{\nicefrac{1}{4}} \nonumber
		\\ &\quad + \Big(\! 1 + 2\sqrt{2} + \sqrt{\big(2+\frac{1}{2}t^2\big)}\Big)\!\bar{F}_{\infty}^{\nicefrac{1}{2}}\Dx^{\nicefrac{1}{2}} + t\Dx.  
	\end{align}
\end{lemma}

\begin{proof}
	 Note that $V(t, \xi) = \bar{V}(\xi)$ and $V_{\Dx}(t, \xi) \!= \bar{V}_{\Dx}(\xi)$ for all $\xi \in \R$ and $t \geq 0$. Thanks to Definition~\ref{def:ProjOP}, this implies $V_{\infty}(t) = \bar{V}_{\infty} = \bar{V}_{\Dx, \infty} = V_{\Dx, \infty}(t)$. It therefore follows from \eqref{eq:LagrSystem} and Lemma~\ref{lem:convInt} that 
	\begin{align*}
		|U(t, \xi) - U_{\Dx}(t, \xi)| &\leq |\bar{U}(\xi) - \bar{U}_{\Dx}(\xi)| 
		\\ & \qquad + \frac{1}{2}\int_0^t\! \!\left(\!V(s, \xi) - V_{\Dx}(s, \xi) - \frac{1}{2}\big(V_{\infty}(s) - V_{\Dx, \infty}(s)\big)\!\right)\!ds \nonumber
		\\ & \leq \|\bar{U} - \bar{U}_{\Dx}\|_{\infty} + \frac{1}{2}t \|\bar{V} - \bar{V}_{\Dx}\|_{\infty} \nonumber
		\\ & \leq (1+2\sqrt{2})\bar{F}_{\mathrm{ac}, \infty}^{\nicefrac{1}{2}}\Dx^{\nicefrac{1}{2}} + t\Dx,
\end{align*}
and from here on we can argue precisely as in the proof of Theorem~\ref{thm:ConvRateu}. \end{proof}

Next we address the issue of deriving a rate of convergence for the family $\{\mu_{\Dx}(t)\}_{\Dx > 0}$ of numerical energy measures. Recall that for any two nonnegative and finite Borel measures $\lambda$ and $\nu$,  we define the {\em bounded Lipschitz} metric (also known as  {\em Fortet--Mourier} metric) by 
\begin{equation}\label{eq:boundedLipschitz}
	d_{\mathrm{BL}}(\lambda, \nu) := \sup \Big \{\int_{\R}\psi(x)\big(d\lambda(x) - d\nu(x)\big)\!: \psi \in \mathrm{BL}(\R), \hspace{0.05cm} \|\psi\|_{\mathrm{BL}} \leq 1 \Big \},
\end{equation}
where $\mathrm{BL}(\R)$ denotes the space of bounded Lipschitz functions on $\R$ equipped with the norm
\begin{equation}\label{eq:BLnorm}
	\|f\|_{\mathrm{BL}} := \sup_{x \in \R}| f(x)| + \sup_{x \neq y} \frac{ \left| f(x) - f(y) \right|}{|x-y|} = \|f\|_{\infty} + \|f\|_{\mathrm{Lip}}. 
\end{equation}
The interested reader is referred to e.g., \cite[Chp. 3]{BogachevBook} or \cite[Chp. 8]{BogachevMeasure} for a detailed exposition on this metric. 

\begin{lemma}\label{lem:convRateMu}
	Given $(\bar{u}, \bar{\mu})\! \in \D$, $\alpha=0$, and $t \geq 0$, let $(u, \mu)(t) = T_t((\bar{u}, \bar{\mu}))$ and $(u_{\Dx}, \mu_{\Dx})(t) = T_t \circ P_{\Dx}((\bar{u}, \bar{\mu}))$, then 
	\begin{align}
		d_{\mathrm{BL}}(\mu(t), \mu_{\Dx}(t)) &\leq (1+2\sqrt{2})t\bar{F}_{\infty}^{\nicefrac{3}{2}}\Dx^{\nicefrac{1}{2}} + \Big(2+\frac{1}{2}t^2+ 4e^{\frac{1}{2}t}\Big)\bar{F}_{\infty}\Dx. \label{eq:convmu}
	\end{align}
\end{lemma}

\begin{proof}
Let $(y, U, V)(t) = S_t \circ L((\bar{u}, \bar{\mu}))$ and $(y_{\Dx}, U_{\Dx}, V_{\Dx})(t) = S_t \circ L \circ P_{\Dx}((\bar{u}, \bar{\mu}))$, and pick any $\psi \in \mathrm{BL}(\R)$ with $\|\psi \|_{\mathrm{BL}} \leq 1$. It follows from \eqref{eq:MapM} that 
\begin{subequations}
\begin{align}
	\left|\int_{\R} \psi(x)\big(d\mu(t) - d\mu_{\Dx}(t) \big) \right|\! \! &= \left |\int_{\R} \psi(y(t, \xi)) V_{\xi}(t, \xi) d\xi - \int_{\R}\psi(y_{\Dx}(t, \xi))V_{\Dx, \xi}(t, \xi)d\xi  \right| \nonumber
	\\ &\leq \int_{\R} \left|\psi(y(t, \xi)) - \psi(y_{\Dx}(t, \xi)) \right|\!V_{\xi}(t, \xi)d\xi \label{eq:convF1}
	\\ &\qquad + \left|\int_{\R}\psi(y_{\Dx}(t, \xi))\left (V_{\xi}(t, \xi) - V_{\Dx, \xi}(t, \xi) \right)\!d\xi \right|\!  \label{eq:convF2},
\end{align}
\end{subequations}
and since $\|\psi\|_{\mathrm{Lip}} \leq 1$ and $V_{\xi}(t, \xi) = \bar{V}_{\xi}(\xi)$, we have the following estimate for \eqref{eq:convF1}, \vspace{-0.05cm}
\begin{align}\label{eq:est_convF1}
	\int_{\R} \big|\psi(y(t, \xi)) - \psi(y_{\Dx}(t, \xi)) \big|V_{\xi}(t, \xi)d\xi &\leq \int_{\R} \left| y(t, \xi) - y_{\Dx}(t, \xi) \right|\!\bar{V}_{\xi}(\xi)d\xi \nonumber 
	\\ & \leq \|y(t) - y_{\Dx}(t) \|_{\infty}\bar{F}_{\infty}. 
\end{align}
For the term \eqref{eq:convF2}, we integrate by parts, which yields \vspace{-0.075cm}
\begin{align}\label{eq:convF2_temp}
	\biggl|\int_{\R}\psi(y_{\Dx}(t, \xi))\bigl(V_{\xi}(t, \xi) &- V_{\Dx, \xi}(t, \xi) \bigr)d\xi \biggr| \nonumber
	\\ &= \left|\int_{\R} \psi'(y_{\Dx}(t, \xi)) y_{\Dx, \xi}(t, \xi) \big(\bar{V}(\xi) - \bar{V}_{\Dx}(\xi)\big)d\xi \right|\!,
\end{align}
because the boundary terms vanish as Definition~\ref{def:ProjOP} implies that left and right asymptotes of $\bar{V}$ and $\bar{V}_{\Dx}$ coincide. Moreover, by following the proof of \cite[(2.15)]{AlphaHS}, we find
\begin{equation*}
	y_{\Dx, \xi}(t, \xi) + \bar{V}_{\Dx, \xi}(\xi) \leq \left(\bar{y}_{\Dx, \xi}(\xi) + \bar{V}_{\Dx, \xi}(\xi) \right)e^{\frac{1}{2}t},
\end{equation*}
which together with Lemma~\ref{lem:convInt} and the fact that $\bar{V}$ and $\bar{V}_{\Dx}$ are increasing and satisfy $\bar{V}_{\Dx}(\xi_{3j}) =  \bar{F}(x_{2j}) =\bar{V}(\xi_{3j})$ for all $j \in \mathbb{Z}$ by Definition~\ref{def:ProjOP}, implies the following upper bound on \eqref{eq:convF2_temp},   \vspace{-0.075cm}
\begin{align*}
	\int_{\R} \left |\psi'(y_{\Dx}(t, \xi)) \right | &y_{\Dx, \xi}(t, \xi)\! \left |\bar{V}(\xi) - \bar{V}_{\Dx}(\xi) \right|\!d\xi \nonumber
	\\ & \leq e^{\frac{1}{2}t} \int_{\R} \left |\bar{V}(\xi) - \bar{V}_{\Dx}(\xi) \right| \!\bar{V}_{\Dx, \xi}(\xi) d\xi \nonumber
	\\ & \qquad +  e^{\frac{1}{2}t} \int_{\R}  \left |\bar{V}(\xi) - \bar{V}_{\Dx}(\xi) \right| \!\bar{y}_{\Dx, \xi}(\xi)d\xi \nonumber
	\\ & \leq 2e^{\frac{1}{2}t} \bar{F}_{\infty} \Dx + e^{\frac{1}{2}t}\sum_{j \in \mathbb{Z}} \int_{\xi_{3j}}^{\xi_{3j+3}} \!\left |\bar{V}( \xi_{3j+3}) - \bar{V}(\xi_{3j}) \right|\bar{y}_{\Dx, \xi}(\xi)d\xi \nonumber
	\\ & \leq 2e^{\frac{1}{2}t} \bar{F}_{\infty} \Dx+ 2e^{\frac{1}{2}t}\sum_{j \in \mathbb{Z}} \! \left(\bar{F}(x_{2j+2}) - \bar{F}(x_{2j})\right)\! \Dx \nonumber
	\\ & \leq 4 e^{\frac{1}{2}t} \bar{F}_{\infty} \Dx. 
\end{align*}
At last, combining this estimate with Lemma~\ref{lem:RateConservative} and  \eqref{eq:est_convF1}, and thereafter taking the supremum over all functions in $\mathrm{BL}(\R)$ satisfying $\|\psi\|_{\mathrm{BL}}\leq 1$, gives \eqref{eq:convmu}. 
\end{proof}

\begin{remark}
	If we impose a first order moment condition on $\bar{\mu}$, i.e., assume that $\bar{\mu} \in \M_1^+(\R)\!:=\{\lambda \in \M^+(\R)\!: \int_{\R}|x|d\lambda < \infty \}$, then one can argue in precisely the same way using the 1-Wasserstein distance. This distance, which is defined for any two measures $\lambda$ and $\nu$ that belong to $\M_1^+(\R)$ and have the same total mass, can be expressed in the following two equivalent ways, see \cite[Sec. 1 and 7]{OptimalTransport}, 
\begin{align*}
	W_1(\lambda, \nu) := \sup_{ \|\psi \|_{\mathrm{Lip}} \leq 1 } \left \{ \int_{\R} \psi(x)\big(d\lambda(x) - d\nu(x) \big) \right \}\! = \|F_{\lambda} - F_{\nu}\|_1. 
\end{align*}
Here $\|\cdot \|_{\mathrm{Lip}}$ denotes the Lipschitz seminorm introduced in \eqref{eq:BLnorm} and $F_{\lambda}(x) = \lambda((-\infty, x))$ and $F_{\nu}(x) = \nu((\infty, x))$ are the cumulative functions associated with $\lambda$ and $\nu$, respectively.

This metric can be applied in the case of conservative solutions, because \eqref{eq:Proj_F_Lp} implies $\bar{\mu}_{\Dx}\in \M_1^+(\R)$, and the proof of \cite[Thm. 2.11]{ConsMetric} reveals that both $\mu(t)$ and $\mu_{\Dx}(t)$ carry the first moment forward in time, i.e., they belong to $\M_1^{+}(\R)$ for all $t\geq 0$. Moreover, Definition~\ref{def:ProjOP} guarantees that $\bar{\mu}_{\Dx}(\R) = \bar{\mu}(\R)$ and $\mu(t, \R) = \mu_{\Dx}(t, \R)$ when $\alpha=0$. As a consequence, we have 
	\begin{equation*}
		\|F(t) - F_{\Dx}(t)\|_1\leq (1+2\sqrt{2})t\bar{F}_{\infty}^{\nicefrac{3}{2}}\Dx^{\nicefrac{1}{2}} + \Big(2+\frac{1}{2}t^2+ 4e^{\frac{1}{2}t}\Big)\bar{F}_{\infty}\Dx,
	\end{equation*}
	whenever $\bar{\mu} \in \M_1^+(\R)$ and $\alpha = 0$. 
\end{remark}

In the nonconservative case ($\alpha \neq 0$), there is no reason to expect that $\mu(t)$ and $\mu_{\Dx}(t)$ have the same total mass for a fixed $t > 0$, because we lack sufficient control over the size of the following sets 
\begin{equation*}
	\left \{x\!: \bar{u}_x(x) \leq -\frac{2}{t} < \bar{u}_{\Dx, x}(x) \right\} \bigcup \left \{x\!: \bar{u}_{\Dx, x}(x) \leq -\frac{2}{t} < \bar{u}_x(x)  \right \} \!. 
\end{equation*}
The former set contains the points in Eulerian coordinates for which the exact solution breaks within  $[0, t]$, whereas the numerical solution either breaks at some later time or does not break at all. The second set consists of those points for which the opposite holds. Note that the $L^2$-error bound on $u_{\Dx, x}$ in Lemma~\ref{lem:Rateux} provides some control on the difference in measure of these sets, as we have by the Chebychev inequality
\begin{equation*}
	\Big| \big \{x: |\bar{u}_x(x) - \bar{u}_{\Dx, x}| \geq \Delta x^{\gamma} \big \} \Big| \leq \Dx^{-2\gamma} \|\bar{u}_x - \bar{u}_{\Dx, x}\|_2 \leq D \Dx^{\nicefrac{(\beta - 4 \gamma)}{2}}, 
\end{equation*}
which tends to zero as $\Dx \rightarrow 0$ if $\gamma < \frac{\beta}{4}$. Yet, one is unable to use this for proving that $\{\mu_{\Dx}(t)\}_{\Dx > 0}$ converges weakly or even vaguely for every fixed $t > 0$.  The reason is that one needs to control $\|V_{\xi}(t) - V_{\Dx, \xi}(t)\|_p$ or $\|V(t) - V_{\Dx}(t)\|_p$, for some $p \in [1, \infty]$, and this is hard. On the contrary, when one derives a convergence rate for $\{u_{\Dx}(t)\}_{\Dx > 0}$ one needs an error estimate for $\{U_{\Dx}(t)\}_{\Dx > 0}$, in which case one only needs to control the time integral appearing in \eqref{eq:estU}. 

\begin{remark}\label{rem:otherConservative}
	A fully discrete numerical method for conservative solutions was proposed in \cite{NumericalConservative}. This method is based on successively evolving the numerical solution time steps of length $\Delta t$ forward along characteristics ($\Delta t$ is related to $\Dx$ through a CFL-like condition), and after each time increment a piecewise linear projection operator is applied. Let $(\tilde{u}_{\Dx}, \tilde{F}_{\Dx})$ denote the resulting numerical solution computed this way on a fixed mesh. Then it is shown that
\begin{equation*}
		\sup_{t \in [0, T]} \left( \|\!\left(\tilde{u}_{\Dx} - u\right)\!(t)\|_{\infty} + \|\! \left(\tilde{F}_{\Dx} - F\right)\!(t)\|_{\infty} \right) \!\leq \mathcal{O}(\Dx^{\nicefrac{1}{2}}),
	\end{equation*}
	provided $(u, F) \in [W^{1, \infty}([0, T] \times \R)]^2$, i.e., before wave breaking takes place, see \cite[Thm. 3]{NumericalConservative}. This result can be improved for our numerical method. In particular, if $(u, F) \in [W^{1, \infty}([0, T] \times \R)]^2$, then 
\begin{equation}\label{eq:equalityGridpoints}
	u_{\Dx}(t, y_{\Dx}(t, \xi_{3j}))\! = \!u(t, y(t, \xi_{3j}))  \hspace{0.25cm} \text{and} \hspace{0.25cm} F_{\Dx}(t, y_{\Dx}(t, \xi_{3j}))\! = \!F(t, y(t, \xi_{3j})), 
\end{equation}
for all $j \in \mathbb{Z}$. Moreover, since we only apply $P_{\Dx}$ from Definition~\ref{def:ProjOP} once, following the proof of  \cite[Thm. 3]{NumericalConservative} and using \eqref{eq:equalityGridpoints} reveals that our method satisfies 
\begin{equation}\label{eq:convRateLip}
	\sup_{t \in [0, T]} \big( \|(u_{\Dx} - u)(t)\|_{\infty} + \|(F_{\Dx} - F)(t)\|_{\infty}\big) \!\leq \mathcal{O}(\Dx). 
\end{equation}
	In practice, this means that our numerical method converges in the case of Lipschitz initial data with order $\mathcal{O}(\Dx)$ prior to wave breaking. This is also true when $\alpha \neq 0$, provided the numerical solutions do not experience wave breaking. 
\end{remark} 

\vspace{-0.3cm}
\section{Explicit examples with convergence rates}\label{sec:ExplicitRates}
We complement the previous three sections by examining three explicit examples that satisfy the conditions of Theorem~\ref{thm:ConvRateu}.  They are inspired by the examples explored numerically in \cite[Sec. 5]{AlphaAlgorithm}. For each example, we perform one or several tests of the numerical convergence rate. More precisely, we compute a sequence $\{(u_{\Dx_k}, F_{\Dx_k})\}_{k}$ of numerical approximations,
where
\begin{equation*}
	\Dx_k = 4^{-k} \hspace{0.35cm} k \geq 1, 
\end{equation*}
and thereafter compute the relative errors 
\begin{equation*}
	\mathrm{Err}_{k}(T) := \sup_{t \in [0, T]} \frac{ \|u(t) - u_{\Dx_k}(t)\|_{\infty}}{\|u(t)\|_{\infty}},
\end{equation*}
and the experimental orders of convergence (EOCs)
\begin{equation*}
	\mathrm{EOC}_k(T) := \frac{\ln ( \nicefrac{\mathrm{Err}_{k-1}(T)}{\mathrm{Err}_{k}(T)})}{\ln (\nicefrac{\Dx_{k-1}}{\Dx_k} )}. 
\end{equation*}

\begin{example}[Multipeakon data]\label{ex:Multipeakon}
Firstly, let us consider a general multipeakon, i.e., a piecewise linear and continuous wave profile $\bar{u}$. This constitutes an important family of initial data, because the piecewise linear nature of the data is naturally preserved by the HS equation. Moreover, it is also one of the few classes of solutions for where explicit expressions can be computed.

 Let $\{a_j\}_{j=1}^{N-1}$, $\{b_j\}_{j=0}^N$, and $\{z_j\}_{j=0}^{N+1}$ be real-valued sequences with $ z_k < z_{k+1}$ for $k \in \{0, .., N\}$, and consider the multipeakon
\begin{align}\label{eq:multipeakonInitial}
	\bar{u}(x) &= \begin{cases}
	b_0, & x < z_0, \\
	\vdots  \\
	a_n x + b_n, & z_n \leq x < z_{n+1}, \\
	a_{n+1}x + b_{n+1}, & z_{n+1} \leq x < z_{n+2}, \\
	\vdots \\
	b_N, &  z_{N+1} \leq x. 
	\end{cases}
\end{align}
The derivative is a bounded, piecewise constant function with finitely many jumps, and its total variation is therefore the sum of the absolute value of these jumps, namely \vspace{-0.15cm}
\begin{equation*}
	TV(\bar{u}_x) = |a_1| + \sum_{n=1}^{N-2}|a_{n+1} - a_{n}| + |a_{N-1}| < \infty. 
\end{equation*}
Consequently, for any $h \in \R$, \vspace{-0.05cm}
\begin{align*}
	\int_{\R} \left(\bar{u}_x(x) - \bar{u}_x(x-h) \right)^2dx &\leq 2 \sup_{j \in \{1, ..., N-1 \}}|a_j| \int_{\R} \left|\bar{u}_x(x) - \bar{u}_x(x-h) \right|dx \\
	&= 2  \|\bar{u}_x \|_{\infty} \sum_{j \in \mathbb{Z}} \int_{jh}^{(j+1)h} \left|\bar{u}_x(x) - \bar{u}_x(x-h) \right|dx \\ 
	&= 2  \|\bar{u}_x \|_{\infty} \int_0^h \sum_{j \in \mathbb{Z}} \left| \bar{u}_x(z+jh) - \bar{u}_x(z+(j-1)h) \right|dz
	\\ &  \leq 2 \|\bar{u}_x \|_{\infty} TV(\bar{u}_x) h, 
\end{align*}
or, in other words, $\bar{u}_x$ belongs to $B_2^{\nicefrac{1}{2}}(\R)$ for any multipeakon data. In fact, this is a particular case of the more general inclusion  \vspace{-0.075cm}
\begin{equation}\label{eq:BVinclusion}
	\mathrm{BV}(\R) \cap L^1(\R) \subset B_2^{\nicefrac{1}{2}}(\R).
\end{equation}
As a consequence, Theorem~\ref{thm:ConvRateu} implies, that for any $(\bar{u}, \bar{\mu}) \in \D$ with $\bar{u}_x \in \mathrm{BV}(\R) \cap L^1(\R)$ and $\bar{\mu}_{\mathrm{sc}} = 0$, we have
\begin{align}\label{eq:rateBV}
	\|u(t) - u_{\Dx}(t) \|_{\infty} \leq \mathcal{O}(\Dx^{\nicefrac{1}{16}}),  
\end{align}
where $(u, \mu)(t) = T_t((\bar{u}, \bar{\mu}))$ and $(u_{\Dx}, \mu_{\Dx})(t) = T_{t} \circ P_{\Dx}(((\bar{u}, \bar{\mu}))$. Moreover, if $\bar{u}$ is given by \eqref{eq:multipeakonInitial} and $a_j \geq 0$ for all $j \in \{1, ..., N-1\}$ with $d\bar{\mu} = \bar{u}_x^2dx$, then no wave breaking takes place for any $t\geq 0$, such that the conservative and $\alpha$-dissipative solution coincide, and \eqref{eq:convRateLip} gives the improved bound,  $\|u(t) - u_{\Dx}(t) \|_{\infty} \leq \mathcal{O}(\Dx)$. 

We have revisited Example~\ref{ex:simpleBreaking} in Figure \ref{fig:testPeakon}, where the exact solution $(u, F)$ is compared to two numerical approximations, $(u_{\Dx_j}, F_{\Dx_j})$ for $j \in \{c, f\}$, with $\Dx_c = \frac{1}{4}$ and $\Dx_f = 10^{-3}$. The solutions are compared at the times $t=0$, $2$, and $4$.  Furthermore, Table~\ref{tab:EOCMult} displays the associated relative errors, which are computed before wave breaking with $T=\nicefrac{3}{2}$ (top row) and after wave breaking $T=3$ (bottom row) when $\alpha = \nicefrac{1}{2}$. We have not included the corresponding EOCs for this example, because machine precision is nearly attained even for $k=1$, such that round-off errors are predominant, i.e, they perturb the computed EOCs.  

\begin{figure}
	\includegraphics[scale=0.9]{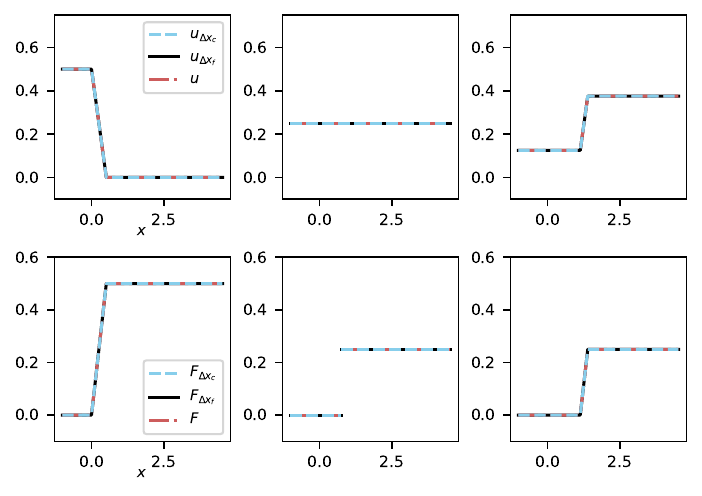}
	\captionsetup{width=.925\linewidth}
	\vspace{-2ex}
	\caption{A comparison of $u$ (top row) and $F$ (bottom row), both with dotted red lines, to that of $u_{\Dx_j}$ (top row) and $F_{\Dx_j}$ (bottom row) for $\Dx_c = \nicefrac{1}{4}$ (blue dashed) and $\Dx_f=10^{-3}$ (black solid) for the multipeakon from Example~\ref{ex:simpleBreaking}. The solutions are compared from left to right at $t=0$, $2$, and $4$ with $\alpha=\nicefrac{1}{2}$.}
	\label{fig:testPeakon}
\end{figure}	

 \begin{table}
  \setlength\tabcolsep{4.0pt}
\begin{tabular}{c|cccccc} 
 \hline
 $k$  & $1$ & $2$ & $3$ & $4$ & $5$ & $6$ \\ 
 \hline
  $\mathrm{Err}_k(\nicefrac{3}{2})$  & $2.3 \!\cdot \! 10^{-15}$ & $5.1 \!\cdot \!10^{-15}$ & $2.0 \! \cdot \!10^{-14}$ & $8.5 \! \cdot \!10^{-14}$ & $3.1 \!\cdot \!10^{-13}$ & $1.3 \! \cdot \!10^{-12}$ \\ 
 \hline 
 $\mathrm{Err}_k(3)$  & $2.0 \!\cdot \!10^{-14}$ & $2.5 \!\cdot \!10^{-14}$ & $7.3 \!\cdot \!10^{-14}$ & $3.8 \!\cdot \!10^{-13}$ & $1.5 \!\cdot \!10^{-12}$ & $5.5\! \cdot \!10^{-11}$ \\ 
\end{tabular}
	\captionsetup{width=.925\linewidth}
	\vspace{-0.35cm}
\caption{Computed relative errors for the multipeakon from Example~\ref{ex:simpleBreaking} when $\alpha = \nicefrac{1}{2}$. These errors are computed before wave breaking takes place, with $T=\nicefrac{3}{2}$ (top row), and after wave breaking has occurred, with $T=3$ (bottom row).}
\label{tab:EOCMult}
\end{table}

However, we cannot expect to attain machine precision for a general multipeakon, because in Example~\ref{ex:simpleBreaking} the uniform mesh aligned perfectly with the nodes situated at $x=0$ and $x=\nicefrac{1}{2}$. Thus, if we were to replace the chosen sequence of mesh parameters with some other sequence, say $\widetilde{\Dx}_k = 3^{-k}$, then the errors would become much worse. Alternatively, if we modify the initial data slightly and instead consider 
\begin{align}\label{eq:worsePeakon}
	\bar{u}(x) &= \begin{cases}
	\frac{1}{3}, & x < 0, \\
	\frac{1}{3}-x, & 0 \leq x \leq \frac{1}{3}, \\
	0, & \frac{1}{3} < x, 
	\end{cases} \nonumber \\
	d\bar{\mu} &= \bar{u}_x^2dx,
\end{align}
again with $\alpha=\nicefrac{1}{2}$, then the relative errors also become much worse as displayed in Table~\ref{tab:EOCMultWorse}. This happens even if the corresponding solution, $(u, F)(t)$, behaves completely similar to the previous one. The main difference now is that one node is located at $x=\nicefrac{1}{3}$. 

Table~\ref{tab:EOCMultWorse} also displays the computed EOCs. It reveals, in particular, that $\mathrm{EOC}_k(\nicefrac{3}{2})$ is approximately equal to one for each $k=1, ... 6$, which agrees with the predicted $\mathcal{O}(\Dx)$-rate from \eqref{eq:convRateLip}, because no wave breaking has occured for $t \leq \nicefrac{3}{2}$. Moreover, the table also displays $\mathrm{EOC}_k(3)$, which also is close to one for each $k=1, ... 6$. This on the other hand, is much better than the $\mathcal{O}(\Dx^{\nicefrac{1}{16}})$-rate predicted by \eqref{eq:rateBV}.

In other words, the numerical method performs really well for multipeakon data, which is to be expected, because the applied projection operator is piecewise linear. 
 \begin{table}
  \setlength\tabcolsep{3.5pt}
\begin{tabular}{c|cccccc} 
 \hline
 $k$  & $1$ & $2$ & $3$ & $4$ & $5$ & $6$ \\ 
 \hline
  $\mathrm{Err}_k(\nicefrac{3}{2})$  & $1.82\! \cdot \! 10^{-1}$ & $4.51\! \cdot\! 10^{-2}$ & $1.10 \! \cdot \!10^{-2}$ & $2.63 \! \cdot \!10^{-3}$ & $6.47 \!\cdot \!10^{-4}$ & $1.62 \! \cdot \!10^{-4}$ \\ 
  $\mathrm{EOC}_k(\nicefrac{3}{2})$ & $-$ & $1.01$ & $1.02$ & $1.03$ & $1.01$ & $1.00$ \\ 
  \hline 
 $\mathrm{Err}_k(3)$  & $2.09 \!\cdot \!10^{-1}$ & $5.18 \!\cdot \!10^{-2}$ & $1.28 \!\cdot \!10^{-2}$ & $3.63 \!\cdot \!10^{-3}$ & $9.08 \!\cdot \!10^{-4}$ & $2.27\! \cdot \!10^{-4}$ \\ 
 $\mathrm{EOC}_k(3)$ & $-$ & $1.01$ & $1.01$ & $0.91$ & $1.00$ & $1.00$
\end{tabular}
	\captionsetup{width=.925\linewidth}
	\vspace{-0.35cm}
\caption{Computed relative errors and EOCs for the multipeakon with initial data \eqref{eq:worsePeakon} in the case of $\alpha = \nicefrac{1}{2}$. The two upper rows display the computed relative errors and EOCs for $T=\nicefrac{3}{2}$, that is, before any wave breaking takes place, while the two last rows show the same quantities computed with $T=3$.}
\label{tab:EOCMultWorse}
\end{table}

  \vspace{-0.15cm}

\end{example}

\begin{example}[Cosine data]\label{ex:cosinus}
	Let $\bar{\mu} \in \M^+(\R)$ with $\bar{\mu}_{\mathrm{sc}} =0$ and 
	\begin{align*}
		\bar{u}(x) &= \begin{cases}
		1, & x < 0, \\
		\cos(\pi x), & 0 \leq x \leq 4, \\
		1, & 4 < x.
		\end{cases}
	\end{align*}
	The exact solution to \eqref{eq:HS} with initial data $(\bar{u}, \bar{\mu})$ experiences wave breaking continuously over the time interval $[\frac{2}{\pi}, \infty)$, and there is only a finite number of characteristics experiencing wave breaking at each time, see \cite[Sec. 5]{AlphaAlgorithm} for further details. Despite this, Theorem~\ref{thm:ConvRateu} implies \vspace{-0.1cm}
\begin{equation*}
		\|u(t) - u_{\Dx}(t)\|_{\infty} \leq \mathcal{O}(\Dx^{\nicefrac{1}{8}}) \hspace{0.35cm} \text{for all } t > 0, 
\end{equation*}
because $\bar{u}_x$ belongs to $C^{0, 1}_c \subset B_2^1$. 
	
	Figure~\ref{fig:testCosine} displays the time evolution for the associated solution (whose expression is implicitly known in Lagrangian coordinates, and one has to numerically invert $y$ to compute the Eulerian expression) when $d\bar{\mu} = \bar{u}_x^2dx$ and $\alpha = \nicefrac{3}{4}$. This solution is compared to two numerical approximations, $(u_{\Dx_j}, F_{\Dx_j})$ for $j \in \{c, f\}$ with $\Dx_c$ and $\Dx_f$ as in the previous example. Moreover, Table~\ref{tab:EOCCosine} shows the corresponding relative errors and EOCs for $T=\nicefrac{3}{5}$ (before any wave breaking has occurred) and $T=\nicefrac{6}{5}$ (after wave breaking has started to take place). The relative errors are approximated with the right-hand side of \eqref{eq:est_u}, but as the Lagrangian solution is only known implicitly, one has to compute a numerical inverse in order to obtain an explicit expression. Unfortunately, this introduces an additional error which is limited to 6 digits of accuracy, such that it starts to influence the computed EOCs for $k\geq 6$, because then this error is of the same order of magnitude as the approximation error (due to square root in \eqref{eq:est_u}). Yet, one still observes that all the EOCs for $T = \nicefrac{3}{5}$, except the ones for $k \geq 6$, are roughly equal to $1$ which agrees with \eqref{eq:convRateLip}. 
		
\begin{figure}
	\includegraphics[scale=0.9]{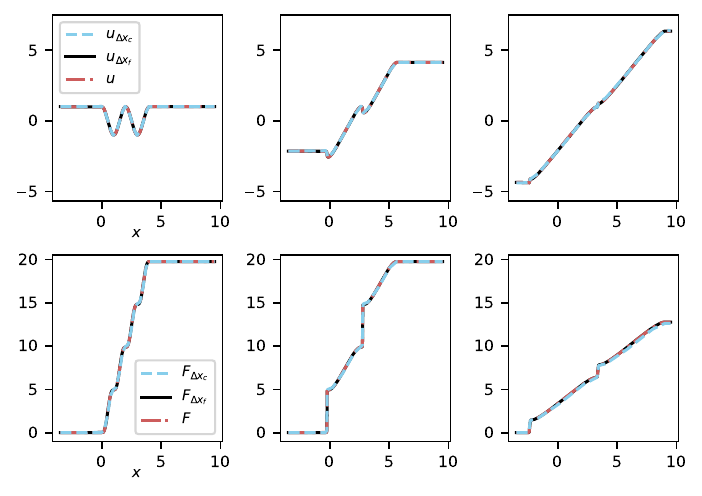}
	\captionsetup{width=.925\linewidth}
	\vspace{-2ex}
	\caption{A comparison of $u$ (top row) and $F$ (bottom row), both with dotted red lines, to that of $u_{\Dx_j}$ (top row) and $F_{\Dx_j}$ (bottom row) for $\Dx_c = \nicefrac{1}{4}$ (blue dashed) and $\Dx_f=10^{-3}$ (black solid) for Example~\ref{ex:cosinus}. The solutions are compared from left to right at $t=0$, $\frac{2}{\pi}$, and $\frac{4}{\pi}$, with $\alpha=\nicefrac{3}{4}$. }
	\label{fig:testCosine}
\end{figure}
	
	\begin{table}
	  \setlength\tabcolsep{4.5pt}
\begin{tabular}{c|ccccccc} 
 \hline
 $k$  & $1$ & $2$ & $3$ & $4$ & $5$ & $6$ & $7$  \\ 
 \hline
  $\mathrm{Err}_k(\nicefrac{3}{5})$  & $1.14$ & $0.24$ & $0.06$ & $0.016$ & $3.96\!\cdot\! 10^{-3}$ & $1.91\! \cdot \!10^{-3}$ & $1.79\!\cdot\! 10^{-3}$  \\ 
  $\mathrm{EOC}_k(\nicefrac{3}{5})$  & $-$ & $1.12$ & $0.97$ & $1.00$ & $1.00$ & $0.53$ & $0.05$\\
   \hline
   $\mathrm{Err}_k(\nicefrac{6}{5})$  & $1.20$ & $0.29$ & $0.071$ & $0.018$ & $4.44 \!\cdot \!10^{-3}$ & $2.70\! \cdot \!10^{-3}$  & $2.54\! \cdot \! 10^{-3}$\\ 
  $\mathrm{EOC}_k(\nicefrac{6}{5})$  & $-$ & $1.03$ & $1.00$ & $1.00$ & $1.00$ & $0.36$ & $0.04$
\end{tabular}
	\captionsetup{width=.925\linewidth}
	\vspace{-0.35cm}
\caption{Computed relative errors and experimental convergence orders for Example~\ref{ex:cosinus}. The two upper rows display the relative errors and experimental orders of convergence before any wave breaking takes place, while the other two rows show the same quantities computed with $T=\nicefrac{6}{5}$.}
\label{tab:EOCCosine}
\end{table}	
\end{example}

In the previous example, infinitesimal energy concentrations took place for each $t \in [\frac{2}{\pi}, \infty)$, but what about the possibility of an infinitesimal amount of energy concentrating initially, i.e., at $t=0$, in which case $\bar{u}_x \notin L^{\infty}(\R)$? This phenomenon has been observed for cusped initial data, see e.g., \cite[Sec. 5]{AlphaAlgorithm} or \cite[Sec. 3]{NumericalConservative}. 

Since $C_c^{0, \beta}\subset L^{\infty}(\R)$ for all $\beta \in (0, 1]$, the case $\bar{u}_x \in C_c^{0, \beta}$ excludes wave breaking times from accumulating initially, the same can be said for $\bar{u}_x \in \mathrm{BV}(\R)$, but it can occur when $\bar{u}_x \in B_2^{\beta}$. Indeed, introduce the Besov space
	\begin{equation*}
		\tilde{B}_2^{\beta}(\R)\! := \big\{f \in L^2(\R)\!: \!\exists C > 0 \text{ s.t. } \|\delta_hf\|_2 \leq Ch^{\beta} \text{ for all }\! h \geq 0 \big\}, 
	\end{equation*}
	then, clearly $\tilde{B}_2^{\beta}(\R) \subseteq B_2^{\beta}$. Moreover, the following imbedding property was proved in \cite[Corollary 33]{Nikolskii}, 
\begin{equation}\label{eq:imbedding}
	\tilde{B}_2^{\beta}(\R) \subset L^q(\R), 
\end{equation}
 with $q \in [2, \infty]$ for $\beta >\frac{1}{2}$, and $q \in [2, p^*)$ in the case $\beta \leq \frac{1}{2}$, where 
 \begin{equation}\label{eq:criticalParameter}
 	p^* = \frac{2}{1-2\beta}.
 \end{equation}
Again, $\tilde{B}_2^{\beta}(\R) \subset L^{\infty}(\R)$ whenever $\beta > \frac{1}{2}$. However, infinitesimal energy concentrations can happen initially if $\bar{u}_x \in \tilde{B}_2^{\beta}$ for $\beta \leq \frac{1}{2}$. In fact, the cusped initial data studied in \cite[Sec. 5]{AlphaAlgorithm} and \cite[Sec. 3]{NumericalConservative} belongs to $\tilde{B}_2^{\nicefrac{1}{6}}(\R)$, which we now prove. 

\begin{example}[Cusp data]\label{ex:Cusp}
Let $a, b \in \R$ with $a \leq b$ and consider the initial data 
\begin{align*}
	\bar{u}(x) &= \begin{cases}
		\left | a \right|^{\frac{2}{3}}\!, &x < a, \\
		 \left |x \right |^{\frac{2}{3}}\!, &a \leq x \leq b, \\
		 \left|b\right|^{\frac{2}{3}}\!,  &b < x, 
		 \end{cases}
\end{align*}
with derivative given by 
\begin{align*}
	\bar{u}_x(x) &= \begin{cases}
	0, & x < a, \\
	\frac{2}{3} \mathrm{sgn}(x)\left |x \right|^{-\frac{1}{3}}\!, & a \leq x \leq b, \\
	0, & b < x. \end{cases}
\end{align*}
By a direct calculation, one verifies that $\bar{u}_x \in L^p(\R)$ for $p \in (0, 3)$.  Hence, if $\bar{u}_x$ is to belong to a certain Besov space $\tilde{B}_2^{\beta}$, then \eqref{eq:imbedding} reveals that $p^* = 3$ provides an upper bound on its $q$-integrability, which in particular implies, due to \eqref{eq:criticalParameter}, that $\beta \leq \frac{1}{6}$. In other words,  the best we can hope to prove is $\bar{u}_x \in \tilde{B}_2^{\beta}(\R)$ for some $\beta \leq \frac{1}{6}$. Before analyzing this in more detail, let us consider two specific situations, for which we obtain a better convergence rate than that guaranteed by Theorem~\ref{thm:ConvRateu}. 
\begin{itemize}
	\item If $a > 0$, then $\bar{u}_x(x) \geq 0$ everywhere, and no wave breaking takes place in the future.  Moreover $\bar{u}_x \in L^{\infty}(\R)$ such that \eqref{eq:convRateLip} applies.
	\item If $b < 0$, then $\bar{u}_x$ is monotonically decreasing and bounded over $[a, b]$. It therefore belongs to $\mathrm{BV}(\R)$, and hence $\bar{u}_x \in B_2^{\nicefrac{1}{2}}$ by \eqref{eq:BVinclusion}. 
\end{itemize}

It remains to consider $a \leq 0 < b$ and $a < 0 \leq b$, in which case wave breaking  happens for each $t \in \big[0, 3|a|^{\nicefrac{1}{3}}\big]$.

Suppose $h>0$ and note that $\mathrm{supp}(\bar{u}_x(\cdot +h)) \subseteq [a-h, b-h]$. Hence, there is no overlap between the support of $\bar{u}_x(\cdot +h)$ and $\bar{u}_x$ when $h \geq b-a = b + |a|$. In this case, recall $\delta_hf(x) = f(x+h) - f(x)$, we therefore get
\begin{align*}
	\int_{\R} ( \delta_h\bar{u}_x)^2(x)dx &= \frac{4}{9}\int_{a-h}^{b-h}|x+h|^{\nicefrac{-2}{3}}dx + \frac{4}{9}\int_a^b|x|^{\nicefrac{-2}{3}}dx 
	\\ &= \frac{4}{3} \Big(|a|^{\nicefrac{1}{3}} + b^{\nicefrac{1}{3}}\Big) \leq \frac{8}{3}h^{\nicefrac{1}{3}}. 
\end{align*}

Assume from now on that $h < b-a$, and decompose $\|\delta_h\bar{u}_x\|_2^2$ in the following way 
\begin{align}
	\int_{\R} ( \delta_h\bar{u}_x)^2(x)dx &= \frac{4}{9}\int_{a-h}^{a} |x+h|^{\nicefrac{-2}{3}}dx + \frac{4}{9} \int_{b-h}^{b}|x|^{\nicefrac{-2}{3}}dx  \nonumber\\
	& \hspace{0.1cm} + \frac{4}{9}\int_{a}^{b-h} \left(\mathrm{sgn}(x+h)|x+h|^{\nicefrac{-1}{3}} - \mathrm{sgn}(x)|x|^{\nicefrac{-1}{3}} \right)^2dx \label{eq:cusp}. 
\end{align}
In order to estimate the two first integrals we exploit the following inequality. 

\begin{prop}\label{prop:ineq}
	Given $q \in (0, 1]$, then for any $c, d, \in \R$ one has
	\begin{equation*}
		\left||c|^{q} - |d|^{q} \right| \leq |c-d|^{q}.
	\end{equation*}
\end{prop}

\begin{proof}
	Let us start by defining the function $f(t) = (1+t)^q -1 -t^q$ and observe that $f(0) = 0$. Moreover, for any $t > 0$, we have $f'(t) = q ((1+t)^{q-1} - t^{q-1}) < 0$ such that $f(t) < 0$ for all $t \in (0, \infty)$. Pick any $\tilde{a}, \tilde{b} > 0$. Substituting $t=\frac{\tilde{b}}{\tilde{a}}$ yields
	\begin{equation*}
		0 > \Big(1+\frac{\tilde{b}}{\tilde{a}}\Big)^q - 1 - \Big(\frac{\tilde{b}}{\tilde{a}}\Big)^q = \Big(\frac{\tilde{a}+\tilde{b}}{\tilde{a}}\Big)^q - 1 - \Big(\frac{\tilde{b}}{\tilde{a}}\Big)^q,
	\end{equation*}
	which in turn implies
	\begin{equation}\label{eq:triangle}
		(\tilde{a} +\tilde{b})^q \leq \tilde{a}^q + \tilde{b}^q \hspace{0.35cm} \text{for all } \tilde{a}, \tilde{b} \geq 0.
	\end{equation}
	Pick any $c, d \in \R$. By \eqref{eq:triangle} and the fact that $|\cdot|^q$ is increasing, we obtain
	\begin{equation*}
		|c|^q =  ||c-d+d||^q \leq (|c-d| + |d|)^q \leq |c-d|^q + |d|^q
	\end{equation*}
	In other words, $|c|^q - |d|^q \leq |c-d|^q$, by arguing similarly with the roles of $c$ and $d$ reversed, the inequality follows.
\end{proof}

By using Proposition~\ref{prop:ineq} we find that 
\begin{align*}
	\int_{a-h}^{a} |x+h|^{\nicefrac{-2}{3}}dx = 3(|a|^{\nicefrac{1}{3}} + \mathrm{sgn}(a+h)|a+h|^{\nicefrac{1}{3}}) \leq \begin{cases} 
	3h^{\nicefrac{1}{3}}, & |a| \geq h, \\
	6h^{\nicefrac{1}{3}}, & |a| < h, 
	\end{cases}
\end{align*}
and 
\begin{align*}
	 \int_{b-h}^{b}|x|^{\nicefrac{-2}{3}}dx =  3(|b|^{\nicefrac{1}{3}} - \mathrm{sgn}(b-h)|b-h|^{\nicefrac{1}{3}}) &\leq \begin{cases}
	3h^{\nicefrac{1}{3}}, & b \geq h, \\
	 6h^{\nicefrac{1}{3}}, & b < h. 
	 \end{cases}
\end{align*}
It remains to consider the last term in \eqref{eq:cusp}. Suppose $h < \min \{|a|, b\}$, then \begin{align}\label{eq:splitting}
	\int_{a}^{b-h} \!\bigl(\mathrm{sgn}(x+h)|x+h|^{\nicefrac{-1}{3}} - \mathrm{sgn}(x)|x|^{\nicefrac{-1}{3}} \bigr)^2dx &= \int_a^{-h} \left(|x|^{\nicefrac{-1}{3}} - |x+h|^{\nicefrac{-1}{3}}\right)^2dx \nonumber
	\\ & + \int_{-h}^0\! \left(|x+h|^{\nicefrac{-1}{3}} + |x|^{\nicefrac{-1}{3}} \right)^2dx
	\\ & + \int_{0}^{b-h}\! \left( |x+h|^{\nicefrac{-1}{3}} - |x|^{\nicefrac{-1}{3}}\right)^2dx. \nonumber
\end{align}
Let us estimate the first integral. By a direct calculation
\begin{align*}
	\int_a^{-h}\! \left(|x|^{\nicefrac{-1}{3}} - |x+h|^{\nicefrac{-1}{3}}\right)^2dx &= \int_{a}^{-h}|x|^{\nicefrac{-2}{3}}dx + \int_a^{-h} |x+h|^{\nicefrac{-2}{3}}dx
	\\ & \quad -2 \int_{a}^{-h}|x|^{\nicefrac{-1}{3}}|x+h|^{\nicefrac{-1}{3}}dx
	\\ &= 3\Big(|a|^{\nicefrac{1}{3}} - h^{\nicefrac{1}{3}} + ||a|-h|^{\nicefrac{1}{3}}  \Big)
	\\ & \quad -2 \int_{a}^{-h}|x|^{\nicefrac{-1}{3}}|x+h|^{\nicefrac{-1}{3}}dx.
\end{align*}
The last term is negative and since $|x+h|^{\nicefrac{-1}{3}} \geq |x|^{\nicefrac{-1}{3}}$ for $x \in [a, -h]$, we get
\begin{align}\label{eq:crucialBound}
	-2 \int_{a}^{-h}|x|^{\nicefrac{-1}{3}}|x+h|^{\nicefrac{-1}{3}}dx \leq -2\int_{-|a|}^{-h}|x|^{\nicefrac{-2}{3}} dx= 6\Big(h^{\nicefrac{1}{3}} - |a|^{\nicefrac{1}{3}}\Big).
\end{align}
\begin{figure}
	\includegraphics[scale=0.9]{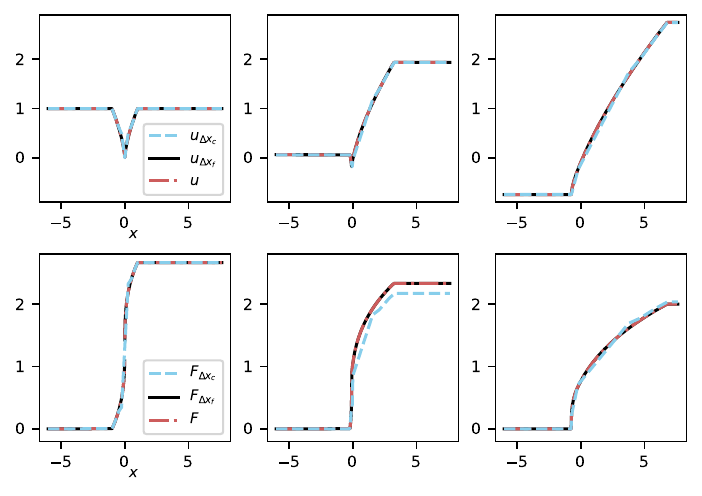}
	\captionsetup{width=.925\linewidth}
	\vspace{-2ex}
	\caption{A comparison of $u$ (top row) and $F$ (bottom row), both with dotted red lines, to that of $u_{\Dx_j}$ (top row) and $F_{\Dx_j}$ (bottom row) for $\Dx_c = \nicefrac{1}{4}$ (blue dashed) and $\Dx_f=10^{-3}$ (black solid) for Example~\ref{ex:Cusp} with $a=-1$ and $b=1$. The solutions are compared from left to right at $t=0$, $\frac{3}{2}$, and $3$, with $\alpha=\nicefrac{1}{2}$.}
	\label{fig:testCusp}
\end{figure}

Combining \eqref{eq:triangle} with the fact that $|\cdot|^q$ is increasing and \eqref{eq:crucialBound} therefore yields
\begin{align*}
	\int_a^{-h}\! \left(|x|^{\nicefrac{-1}{3}} - |x+h|^{\nicefrac{-1}{3}}\right)^2dx &\leq 3\Big(|a|^{\nicefrac{1}{3}} - h^{\nicefrac{1}{3}} + \underbrace{||a|-h|^{\nicefrac{1}{3}}}_{\leq |a|^{\nicefrac{1}{3}} + h^{\nicefrac{1}{3}}}  \Big) + 6\Big(h^{\nicefrac{1}{3}} - |a|^{\nicefrac{1}{3}}\Big) \vspace{-0.2cm}
	\\ &\leq 6h^{\nicefrac{1}{3}}. 
\end{align*}
One can argue similarly for the last integral in \eqref{eq:splitting}. Moreover, using $(c+d)^2 \leq 2(c^2 + d^2)$ with $c=|x+h|^{-\nicefrac{1}{3}}$ and $d=|x|^{-\nicefrac{1}{3}}$ for the second term in \eqref{eq:splitting} gives 
\begin{equation*}
	\int_{-h}^0\! \left(|x|^{\nicefrac{-1}{3}} + |x+h|^{\nicefrac{-1}{3}}\right)^2 \!dx \leq 2\int_{-h}^0\! \left( |x|^{\nicefrac{-2}{3}} + |x+h|^{\nicefrac{-2}{3}} \right)\!dx \leq 12h^{\nicefrac{1}{3}}, 
\end{equation*}
and we thus end up with 
\begin{equation*}
	\|\delta_h\bar{u}_x \|_2^2 \leq \frac{40}{3} |h|^{\nicefrac{1}{3}}, \hspace{0.5cm} \text{when } h < \min \{|a|, b \}. 
\end{equation*}
The cases $|a| \leq h \leq b$ and $b < h < |a|$ can be treated similarly, and, moreover, as one of the terms in \eqref{eq:splitting} are redundant in either case, we finally obtain
\begin{equation*}
	\|\delta_h\bar{u}_x \|_2 \leq  \frac{40}{3} h^{\nicefrac{1}{6}} \hspace{0.3cm} \text{for all } h \geq 0.
\end{equation*}
Recall \eqref{eq:semiNorm}, we have thus shown that $|\bar{u}_x|_{2, \nicefrac{1}{6}} \leq \nicefrac{40}{3}$ and that the convergence rate for this example depends on the parameters $a$ and $b$ in the following way 
\begin{align*}
	\|u(t) - u_{\Dx}(t) \|_{\infty} &\leq \begin{cases}
	\mathcal{O}(\Dx), & \text{if } a > 0, \\
	\mathcal{O}(\Dx^{\nicefrac{1}{16}}), & \text{if } b < 0, \\
	\mathcal{O}(\Dx^{\nicefrac{1}{48}}), & \text{if } a \leq 0 < b \text{ or } a < 0 \leq b, 
	\end{cases}
\end{align*}
whenever $\bar{\mu}_{\mathrm{sc}}=0$. 

The exact solution, when $a = -1$ and $b=1$, is compared with two numerical approximations in Figure~\ref{fig:testCusp}, where $\Dx_c$ and $\Dx_f$ are as before.  The solutions are compared at the times $t=0$, $\nicefrac{3}{2}$, and $3$. Moreover, Table~\ref{tab:EOCCusp} displays the associated relative errors and computed EOCs, which are significantly better than the predicted $\mathcal{O}(\Dx^{\nicefrac{1}{48}})$-rate. Yet, the computed EOCs in Table~\ref{tab:EOCCusp} are lower than the corresponding ones in Table~\ref{tab:EOCCosine}, except from those with $k\geq 6$ (where the aforementioned additional numerical errors influence the EOCs), which suggests that the regularity of $\bar{u}_x$ also influences the numerical convergence rate.

\begin{table}
 \setlength\tabcolsep{4.0pt}
\begin{tabular}{c|ccccccc} 
 \hline
 $k$  & $1$ & $2$ & $3$ & $4$ & $5$ & $6$ & $7$ \\ 
 \hline
  $\mathrm{Err}_k(3)$  & $0.10$ & $0.045$ & $0.019$ & $6.60 \cdot 10^{-3}$ & $2.22 \cdot 10^{-3}$ & $7.95 \cdot 10^{-4}$ & $2.80 \cdot 10^{-4}$ \\ 
  \hline
  $\mathrm{EOC}_k(3)$  & $-$ & $0.58$ & $0.63$ & $0.74$ & $0.79$ & $0.74$ & $0.75$
\end{tabular}
	\captionsetup{width=.925\linewidth}
	\vspace{-0.35cm}
\caption{Computed relative errors and EOCs with $T=3$ for Example~\ref{ex:Cusp} when $\alpha = \nicefrac{1}{2}$, that is, these quantities are computed in $[0,3]$, after the last wave breaking occurrence for the exact solution. } 
\label{tab:EOCCusp}
\end{table}
\end{example}

\section*{Acknowledgments}
The author gratefully acknowledges the hospitality of the Mittag--Leffler Institute, Sweden, for creating a great working environment for research, during the fall semester 2023. 

The author is deeply indebted to Susanne Solem and Katrin Grunert for their careful reading of the manuscript and many helpful discussions in the course of this work. 

\bibliographystyle{plain}

\appendix

\section{Exact solution to the two multipeakon examples}\label{app:multipeakon}
Given $\alpha \in [0, 1]$, consider the Cauchy problem with initial data
\begin{align*}
	\bar{u}(x) &= \begin{cases}
	\frac{1}{2}, & x < 0, \\
	-x + \frac{1}{2}, & 0 \leq x \leq \frac{1}{2}, \\
	0, & \frac{1}{2} < x, 
	\end{cases} \\
	\bar{F}(x) &= \begin{cases}
	0, &  x < 0, \\
	x, & 0 \leq x \leq \frac{1}{2}, \\
	\frac{1}{2}, & \frac{1}{2} < x. 
	\end{cases}
\end{align*}
Applying the mapping $L$ from \eqref{eq:MapL} results in the Lagrangian triplet $(\bar{y}, \bar{U}, \bar{V})$ given by 
\begin{align*}
	\bar{y}(\xi) &= \begin{cases}
	\xi, & \xi < 0, \\
	\frac{1}{2}\xi, & 0 \leq \xi \leq 1, \\
	\xi - \frac{1}{2}, & 1 < \xi, 
	\end{cases} \\
	\bar{U}(\xi) &= \begin{cases}
	\frac{1}{2}, & \xi < 0, \\
	\frac{1}{2}(-\xi + 1), & 0 \leq \xi \leq 1, \\
	0, & 1 < \xi
	\end{cases} \\
	\bar{V}(\xi) &= \begin{cases}
	0, & \xi < 0, \\
	\frac{1}{2} \xi, &0 \leq \xi \leq 1, \\
	\frac{1}{2}, & 1 < \xi,
	\end{cases}
\end{align*}	
which inserted into \eqref{eq:waveBreaking} gives
\begin{align}\label{eq:ExBreak}
	\tau(\xi) &= \begin{cases}
	2, & 0 \leq \xi \leq 1, \\
	\infty, & \text{otherwise}. 
	\end{cases}
\end{align}
In other words, no wave breaking occurs inside the time interval $[0, 2)$, and the solution to \eqref{eq:LagrSystem} with initial data $(\bar{y}, \bar{U}, \bar{V})$ therefore reads, for $t \in [0, 2)$, 
\begin{subequations}
\begin{align}\label{eq:exLagrBefore}
	y(t, \xi) &= \begin{cases}
		\xi - \frac{1}{16}t^2 + \frac{1}{2}t, & \xi < 0, \\
		\frac{1}{8}(t-2)^2\xi - \frac{1}{16}t^2 + \frac{1}{2}t, & 0 \leq \xi \leq 1, \\
		\xi + \frac{1}{16}t^2 - \frac{1}{2}, & 1 < \xi, 
		\end{cases} \\ 
	U(t, \xi) &= \begin{cases}
	-\frac{1}{8}t + \frac{1}{2}, & \xi < 0, \\
	\frac{1}{4}(t-2)\xi - \frac{1}{8}t  + \frac{1}{2},  & 0 \leq \xi \leq 1, \\
	\frac{1}{8}t, & 1 < \xi, 
	\end{cases} \\ 
	V(t, \xi) &= \bar{V}(\xi). 
\end{align}
\end{subequations}
Observe that $y(t, \cdot)$ is strictly increasing for each $t \in [0, 2)$, such that its inverse can be computed by solving $x=y(t, \xi)$ in terms of $\xi$. It therefore follows from \eqref{eq:MapM} that
\begin{align}\label{eq:eulBeforeBreaking}
	u(t, x) &= \begin{cases}
	-\frac{1}{8}t + \frac{1}{2}, & x < \frac{1}{16}(8-t)t, \\
	\frac{1}{4(t-2)}(8x - (t+4)), & \frac{1}{16}(8-t)t \leq x \leq \frac{1}{16}(t^2 + 8), \\
	\frac{1}{8}t, & \frac{1}{16}(t^2 + 8) < x, 
	\end{cases} \\
	F(t, x) &= \begin{cases}
	0, &  x < \frac{1}{16}(8-t)t, \\
	\frac{1}{4(t-2)^2}(16x + t^2 -8t), & \frac{1}{16}(8-t)t \leq x \leq \frac{1}{16}(t^2 + 8), \\
	\frac{1}{2}, & \frac{1}{16}(t^2 + 8) < x.  
	\end{cases}
\end{align}
Moreover, as some energy possibly is lost  at the wave breaking which takes place at time $t=2$ for all $\xi \in [0, 1]$, cf. \eqref{eq:ExBreak}, one has to rescale the Lagrangian energy. In particular, for $t \geq 2$, we have 
\begin{align*}
	V(t, \xi) = \begin{cases}
	0, & \xi < 0, \\
	\frac{1}{2}(1-\alpha)\xi, & 0 \leq \xi \leq 1, \\
	\frac{1}{2}(1-\alpha), & 1 < \xi. 
	\end{cases}
\end{align*}
and by solving \eqref{eq:LagrSystem} with data $\lim_{t \uparrow 2}(y, U)(t) = (y, U)(2)$ one obtains
\begin{align}\label{eq:exLagrAfter}
	y(t, \xi) &= \begin{cases}
	\xi - \frac{1}{16}(1-\alpha)t^2 + \frac{1}{4}(2-\alpha)t + \frac{1}{4}\alpha, & \xi < 0, \\
	\frac{1}{8}(1-\alpha)(t-2)^2\xi - \frac{1}{16}(1-\alpha)t^2 + \frac{1}{4}(2-\alpha)t + \frac{1}{4}\alpha, & 0 \leq \xi \leq 1, \\
	\xi + \frac{1}{16}(1-\alpha)t^2 + \frac{1}{4}\alpha t - \frac{1}{2} - \frac{1}{4}\alpha, & 1 < \xi, 
	\end{cases}  \nonumber \\
	U(t, \xi) &= \begin{cases}
	-\frac{1}{8}(1-\alpha)t + \frac{1}{4}(2-\alpha), & \xi < 0, \\
	\frac{1}{4}(1-\alpha)(t-2)\xi -\frac{1}{8}(1-\alpha)t + \frac{1}{4}(2-\alpha), & 0 \leq \xi \leq 1, \\
	\frac{1}{8}(1-\alpha)t + \frac{1}{4}\alpha, & 1 < \xi. 
	\end{cases}
\end{align}
No more  wave breaking occurs, such that \eqref{eq:exLagrBefore} together with \eqref{eq:exLagrAfter}  provide a global solution to \eqref{eq:LagrSystem}. Moreover, $y(t, \cdot)$ is strictly increasing for each $t >2$. Its inverse can therefore be computed, which due to \eqref{eq:MapM} implies
\begin{align}\label{eq:eulAfter}
	u(t, x) &= \begin{cases}
	-\frac{1}{8}(1-\alpha)t + \frac{1}{4}(2-\alpha), & x < x_1(t),  \\
	\frac{2}{t-2}\big(x - \frac{1}{8}(t+4) \big), & x_1(t) \leq x \leq x_2(t), \\
	\frac{1}{8}(1-\alpha)t + \frac{1}{4}\alpha, & x_2(t) < x, 
	\end{cases} \nonumber \\
	F(t, x) &= \begin{cases}
	0, & x < x_1(t), \\
	\frac{4}{(t-2)^2}\big(x + \frac{1}{16}(1-\alpha)t^2 - \frac{1}{4}(2-\alpha)t - \frac{1}{4}\alpha\big), & x_1(t) \leq x \leq x_2(t), \\
	\frac{1}{2}(1-\alpha), & x_2(t) < x, \end{cases}
\end{align}
where 
\begin{align*}
	x_1(t)\! :&= -\frac{1}{16}(1-\alpha)t^2 + \frac{1}{4}(2-\alpha)t + \frac{1}{4}\alpha, \\
	x_2(t)\!:&= \frac{1}{16}(1-\alpha)t^2 + \frac{1}{4}\alpha t + \frac{1}{4}(2-\alpha). 
\end{align*}
Combining \eqref{eq:eulBeforeBreaking} and \eqref{eq:eulAfter} yields the globally defined $\alpha$-dissipative solution $(u, F)(t, \cdot)$. 

The $\alpha$-dissipative solution with initial data \eqref{eq:worsePeakon} can be computed in precisely the same way. In particular, for $t < 2 $ it is given by 
\begin{align*}
	u(t, x) &= \begin{cases}
	-\frac{1}{12}t + \frac{1}{3}, & x < -\frac{1}{24}t^2 + \frac{1}{3}t, \\
	\frac{2}{t-2}\big(x-\frac{1}{12}t - \frac{1}{3} \big), &  -\frac{1}{24}t^2 + \frac{1}{3}t \leq x \leq \frac{1}{24}t^2 + \frac{1}{3}, \\
	\frac{1}{12}t, & \frac{1}{24}t^2 + \frac{1}{3} < x, 
	\end{cases} \\
	F(t, x) &= \begin{cases}
	0, & x < -\frac{1}{24}t^2 + \frac{1}{3}t, \\
	\frac{4}{(t-2)^2} \big(x + \frac{1}{24}t^2 - \frac{1}{3}t \big), &  -\frac{1}{24}t^2 + \frac{1}{3}t \leq x \leq \frac{1}{24}t^2 + \frac{1}{3}, \\
	\frac{1}{3}, & \frac{1}{24}t^2 + \frac{1}{3} < x, 
	\end{cases}
\end{align*}
while for $t \geq 2$ it reads
\begin{align*}
	u(t, x) &= \begin{cases}
	-\frac{1}{12}(1-\alpha)t + \frac{1}{6}(2-\alpha), & x < x_1(t), \\
	\frac{2}{t-2} \big(x - \frac{1}{12}t - \frac{1}{3} \big), & x_1(t) \leq x \leq x_2(t), \\
	\frac{1}{12}(1-\alpha)t + \frac{1}{6}\alpha, & x_2(t)<x, 
	\end{cases} \\
	F(t, x) &= \begin{cases}
	0, & x < x_1(t), \\
	\frac{4}{(t-2)^2} \big( x + \frac{1}{24}(1-\alpha)t^2 - \frac{1}{6}(2-\alpha)t -\frac{1}{6}\alpha \big), & x_1(t) \leq x \leq x_2(t), \\
	\frac{1}{3}(1-\alpha), & x_2(t) < x, 
	\end{cases}
\end{align*}
where
\begin{align*}
	x_1(t) \!:&= -\frac{1}{24}(1-\alpha)t^2 + \frac{1}{6}(2-\alpha)t + \frac{1}{6}\alpha, \\
	x_2(t) \!:&= \frac{1}{24}(1-\alpha)t^2 + \frac{1}{6}\alpha t + \frac{1}{6}(2-\alpha). 
\end{align*}
\end{document}